\documentclass[11pt,dvipsnames]{article}

\usepackage[main, final]{neurips_2026}

\usepackage[utf8]{inputenc}
\usepackage[T1]{fontenc}

\usepackage{microtype}
\usepackage{mathtools}
\usepackage{booktabs}
\usepackage{graphicx}
\usepackage{amsmath,amssymb,amsfonts,amsxtra}
\usepackage{amsthm}
\usepackage{xspace}
\usepackage{enumitem}
\usepackage{hyperref}
\usepackage{url}
\usepackage{mathrsfs}
\usepackage{bm}
\usepackage[capitalize,nameinlink]{cleveref}
\crefname{equation}{}{equations}
\crefname{chapter}{Appendix}{chapters}
\crefname{item}{}{items}
\crefname{enumi}{}{}
\usepackage{algorithm}
\usepackage{algorithmic}
\usepackage{comment}
\usepackage{tcolorbox}
\tcbuselibrary{skins,breakable}
\usepackage{makecell,multirow}       %
\usepackage{nicefrac}       %
\usepackage{thmtools}  
\usepackage{xspace}
\usepackage{footnote, tablefootnote}
\usepackage{float}
\floatstyle{plaintop}
\restylefloat{table}
\usepackage{epstopdf}
\usepackage{latexsym}
\usepackage{natbib}
\usepackage{stmaryrd}
\usepackage{graphicx}
\usepackage{pgfplots}
\usepackage{tikz}
\usepgfplotslibrary{groupplots}
\usepgfplotslibrary{fillbetween}
\pgfplotsset{compat=1.3}
\pgfplotsset{
    myplotstyle/.style={
    ylabel style={align=center, font=\bfseries\boldmath},
    xlabel style={align=center, font=\bfseries\boldmath},
    x tick label style={font=\bfseries\boldmath},
    y tick label style={font=\bfseries\boldmath},
    scaled ticks=false,
    every axis plot/.append style={thick},
    },
}
\usepackage{wrapfig}
\newcommand{\ip}[2]{\langle {#1},\, {#2} \rangle}

\newcommand{\norm}[1]{\|{#1}\|}

\newcommand{\bdc}{\textsc{Bdc}\xspace}

\newcommand{\sbdc}{\textsc{SBdc}\xspace}
\newcommand{\dc}{\textsc{Dc}\xspace}
\newcommand{\bdca}{\textsc{Bdca}\xspace}

\newtheorem{theorem}{Theorem}[section]

\newtheorem{prop}[theorem]{Proposition}
\newtheorem{lemma}[theorem]{Lemma}
\newtheorem{corr}[theorem]{Corollary}

\newtheorem{remark}[theorem]{Remark}
\newtheorem{example}[theorem]{Example}

\theoremstyle{definition}
\newtheorem{defn}{Definition}
\newtheorem{assumption}{Assumption}
\numberwithin{equation}{section}

\newcommand{\R}{\mathbb{R}}

\newcommand{\expect}[1]{\mathbb{E}\left[#1\right]}
\newcommand{\expectk}[1]{\mathbb{E}_{|k}\left[#1\right]}

\newcommand{\gap}{\mathrm{gap}}

\newcommand{\tbar}{\bar{\bm\theta}}
\newcommand{\ibar}[4]{\theta_{i{#1}}^{#2};\bar{\bm\theta}_{i{#3}}^{#4}}

\DeclareMathOperator*{\argmin}{argmin}

\usepackage[utf8]{inputenc} %
\usepackage[T1]{fontenc}    %
\usepackage{hyperref}       %

\definecolor{light-gray}{gray}{0.9}
\definecolor{darkblue}{rgb}{0.0,0.0,0.65}
\definecolor{darkred}{rgb}{0.4,0.0,0.0}
\definecolor{mygreen}{RGB}{0,100,0}

\hypersetup{
   colorlinks = true,
   citecolor  = darkblue,
   linkcolor  = darkred,
   filecolor  = darkblue,
   urlcolor   = darkblue,
 }
 
\usepackage{url}            %
\usepackage{booktabs}       %
\usepackage{amsfonts}       %
\usepackage{nicefrac}       %
\usepackage{microtype}      %
\usepackage{xcolor}         %

\title{The Multi-Block DC Function Class: \\ Theory, Algorithms, and Applications}

\author{%
  $^{1,2\,}$Pouria Fatemi\thanks{Authors contributed equally.} \qquad 
  $^{3,4,5\,}$Hoomaan Maskan\footnotemark[1]
  \qquad
  $^{3\,}$Alp Yurtsever \qquad
  $^{1,2\,}$Suvrit Sra \\[1em]
  $^{1}\,$Technical University of Munich, School of CIT, Dept. of Mathematics\\[0.1em]
  $^{2}\,$Munich Center for Machine Learning\\[0.1em]
  $^{3}\,$Umeå University, Dept. of Mathematics and Mathematical Statistics \\[0.1em]
  $^{4}\,$Uppsala University, Dept. of Information Technology\\[0.1em]
  $^{5}\,$Uppsala University, Science for Life Laboratory\\
}

\begin{document}

\maketitle

\begin{abstract}
  We present the Multi-Block DC (\bdc) class, a rich class of structured nonconvex functions that admit a \dc (``difference-of-convex'') decomposition across parameter blocks. This multi-block class not only subsumes the usual \dc programming, but also turns out to be provably more powerful.
 Specifically, we demonstrate how standard models (e.g., polynomials and tensor factorization) \emph{must have} \dc decompositions of exponential size, while their \bdc formulation is polynomial.
This separation in complexity also underscores another \emph{key aspect}: unlike \dc formulations, obtaining \bdc formulations for problems is vastly easier and constructive. We illustrate this aspect by presenting explicit \bdc formulations for modern tasks such as deep ReLU networks, a result with no known equivalent in the \dc class. Moreover, we complement the theory by developing algorithms with non-asymptotic convergence theory, including both batch and stochastic settings, and demonstrate the broad applicability of our method through several applications.
\end{abstract}

\section{Introduction}
Difference-of-convex (\dc) programming provides a powerful framework for nonconvex optimization, but existing formulations are difficult to adapt to the block-structured nature of modern machine learning. To this end, we introduce a multi-block difference-of-convex (\bdc) framework, in which the objective function admits a \dc decomposition with respect to each block when the remaining variables are fixed. The shift from single block to multiple blocks unravels a wealth of new considerations, of which key drivers for our work are complexity theory, algorithms, and empirical behavior. %

\dc formulations have a long history and have recently received renewed attention for structure-aware modeling in modern applications~\citep{khamaru2018convergence,davis2022gradient,maskan2025revisiting}. 
However, identifying practically useful \dc decompositions for a given problem is difficult. Although the existence of \dc decompositions can be guaranteed under mild smoothness assumptions \citep{tuy2016convex}, such  results are non-constructive and provide limited guidance. As shown by \citet{ahmadi2018dc}, identifying suitable \dc decompositions, even for polynomials of degree four, is strongly NP-hard.

These challenges call for a more flexible perspective.
In this work, we depart from seeking monolithic \dc decompositions and pursue \bdc formulations instead. These formulations strictly generalize classical \dc functions, meaning that there exist \bdc functions that do not admit any global \dc decomposition. Moreover, \bdc formulations elegantly capture the structure of a broad range of applications, including layered machine learning models and factorized representations. 
We show that
\bdc decompositions are relatively simpler to construct and thus they embody a versatile modeling class with strong theoretical foundations. Furthermore, even when a function is \dc, often its \dc representation can be exponentially more complex than its \bdc representation.

\noindent\textbf{Main Contributions.} We make the following main contributions in this work:
\begin{enumerate}[label=(\roman*), itemsep = 0mm, topsep = 0mm, leftmargin = 5mm]
\item We introduce \bdc, the class of multi-block difference-of-convex functions, which provides a principled framework for modeling structured nonconvex optimization problems. We establish fundamental properties of this class, most notably, we prove that for a specific atom dictionary, \bdc representations can be \textbf{\emph{exponentially smaller}} in complexity than \dc (\Cref{prop:dc_complexity}). Furthermore, we provide concrete tools for constructing \bdc\ formulations, highlighting \bdc's flexibility (\Cref{prop:bdc-Closure}).

\item We derive explicit \bdc representations for deep ReLU networks, showing how the network admits a layer-wise \dc structure (\Cref{thm:bdc-struct-Relu}). Using this result, we construct \bdc decompositions for standard supervised-learning objectives with mean-squared error and cross-entropy losses using a general composition principle (\Cref{thm:bdc-conjugate-compact}) that may be of independent interest.
\item We study multi-block variants of the \dc Algorithm tailored to \bdc functions, exploiting block structure for efficient updates, and establish new convergence guarantees under generalized smoothness and stochastic gradients assumptions. 
\end{enumerate}

\section{Problem Setup and Related Work}\label{sec:prob_set}
We consider the general optimization problem
\begin{align}\label{eqn:bdc_main}
\min_{\bm{\theta} \in \mathcal{X}} ~ f(\bm{\theta}),
\end{align}
where $f : \mathcal{X} \to \mathbb{R}$ is possibly nonconvex and 
$\mathcal{X} \subseteq \mathbb{R}^d$ is the non-empty open convex domain with $f^\star := \inf_{\bm\theta\in\mathcal{X}} f(\bm\theta) > - \infty$ .
We assume that $\mathcal{X}$ admits a Cartesian product decomposition $\mathcal{X} = \mathcal{X}_1 \times \dots \times \mathcal{X}_n$,  where each $\mathcal{X}_i \subseteq \mathbb{R}^{d_i}$
and $d = \sum_{i=1}^n d_i$. 
We also define $\bar{\mathcal{X}}_i =
\mathcal{X}_1 \times \dots \times \{0\}^{d_i} \times \dots \times \mathcal{X}_n$, the set obtained from $\mathcal{X}$ by forcing the $i$th block to be zero.
Let $ D_i \in \mathbb{R}^{d_i \times d}$ be the selection matrix that extracts 
the $i$th block of $\bm{\theta}$. 
Equivalently, $D_i$ is obtained by taking $d_i$ distinct rows of the identity matrix $ I_d$, 
so that $\{D_i\}_{i=1}^n$ forms a non-overlapping partition of the coordinates and satisfies $\sum_{i=1}^n D_i^\top D_i =  I_d$.
For clarity, we use boldface letters (e.g., $\bm{\theta} \in \mathcal{X}$) to denote full decision variables, 
and non-boldface symbols (e.g., $\theta_i \in \mathcal{X}_i$) to denote individual blocks. 
We define, for each $i \in [n]$,
$
 \theta_i := D_i \bm{\theta},
$
so that $\theta_i \in \mathbb{R}^{d_i}$ represents the $i$th block of $\bm{\theta}$.
We also define the block-extended vector
$
\bm{\theta}_i := D_i^\top D_i \bm{\theta},
$
which coincides with $\bm{\theta}$ on block $i$ and is zero elsewhere.
Its complement is 
$
\smash{\bar{\bm{\theta}}_i := ( I_d - D_i^\top D_i)\bm{\theta}},
$
so that $\bm{\theta} = \bm{\theta}_i + \bar{\bm{\theta}}_i$.
We are now ready to state our key structural assumption on $f$.

\begin{assumption}
\label{assump:multi-block-DC}
We assume that $f : \mathcal{X} \to \mathbb{R}$ admits a \dc decomposition 
with respect to each block $\theta_i$ when all other blocks are fixed. 
Formally, this means that for each $i \in [n]$, there exist functions $g_i, h_i : \mathcal{X}_i \times \bar{\mathcal{X}}_i \to \mathbb{R}$, such that for every $\bm{\theta} \in \mathcal{X}$,
\begin{equation}\label{eqn:bdc_func}
    f(\bm{\theta}) = g_i(\ibar{}{}{}{}) - h_i(\ibar{}{}{}{}),
\end{equation}
where $g_i(\cdot\,;\bar{\bm{\theta}}_i)$ and $h_i(\cdot\,;\bar{\bm{\theta}}_i)$ are convex in $\theta_i$. 
We refer to this property as a \bdc decomposition.
\end{assumption}
\subsection{BDC versus DC}
A reader familiar with \dc programming might wonder about the precise relation between \bdc functions and classical \dc functions.
Note that \bdc is a strict superset of \dc: there exist \bdc functions that do not admit a global \dc decomposition.
For instance, \citet{VESELY2018167} construct a function in $\R^2$ that is \dc on every convex curve but not globally \dc, and hence is \bdc.%

The more important distinction between the \bdc and \dc classes lies in their \emph{representation efficiency}.
In particular, \Cref{prop:dc_complexity} shows that for a fixed atom dictionary, \dc representations require exponentially many components even for simple monomials, whereas \bdc representations need only polynomially many.
Moreover, \bdc decompositions are typically much easier to construct in practice than global \dc decompositions as they can often be obtained constructively and therefore translate directly into algorithms.
This flexibility is valuable for complex models such as deep ReLU networks, where global \dc decompositions are hard to derive, but \bdc decompositions can be built systematically (\Cref{thm:bdc-struct-Relu}).

\begin{example}
\label{example:tensor}
Let $\mathcal{T} \in \mathbb{R}^{m_1 \times \cdots \times m_n}$ be an $n$th-order tensor, 
and let $\theta_i \in \mathbb{R}^{m_i \times r}$ denote the $i$th factor matrix. 
The canonical polyadic (CP) decomposition solves 
\[
\min_{\theta_1,\dots,\theta_n} ~ 
\frac{1}{2} \big\|\mathcal{T} - \llbracket \theta_1,\dots,\theta_n \rrbracket\big\|_F^2,
\]
where $\llbracket \theta_1,\dots,\theta_n \rrbracket$ denotes the rank-$r$ CP reconstruction.
This problem is nonconvex jointly in all $\theta_i$'s, but convex in each $\theta_i$ when the others are fixed. 
This gives a \bdc decomposition with $h_i = 0$. This observation underlies the classical alternating least-squares (ALS) algorithm, which performs exact multi-block minimization steps. 
Although this is a purely multi-block convex structure ($h_i = 0$), it illustrates how \bdc decompositions are easier to obtain than \dc decompositions. 
In particular, in \Cref{prop:dc_complexity} we show that the DC decomposition requires exponentially many terms for this task.
\end{example}

\subsection{Related Work}
\dc programming has been employed in a wide range of machine learning applications from kernel selection~\citep{argyriou2006dc} to discrepancy estimation for domain adaptation~\citep{awasthi2024best}. 
The classical method for solving \dc problems is the \dc Algorithm (DCA), introduced by~\citet{tao1986algorithms}. 
The first asymptotic convergence results for DCA were established by~\citet{tao1997convex}, with a simplified analysis under differentiability assumptions later provided by~\citet{lanckriet2009convergence}. More recently, non-asymptotic convergence rates of $\mathcal{O}(1/k)$ 
have been established~\citep{khamaru2018convergence,yurtsever2022cccp,abbaszadehpeivasti2023rate}. 
For a comprehensive survey, we refer to~\citep{le2018dc, le2024open}. 

Despite its generality, the \bdc function class remains largely unexplored. The only prior study we are aware of is \citep{pham2022alternating} that considers only the two-block case (termed partial \dc decomposition) and proposes the alternating \dc algorithm along with its convergence guarantee to weak and Fréchet/Clarke critical points under proper assumptions. 
However, they investigate neither the constructive structure (algebra) of \bdc functions nor their broader application potential, topics that we address through a general \bdc formulation and algorithms with convergence guarantees. 
Note that for a fixed block partition, BDC is more general than the class of differences of multiconvex functions \citep{doi:10.1137/120887795}. A representation \(f=W-Q\), with both \(W\) and \(Q\) multiconvex, requires a single global pair whose components are convex in every block separately whereas BDC permits a different pair \((g_i,h_i)\) for each block, requiring convexity only in the active block. 
Consequently, every multiconvex difference is \bdc, but the converse fails (see \citep[Section 4]{pham2022alternating} for a two-block example). 
Additionally, the \bdc class has a better expressability than the difference of multiconvex class.
The detail of this comparison is deferred to \Cref{app:BDCvsMulti}.

Finally, our framework should not be confused with the block-coordinate DCA of \citet{maskan2024block}, which tackles the simpler classical \dc problem with a fixed global decomposition and develops a block-coordinate algorithm. In contrast, we introduce and study the \bdc problem class which induces a broader and more flexible formulation. Our work, moreover, calls for a conceptual shift: rather than seeking a global \dc decomposition, we advocate a multi-block decomposition, as this is vastly easier to construct, more expressive, and often algorithmically advantageous.
\section{The BDC Function Class}
\label{sec:BDC-func-class}
In this section, we motivate and study the \bdc function class.
First, we prove that for a fixed atom class, the complexity of a \dc decomposition for monomials is exponentially higher than its \bdc counterpart. Then, we represent a \bdc formulation of MSE and CE loss functions.
Finally, we propose an explicit \bdc decomposition for deep ReLU networks using a similar approach to \citep{berrada2017trusting}.
Now, we present our closure proposition. The proof is given in \Cref{app:prf_bdc-Closure}. This result is used in later sections.
\begin{prop}
  \label{prop:bdc-Closure}
  Let $f_i$ be \bdc functions for $i = 1,\ldots,m$. Then, the following functions are also $\bdc$:
  $ \qquad$(i) $ \sum_{i= 1}^m \alpha_i f_i$, for $\alpha_i \in \mathbb{R},\qquad$
   (ii) $\max_{i = 1, \ldots m} f_i, \qquad$
 (iii) $\min_{i = 1, \ldots m} f_i$.
\end{prop}
\subsection{DC and BDC Complexity of a Monomial}
As discussed in \Cref{example:tensor}, here we show the exponential separation in monomial representation complexity between \bdc and \dc classes. 
Let $\bm \theta = (\theta_1,\cdots,\theta_n)$ and \(f(\bm\theta)=\theta_1^{b_1}\theta_2^{b_2}\cdots \theta_n^{b_n}\) with  \(s=\sum_{i=1}^n b_i\). 
We measure \dc and \bdc complexities by the minimum number of atoms needed to represent a decomposition. Take \dc function \(f(\bm \theta)=g(\bm \theta)-h(\bm \theta)\) with
$
g(\bm \theta)=\sum_{i=1}^{r}\alpha_i\,\phi_i(\bm \theta)$ and $
{h(\bm \theta)=\sum_{i=r+1}^{r+q}\alpha_i\,\phi_i(\bm \theta)},
$
where each $\alpha_i>0$ and \(\phi_i\) is a convex atom: \(\phi_i(\bm \theta)=(u_i^\top \bm \theta)^s\) if \(s\) is even, and \(\phi_i(\bm \theta)=(u_i^\top \bm \theta + \kappa_i)^{s+1}\) if \(s\) is odd. The choice of the simple power atoms was motivated by their link to the notion of \emph{Waring rank} \citep{carlini2012solution} and the \emph{polarization property} \citep{drapal2009symmetric} used in the proof of \Cref{prop:dc_complexity}. We denote by \(N\) the \emph{minimum atom count}, i.e., the minimum of \(r+q\) over all such decompositions. 
We bound $N$ in the following theorem. The detailed proof and definitions needed for this result are given in \Cref{app:proof_dc_complexity}.
\begin{theorem}
\label{prop:dc_complexity}
Consider \(f(\bm \theta)=\prod_{i=1}^n \theta_i^{b_i}\) with \(1\le b_1\le\cdots\le b_n\) and \(s=\sum_i b_i\). Then the minimum atom count \(N\)  for \dc decomposition is either of the following:
\begin{itemize}\setlength\itemsep{3pt}\setlength\parskip{0pt}\setlength\parsep{0pt}
\item If \(s\) is even and atoms are of the form \((u^\top \bm\theta)^s\), then
$
\prod_{i=2}^n (b_i+1)\le N \le \Big\lfloor \tfrac12 \prod_{i=1}^n (b_i+1)\Big\rfloor.
$
\item If \(s\) is odd and atoms are of the form \((u^\top \bm\theta+\kappa)^{s+1}\), then
$
N=\prod_{i=1}^n (b_i+1).
$
\end{itemize}
\end{theorem}
As \Cref{prop:dc_complexity} shows, the \dc decomposition of a monomial requires a large number of atoms.  
In contrast, a \bdc decomposition can be significantly more compact.  
In the simplest case, each $\theta_i^{b_i}$ is treated as a standalone block, reducing the atom count \textit{exponentially} compared to the \dc decomposition.  
More generally, one may split the monomial into a few larger blocks, decompose each block, and then multiply the resulting sums, thereby reducing the complexity.  
For instance, the monomial $\theta_1\theta_2\theta_3^2\theta_4^4$ requires at least $30$ atoms in a \dc representation, which matches the lower bound of \Cref{prop:dc_complexity}.
Instead taking the trivial blocks $ \theta_1 $, $ \theta_2 $, $ \theta_3^2 $, and $\theta_4^4$, yields a \bdc decomposition with only $4$ atoms.  
Alternatively, splitting into two blocks, $\theta_1\theta_2$ and $\theta_3^2\theta_4^4$, results in $2+7=9$ atoms in total through \Cref{prop:dc_complexity}.  
An explicit \bdc decomposition in this case is
\begin{align*}
\theta_1\theta_2\theta_3^2\theta_4^4
&=   \; \frac{1}{14400}\bigg[ \left(\theta_1+\theta_2\right)^2-\left(\theta_1-\theta_2\right)^2 \bigg] \times \bigg[5\left(\left(\theta_3+\theta_4\right)^6+\left(\theta_3-\theta_4\right)^6\right) \\
& \quad+  3\left(\left(\theta_3+3\theta_4\right)^6+\left(\theta_3-3\theta_4\right)^6\right) -8\left(\left(\theta_3+2\theta_4\right)^6+\left(\theta_3-2\theta_4\right)^6 + 420\theta_4^6\right)\bigg].
\end{align*}

This example illustrates the flexibility of block-dependent representations.

\subsection{BDC Formulation of MSE Loss~}\label{sub:MSE}
For a label $y \in \mathbb{R}$ and the \bdc function $f(\bm\theta)=g_i(\ibar{}{}{}{})-h_i(\ibar{}{}{}{})$, with nonnegative $g_i$ and $h_i$, the MSE loss is
$
\mathcal{L}_{  y}^{\mathrm{MSE}}(\bm\theta) := \big(f(\bm\theta) - y\big)^2.
$
This yields the explicit \bdc\ decomposition
\begin{equation}
\label{eq:bdc-loss-pos}
\mathcal{L}_{ y}^{\mathrm{MSE}}(\bm\theta)  =  2\big(g_i^2(\ibar{}{}{}{})+(h_i(\ibar{}{}{}{})+y)^2\big) - (g_i(\ibar{}{}{}{})+h_i(\ibar{}{}{}{})+y)^2,
\end{equation}
a difference of two multi-block convex functions if $y \ge 0$.

\begin{remark}
If labels $y$ are not guaranteed to be nonnegative, one can shift labels and outputs by a constant $c \ge 0$ so that $y+c \ge 0$. This translation does not affect the  \bdc structure, so the assumption $y\ge 0$ is not restrictive.
\end{remark}
\subsection{BDC Formulation of Cross Entropy Loss~}\label{sub:bdc_form}
Before we can establish a \bdc\ formulation of the CE loss, we need a general result that extends \bdc\ decompositions to more complex structures.  Specifically, we develop a composition principle ensuring that if the input admits a \bdc\ decomposition, then the expression obtained through a conjugate function can also be written explicitly in \bdc\ form. The following proposition establishes this principle (see \Cref{app:proof-bdc-conjugate-compact} for the proof).
In contrast to many \dc composition rules that only guarantee existence, our result is \emph{constructive}.

\begin{prop}
\label{thm:bdc-conjugate-compact}
Let $U \subset \mathbb{R}^m$ be non-empty compact and let $f:U\to\mathbb{R}$ be lower semicontinuous real-valued. Define the conjugate function
$
f^{*}(  t)=\max_{  u\in U}\{\langle   u,  t\rangle-f(  u)\}.
$
Suppose
$
  E(\bm\theta)=(E_1(\bm\theta),\ldots,E_m(\bm\theta)),
$
and that for each component $j=1,\dots,m$ and each block $i\in[n]$,
the mapping $\bm\theta \mapsto E_j(\bm\theta)$ admits a \bdc\ decomposition of the form
${
E_j(\bm\theta)=a_{ij}( \theta_i;\bar{\bm\theta}_i)-b_{ij}( \theta_i;\bar{\bm\theta}_i),}
$
where
$a_{ij}(\cdot;\bar{\bm\theta}_i)$ and $b_{ij}(\cdot;\bar{\bm\theta}_i)$
are convex in $\theta_i$. 
For $j\in[m]$ define
\[
\underline u_j:=\min_{  u\in U}u_j, 
\bar u_j:=\max_{  u\in U}u_j,
c_j^+:=\max\{-\underline u_j,0\},\text{and}\
d_j^+:=\max\{\bar u_j,0\}.
\]
Let $c^+, d^+ \in \mathbb{R}^m$ denote the vectors with components $c_j^+, d_j^+$.
Then $f^{*}\!\circ E$ is \bdc, with the explicit decomposition
$
f^{*}(  E(\bm\theta))=g_i( \theta_i;\bar{\bm\theta}_i)-h_i( \theta_i;\bar{\bm\theta}_i),
$
where 
\begin{align*}
h_i( \theta_i;\bar{\bm\theta}_i)
&:=\langle   c^+,  a_i( \theta_i;\bar{\bm\theta}_i)\rangle
 +\langle   d^+,  b_i( \theta_i;\bar{\bm\theta}_i)\rangle\quad\text{and}\quad
g_i( \theta_i;\bar{\bm\theta}_i)
:= f^{*}(  E(\bm\theta)) + h_i( \theta_i;\bar{\bm\theta}_i)
\end{align*}
with $
  a_i( \theta_i;\bar{\bm\theta}_i):=(a_{i1}( \theta_i;\bar{\bm\theta}_i),\ldots,a_{im}( \theta_i;\bar{\bm\theta}_i))$~and~$
  b_i( \theta_i;\bar{\bm\theta}_i):=(b_{i1}( \theta_i;\bar{\bm\theta}_i),\ldots,b_{im}( \theta_i;\bar{\bm\theta}_i)).
$
\end{prop}

The CE loss for a label $y\in\{1,\dots,C\}$ is
$
\mathcal L_{ y}^{\mathrm{CE}}(\bm\theta)
= \operatorname{LSE}\!\big(E(\bm\theta)\big) - E_{y}(\bm\theta),
$
where $\operatorname{LSE}(\cdot)$ is the log-sum-exp function with variational form
\[
\operatorname{LSE}(t) = \max_{  p\in\Delta_C} \big\{ \langle   p,  t\rangle - \operatorname{Ent}(  p) \big\},
\quad
\Delta_C:=\{  p\ge   0,\   1^\top   p=1\},
\]
and $\operatorname{Ent}(  p):=\textstyle\sum_{c=1}^C p_c\log p_c$. Since $\operatorname{LSE}$ is the conjugate function of $\operatorname{Ent}$, we utilize  \Cref{thm:bdc-conjugate-compact} to obtain a \bdc decomposition for the CE loss.
\begin{corr}
    Applying \Cref{thm:bdc-conjugate-compact} with $U=\Delta_C$ and $f=\operatorname{Ent}$ gives
$\underline u_j=0$ and $\bar u_j=1$ for all $j$, hence $  c^{+}=  0$ and $  d^{+}=  1$.
Therefore, for every $y$, $\mathcal L_{  y}^{\mathrm{CE}}(\bm\theta)
= g_i(\ibar{}{}{}{}) - h_i(\ibar{}{}{}{})$ where
\begin{align*}
h_i(\ibar{}{}{}{})
:= a_{iy}(\ibar{}{}{}{}) + 1^\top   b_i(\ibar{}{}{}{}), \quad
g_i(\ibar{}{}{}{})
&:= \operatorname{LSE}(E(\bm\theta)) +   1^\top   b_i(\ibar{}{}{}{}) + b_{iy}(\ibar{}{}{}{}).
\end{align*}
\end{corr}

\subsection{BDC Formulation of a Deep ReLU Network}\label{sec:deepreL_Net}
Consider an $L$-layer ReLU network parameterized by 
$
\bm\theta = \big(W_1, b_1, \dots, W_L, b_L\big).
$
For input $x\in\mathbb{R}^d$, define 
$a_0(x) = x,$ and for $l=1,\dots,L - 1$
\begin{align*}
F_{  x}(\bm\theta) &=   W_L a_{L-1}(  x) +  b_L, \quad
  a_l(  x) = \sigma\big(  W_l\,  a_{l-1}(  x)+  b_l\big),
\end{align*}
where $  W_l$ are weight matrices, $  b_l$ are bias vectors, and $\sigma(\cdot)$ denotes the ReLU activation. Here $  W_L\in\mathbb{R}^{C\times d_L}$ and $ b_L\in\mathbb{R}^C$ represent the weights of the output layer. For regression, we take $C=1$; for classification, $C$ is the number of classes. 

Now, we aim to express the network output as a \bdc function in each class. We begin by writing each activation using two nonnegative multi-block component-wise convex functions
$
  a_l =   Z_l^+ -   Z_l^-. 
$
Using $\sigma(a-b)=\max\{a,b\}-b$, we obtain
\(
  Z_{l+1}^+ -   Z_{l+1}^- 
= \sigma(  W_{l+1}  a_l+  b_{l+1}) \;=\;   a_{l+1}(  x),
\)
which shows that our construction is consistent with $a_l(x)$. Similar decomposition rule was adopted in \citep{berrada2017trusting}.
The network output is expressed as a \bdc function through the following initialization and forward recursion:

\begin{tcolorbox}[
  title=\textbf{ReLU BDC Decomposition},
  colback=gray!10,
  colframe=gray!50,
  coltitle=black,
  fonttitle=\bfseries,
  center title,
  boxrule=0.5pt,
  arc=3pt,
  left=6pt,
  right=6pt,
  top=6pt,
  bottom=6pt,
  breakable
]

\textbf{\emph{Initialization} ($l=1$).}
\(\quad Z_1^+ = \sigma(W_1 x + b_1),
\qquad
Z_1^- = 0.\)

\vspace{1mm}
\noindent
\textbf{\emph{Forward recursion ($l \to l{+}1$).}}
Given $(Z_l^+, Z_l^-)$, define
\[
\begin{aligned}
p_{l+1}
&= \sigma(W_{l+1})\, Z_l^+
 + \sigma(-W_{l+1})\, Z_l^-
 + b_{l+1}, \\[0.5ex]
Z_{l+1}^-
&= \sigma(W_{l+1})\, Z_l^-
 + \sigma(-W_{l+1})\, Z_l^+, \\
 Z_{l+1}^+
&= \max\bigl\{ p_{l+1},\, Z_{l+1}^- \bigr\}.
\end{aligned}
\]

\noindent\textbf{\emph{Output layer.}}
Define $F_x(\bm{\theta})
= A(\bm{\theta}) - B(\bm{\theta})$ where
{
\small
\begin{align}
\begin{aligned}
   &  A(\bm\theta) = \sigma(  W_L)   Z_{L-1}^+ + \sigma(-  W_L)  Z_{L-1}^- + \sigma(  b_L),\\
& B(\bm\theta) = \sigma(  W_L)   Z_{L-1}^- + \sigma(-  W_L)   Z_{L-1}^+ + \sigma(-  b_L).
\end{aligned}
\label{equ:relu-bdc-dec}
\end{align}
}
\end{tcolorbox}

This recursion guarantees $  Z_l^\pm\ge   0$ and that each component of $  Z_l^\pm$ is convex in the chosen block $\theta_l = (  W_l,  b_l)$ when all other parameters are fixed; the used operations (nonnegative linear maps and coordinatewise maxima) preserve convexity and nonnegativity layer by layer. 
The following \Cref{thm:bdc-struct-Relu}  proves that each component of $A(\bm\theta)$ and $B(\bm\theta)$ in \eqref{equ:relu-bdc-dec} is a convex function in every block (See \Cref{app:proof-bdc-struct-Relu} for the proof).

\begin{theorem}
\label{thm:bdc-struct-Relu}
For each block $\theta_l = (  W_l,  b_l)$, the construction in \eqref{equ:relu-bdc-dec} gives functions $ A(\bm\theta)$ and $ B(\bm\theta)$ such that each coordinate of $ A(\cdot\; ; \bar{\bm\theta}_l)$ and $B(\cdot\; ; \bar{\bm\theta}_l)$ is nonnegative and convex in $ \theta_l$, and we have $F_{  x}(\bm\theta)= A(\bm\theta)-B(\bm\theta)$.
\end{theorem}
Our result in \Cref{thm:bdc-struct-Relu} provides an explicit \bdc formulation for deep ReLU networks. Combining this with the results from \Cref{sub:MSE} and \Cref{sub:bdc_form}  provides a framework for training. While it is known (as an existence result) that deep ReLU networks are \dc, explicit \dc decompositions are currently available mainly for \emph{shallow} networks \citep{10779190,Tschernutter2024AGC}. 

\section{BDC Algorithm}
In this section, we pair the \bdc formulation with suitable algorithmic tools by analyzing two natural variants of the classical \dc algorithm and establish convergence guarantees under generalized smoothness conditions and in stochastic settings. Throughout, we assume exact solutions to the \bdca subproblems. All of our guarantees are extendable to approximate solution regimes under common error-summability assumptions (See \cite{pmlr-v97-xu19c}).

To define subgradients of $f(\bm \theta)$ as in \eqref{eqn:bdc_func} for continuously differentiable $g_i$ and locally Lipschitz $h_i$, we use the Clarke subdifferential, denoted by $\overline{\partial}$. 
Then, the partial Clarke subdifferential of $f$ satisfies 
\(
\overline{\partial}_if(\bm{\theta}) = \nabla g_i(\ibar{}{}{}{}) - \partial h_i(\ibar{}{}{}{}),
\)
for all $i\in\{1,\ldots,n\}$.
We measure convergence using a weak Clarke residual
\[   \mathcal{\bm G}^2(\bm\theta) :=\sum_{i=1}^n \mathrm{dist}(0,\overline{\partial}_i f (\bm\theta))^2.\] In particular, $ \mathcal{\bm G}^2(\bm\theta) = 0$ implies $\nabla g_i(\ibar{}{}{}{}) \in \partial h_i(\ibar{}{}{}{})$ for all $i=1,\ldots,n$.

\subsection{BDCA under L-smoothness Assumption}\label{sec:bdc-lsmooth}
Consider \bdc problem \eqref{eqn:bdc_main}. 
The \bdc algorithm at the $k^{\text{th}}$ iteration selects a block $i_k$ uniformly at random and updates this block by minimizing a surrogate function as follows:
\begin{align}\label{eqn:bdc_L_smooth}
                \theta^{k+1}_{i_k}  \in &\argmin_{{{\theta}_{i_k} } \in \mathcal{X}_{{i_k}}}  g_{i_k}(\ibar{_k}{}{_k}{k})-\ip{u_{i_k}^k}{\theta_{i_k}}.
\end{align}
where $u_{i_k}^k\in\partial_{i_k} h_{i_k}(\ibar{_k}{k}{_k}{k})$ denotes the convex subdifferential of $h_{i_k}$ with respect to the $i_k^{\textrm{th}}$ block. 
We set $\theta^{k+1}_{i_k}$ as above and keep all other blocks unchanged. 
The following proposition establishes the convergence rate of the block update \eqref{eqn:bdc_L_smooth} under regularity assumptions. 
\begin{assumption}\label{ass:smooth-max}
    For each $i\in [n]$, the function $g_i(\cdot;\bar{\bm\theta_i})$ is $L_i$-smooth on $\mathcal{X}_{i}$, with $L:=\max_{i\in[n]} L_i$.
\end{assumption}

\begin{assumption}\label{ass:Lipschitz_h}
    For each $i\in [n]$, the function $h_i$ is Lipschitz continuous with constant R on $\mathcal{X}$.
\end{assumption}

\begin{prop}
\label{corr:BDC_lsmooth_convergence}
    Under Assumptions~\ref{assump:multi-block-DC}, \ref{ass:smooth-max}, and \ref{ass:Lipschitz_h}, the sequence $\{\bm\theta^k\}_{k\ge1}$ generated by the update \eqref{eqn:bdc_L_smooth} with initialization $\bm\theta^0$ for $\mathcal{X}_i=\mathbb{R}^{d_i}$ and any $i\in[n]$ satisfies
    \begin{align}
        \min_{0 \leq k<K} \mathbb E \left[\mathcal{\bm G}^2(\bm\theta^k)\right]
        \leq \frac{2Ln}{K} \left(f(\bm{\theta}^0) - f^\star\right),%
    \end{align}
    where the expectation is taken with respect to the random block selections \(i_1, \dots, i_K\).
\end{prop}

In \Cref{app:BDCA_determined_smthness} we present an extension of this result to the case where the problem is constrained to a block-separable, closed and convex set.

\subsection{Proximal BDCA under Generalized Smoothness Assumption}\label{sec:BDCA_determined_gen_smthness}

Many optimization objectives  fail to satisfy a global Lipschitz smoothness assumption. Nevertheless, recent works have identified broader smoothness conditions that hold in important learning settings~\citep{zhang2019gradient,crawshaw2022robustness}. 
These conditions relax classical $L$-smoothness by allowing the Hessian norm to grow with the gradient magnitude rather than remain uniformly bounded. Motivated by these developments, and supported by empirical observations that a multi-block version of the generalized smoothness holds in neural network training (see \Cref{fig:gen_smoothness}), we adopt a generalized smoothness assumption on the components of $g(\bm\theta)$. 
We formalize this observation through the notions of $\ell$-smoothness and its first-order counterpart $(r,\ell)$-smoothness from \citep{li2024convex}.

\begin{defn}[\(\ell\)-smoothness] \label{def:lsmooth}
    Let $\ell:[0,+\infty)\to(0,+\infty)$ be continuous and non-decreasing. A twice differentiable function $\phi:\mathcal{D} \to \mathbb{R}$ is said to be $\ell$-smooth if it satisfies \(\|\nabla^2 \phi (\bm\theta)\|\leq \ell(\|\nabla \phi(\bm\theta)\|)\) at almost every point $\bm\theta \in \mathcal{D}$ with respect to the Lebesgue measure.
\end{defn}
\begin{wrapfigure}{r}{0.48\textwidth} %
  \centering

\begin{tikzpicture}

\begin{axis}[
    width=5cm, height=5cm,
    axis lines=left,
    grid=both,
    minor grid style={dotted,gray!35},
    major grid style={dashed,gray!55},
    xlabel={$ \log\!\bigl(\lVert\nabla g_i(\cdot\, ; \Bar{\bm\theta}_i)\rVert\bigr)$}, %
    ylabel={\footnotesize $\log(\text{estimated smoothness per block})$},
    colorbar,
    colorbar style={
    ticks=none,                 %
    ylabel={early $\to$ late (1500 total iterations)},  %
    ylabel style={font=\footnotesize},
  },
    legend style={at={(0.02,0.98)}, anchor=north west, draw=none, fill=none},
    legend cell align=left,
    scatter/use mapped color={draw=mapped color, fill=mapped color},
    mark options={line width=0.25pt, draw=black},
    clip marker paths=true,
]

\addplot[
    only marks, scatter, scatter src=explicit,
    mark=o, mark size=1.5pt, opacity=0.85,
]
table[
    col sep=comma,
    x=logG, y=logLhat, meta=t
] {Figs/blk_layer1_stride20_last10.csv};
\addlegendentry{$( W_1, b_1)$}

\addplot[
    only marks, scatter, scatter src=explicit,
    mark=o, mark size=2.6pt, very thick,
    forget plot,
]
table[
    col sep=comma,
    x=logG, y=logLhat, meta=t
] {Figs/blk_layer1_last10.csv};

\addplot[
    only marks, scatter, scatter src=explicit,
    mark=triangle*, mark size=1.9pt, opacity=0.85,
]
table[
    col sep=comma,
    x=logG, y=logLhat, meta=t
] {Figs/blk_layer2_stride20_last10.csv};
\addlegendentry{$( W_2, b_2)$}

\addplot[
    only marks, scatter, scatter src=explicit,
    mark=triangle*, mark size=3.0pt, very thick,
    forget plot,
]
table[
    col sep=comma,
    x=logG, y=logLhat, meta=t
] {Figs/blk_layer2_last10.csv};

\addplot[
    only marks, scatter, scatter src=explicit,
    mark=x, mark size=2.3pt, opacity=0.90,
]
table[
    col sep=comma,
    x=logG, y=logLhat, meta=t
] {Figs/blk_head_u_stride20_last10.csv};
\addlegendentry{$( W_3, b_3)$}

\addplot[
    only marks, scatter, scatter src=explicit,
    mark=x, mark size=3.4pt, ultra thick,
    forget plot,
]
table[
    col sep=comma,
    x=logG, y=logLhat, meta=t
] {Figs/blk_head_u_last10.csv};
\end{axis}

\end{tikzpicture}

  \vspace{-\intextsep}
  \caption{Estimated smoothness constant of $g_i(\cdot\, ; \Bar{\bm\theta}_i)$ in \eqref{eq:bdc-loss-pos} vs its gradient norm. For more details of the experiment see \Cref{app:simulation_gen_smooth}}
  \label{fig:gen_smoothness}
\end{wrapfigure}

\begin{defn}[\((r,\ell)\)-smoothness]\label{app_def:rlsmooth}
Let $r,\ell:[0,+\infty)\to(0,+\infty)$ be continuous functions such that $\ell$ is non-decreasing and $r$ is non-increasing. A differentiable function $\phi:\mathcal{D} \to \mathbb{R}$ is said to be $(r,\ell)$-smooth if for all $\bm\theta \in \mathcal{D},$ we have $\mathcal B (\bm\theta,\;\!r (\|\nabla \phi(\bm\theta)\|)) \subseteq \mathcal D$, and that \(
\|\nabla \phi(\bm\theta^1)-\nabla \phi(\bm\theta^2)\| \leq \ell \left(\|\nabla \phi(\bm\theta)\|\right)\,\|\bm\theta^1-\bm\theta^2\|
\) holds for all $\bm\theta^1$ and $\bm \theta^2$ in this ball.
\end{defn}

It is worth noting that these two notions are related. Specifically, every twice differentiable $(r,\ell)$-smooth function is $\ell$-smooth. Conversely, an $\ell$-smooth function that is closed on its open domain can be shown to satisfy $(r,m)$-smoothness for appropriately constructed functions $r$ and $m$ (see \Cref{app:discussion_l_smoothness}).

A typical convergence analysis under $(r,\ell)$-smoothness requires control of the update magnitude 
$\|\bm\theta^{k+1}-\bm\theta^{k}\|$, so that successive iterates remain within the local neighborhood where the smoothness bound is valid. 
In contrast with algorithms like Gradient Descent (GD), where this follows directly from the explicit step size and bounded gradient norms, such a bound does not arise automatically for \bdca in general. 
To enforce bounded iterate differences, we add and subtract $\frac{\rho}{2}\|\theta_{i_k} \! \|^2$ to \eqref{eqn:bdc_main} on each block by exploiting the non-uniqueness of the \dc decomposition. 
This leads to a proximal regularization ensuring bounded iterate differences. 
At $k^{\text{th}}$ iteration, the \emph{proximal \bdc algorithm} selects a block $i_k$ uniformly at random, and updates the selected block by solving the surrogate problem
\begin{align}\label{eqn:deterministic_multi-block_update}
    \theta^{k+1}_{i_k}  \in \argmin_{{{\bm \theta}_{i_k} } \in \mathcal{X}_{{i_k}}}  g_{i_k}(\ibar{_k}{}{_k}{k} ) -\ip{u_{i_k}^k}{\theta_{i_k}} + \frac{\rho}{2}\|\theta_{i_k}^k - \theta_{i_k}\|^2,
\end{align} 
where $u_{i_k}^k\in\partial_{i_k} h_{i_k}(\ibar{_k}{k}{_k}{k})$. 
All other blocks remain unchanged. 
The next theorem establishes a convergence rate for \eqref{eqn:deterministic_multi-block_update}. 

\begin{assumption}\label{ass:differentiability}
    For every $i\in[n]$, the function $g_i$ is differentiable and for any fixed $\bar{\bm{\theta}}_i$, $g_i(\cdot; \bar{\bm{\theta}}_i)$ is closed within its open domain $\mathcal{X}_i$. 
\end{assumption}

\begin{theorem}
\label{thm_bdc_gen_smooth_convergence}
    Suppose Assumptions~\ref{assump:multi-block-DC}, \ref{ass:Lipschitz_h}, and \ref{ass:differentiability} hold, and let $\{\bm\theta^k\}$ be the sequence generated by \eqref{eqn:deterministic_multi-block_update} from an initial point $\bm\theta^0\in\mathcal X$. 
    Assume that for any fixed $\bar{\bm{\theta}}_i$, the function $g_i(\cdot; \bar{\bm{\theta}}_i) \!:\! \mathcal{X}_i \to \mathbb{R}$ is $\ell$-smooth where $\ell$ is subquadratic. 
    Further assume that 
    \(
        h_{i_k}(\ibar{_k}{k}{_k}{k}) - h_{i_0}(\ibar{_0}{0}{_0}{0})
        \leq H
    \)
    for some constant $H \geq 0$. 
    Define \(       G := \max_j g_j(\theta_j^0;\bar{\bm\theta}_j^0) - g^* + H, \,
            E := \sup\{u>0 : u^2 \le 2\ell(2u) G \}, 
        \text{ and }
            L := \ell(2E).\)
    Choose $\rho \ge L\,\frac{2(E+R)}{E}$. Then the iterates satisfy
    \begin{align}
        \min_{0 \leq k < K} \mathbb E\!\left[\mathcal G^2(\bm\theta^k)\right]
        \leq
        \frac{2n(L+\rho)}{K} \bigl(f(\bm\theta^0)-f^\star\bigr).
    \end{align}
\end{theorem}
The proof and further discussion are deferred to \Cref{app:BDCA_determined_gen_smthness}. Compared to \citep{li2024convex}, the rate is scaled by a factor of $n$, reflecting the random block selection scheme \Cref{eqn:deterministic_multi-block_update}.

\subsection{Stochastic Proximal BDCA under Generalized Smoothness}\label{sec:BDCA_stochastic_gen_smthness}
In this section, we consider a stochastic variant of \eqref{eqn:bdc_main}. Assume that, for each block $i$, the functions $g_i$ and $h_i$ admit stochastic representations of the form $f(\bm \theta) 
    = g_i(\ibar{}{}{}{}) - h_i(\ibar{}{}{}{}) $ with
\begin{align}
    \label{eqn:stochastic-main-obj}
    g_i(\ibar{}{}{}{}):= \mathbb{E}_{s\sim \mathbb{P}} \!\left[ g_i(\ibar{}{}{}{},s)  \right],\quad h_i(\ibar{}{}{}{}) 
    := \mathbb{E}_{s\sim \mathbb{P}} \!\left[ h_i(\ibar{}{}{}{},s) \right],
\end{align}
where $(\Omega,\Sigma_\Omega,\mathbb P)$ is a probability space and, for each $s\in\Omega$, the functions 
$g_i(\cdot;\cdot,s)$ and $h_i(\cdot;\cdot,s)$ are convex functions on $\mathcal X_i$ w.r.t $\theta_i$ when $\bar{\bm\theta}_i$ is fixed.
At iteration $k$, we draw an i.i.d. sample $s^k \sim \mathbb{P}$ and select a block $i_k$ uniformly at random. 
We then form an unbiased estimator $\hat u_{i_k}^k$ of $u_{i_k}^k \in \partial_{i_k} h_{i_k}(\ibar{_k}{k}{_k}{k})$ using $s^k$. Replacing the deterministic update \eqref{eqn:deterministic_multi-block_update}, 
the stochastic update is given by
\begin{equation} \label{eqn:stochastic_multi-block_update}
            \begin{aligned}
                 \theta^{k+1}_{i_k}  \in \argmin_{{{\bm \theta}_{i_k} } \in \mathcal{X}_{{i_k}}} \; g_{i_k}(\ibar{_k}{}{_k}{k},s^k)
                 &-\ip{\hat{ u}_{i_k}^k}{{\theta}_{i_k}}  + \frac{\rho}{2}\|\theta_{i_k}^k-\theta_{i_k}\|^2 .
            \end{aligned}\tag{\sbdc}
        \end{equation}  
\sbdc was previously studied for the global \dc function class in \citep{nitanda2017stochastic,pmlr-v97-xu19c}. Here, we focus on the \bdc class. In this section, we make the following assumption. 
\begin{assumption}\label{ass:bounded_var}
Let $\{\mathcal F_k\}_{k\ge 0}$ be the natural filtration generated by the algorithm, and let
$s^k \sim \mathbb P$ denote the random sample used at iteration $k$ independent of $\mathcal F_k$.
For each block $i \in [n]$ and each iteration $k$, let
\(u_i^k \in \partial_i h_i(\theta_i^k,\bar{\bm\theta}_i^k).\)
Assume that the stochastic estimators $\hat u_i^k$ and $\nabla_i \hat g_i(\theta_i^k,\bar{\bm\theta}_i^k)$, computed using $s^k$, satisfy
\(\mathbb E\!\left[\hat u_i^k \mid \mathcal F_k\right] = u_i^k
\text{ and }
\mathbb E\!\left[\nabla_i \hat g_i(\theta_i^k,\bar{\bm\theta}_i^k)\mid \mathcal F_k\right]
=
\nabla_i g_i(\theta_i^k,\bar{\bm\theta}_i^k).\)
Furthermore, there exists $\sigma \ge 0$ such that, for all $k \ge 0$ and all $i=1,\ldots,n$,
\(
\mathbb E\!\left[
\left\|
\bigl(\nabla_i \hat g_i(\theta_i^k,\bar{\bm\theta}_i^k\bigr)
- \hat u_i^k \bigr) -
\bigl(\nabla_i g_i(\theta_i^k,\bar{\bm\theta}_i^k\bigr) -  u_i^k \bigr)
\right\|^2
\,\middle|\, \mathcal F_k
\right]
\le \sigma^2 .
\)
\end{assumption}    
\noindent The following theorem formulates the convergence of the stochastic \bdc (\sbdc) algorithm.

\begin{theorem}
[informal]\label{thm:sbdc_convergence}
Suppose Assumptions~\ref{assump:multi-block-DC}, \ref{ass:Lipschitz_h}, \ref{ass:differentiability}, and \ref{ass:bounded_var} hold. Further assume that, for every block $i$ and any fixed $\bar{\bm{\theta}}_i$, the function $g_i(\cdot; \bar{\bm{\theta}}_i) \!:\! \mathcal{X}_i \to \mathbb{R}$ is $\ell$-smooth with subquadratic $\ell$. 
Then, for any initialization $\bm\theta^0\in \mathcal X$,
target accuracy $\epsilon>0$, and $0<\delta<1$, the stochastic multi-block
DC algorithm \eqref{eqn:stochastic_multi-block_update}, with appropriately
chosen parameters $\rho=\Theta(\sqrt K)$ and
$\sigma^2=\mathcal O(K^{-1/2})$, requires
$K=\mathcal O(n^2/\epsilon^4)$ iterations to obtain, with probability
at least $1-\delta$ over both the samples and the random block choices,
\[
\min_{0 \leq k<K}\mathcal{\bm G}^2(\bm\theta^k)\leq\epsilon^2.
\]
\end{theorem}

The formal version of \Cref{thm:sbdc_convergence} as well as its proof and a detailed discussion are presented in \Cref{app:stochastic_gen_smooth_convergence}.
The condition $\sigma^2=\mathcal{O}(1/\sqrt{K})$ can be achieved through a comparable number of samples in the mini-batch or through variance reduction techniques. 
In particular, by \Cref{app:lem_batch_variance} (see \Cref{app:section_discussion_variance}) batch size should be $\Theta(n/\epsilon^2)$ and this means a sample complexity $\mathcal{O}(n^3/\epsilon^6)$. 
A similar assumption has appeared in previous works such as \citep{nitanda2017stochastic,yurtsever2019conditional}. 

\section{Applications}
We explain the versatility of the \bdc framework through illustrative applications and experiments. All numerical experiments were conducted on a Mac equipped with an Apple M5 Pro chip with 48 Gb of memory.

\noindent\textbf{Proximal Alternating Linearized Minimization.} 
The class of nonsmooth nonconvex optimization problems of the form
\begin{equation}
\min_{\theta_i\in\mathcal{X}_{d_i}} ~~ \sum_{i=1}^n f_i(\theta_i) + Q(\bm\theta) ,
\end{equation}
was studied by \citet{bolte2014proximal} under the assumption that $Q$ is $L$-smooth. 
Under suitable conditions, PALM can be cast in the form of \bdca. 
\citet{bolte2014proximal} established a non-asymptotic convergence guarantees for the PALM algorithm assuming the KL property.
We obtain convergence guarantees for this problem class within the \bdc framework under generalized smoothness assumptions, without requiring the KL property. 

\noindent\textbf{Multiplicative Multitask Feature Learning.}
MMFL aims to train a neural network that learns shared representations across multiple tasks. A shared vector \(\bm{c} \in \mathbb{R}^d\) modulates feature relevance across \(T\) tasks, and is multiplied by the weight vector \(\beta_t \in \mathbb{R}^d\), where \(d\) is the number of features. Sparse regularization excludes redundant features. For a detailed discussion, see \citep{wang2016multiplicative}.
The MMFL problem with sparsity imposed on $\bm c$ can be formulated as
\begin{align}\label{eqn:MMFL}
    \min_{\bm c\geq 0,\,\beta_t}~~\sum_{t=1}^T \textrm{loss}(\text{diag}(\bm c)\beta_t, X_t,y_t) + \lambda_1\sum_{t=1}^T\|\beta_t\|_{p}^p + \lambda_2 \|\bm c\|_0,
\end{align}
where $X_t \in \mathbb R^{n_t \times d}$, $\lambda_1,\lambda_2\geq0$ $y_t$ denote the data and labels for task $t$, and $p\geq 1$. 
To better approximate hard feature selection, nonconvex surrogates such as  $\|\bm c\|_1 - \|\bm c\|_Q$ (largest-\(Q\) norm) or the capped $\ell_1$ penalty ( $\sum_{j=1}^d\min\{|\bm c_j|,\gamma\}=\|\bm c\|_1 - \sum_{j=1}^d \max\{|c_j|-\gamma,0\}$ for $\gamma>0$) are preferred \citep{gong2012multi}. 
Replacing these penalties in \eqref{eqn:MMFL} results in a \bdc optimization task.

\noindent\textbf{Rank Regularization.} 
Consider an optimization problem of the following form:
\begin{align}
\label{app:low_rank}
    \min_{X, Y} ~ f(X, Y) + \lambda \operatorname{rank}(X)
\end{align}
where \( X \) and \( Y \) are two matrices in \( \mathbb{R}^{n \times m} \) and the function \( f(\cdot) \) is \bdc. 
This type of problem has several applications, such as 
matrix completion \citep{hazan2023partial}. 
Due to the rank term, \eqref{app:low_rank} is NP-hard and a convex surrogate known as the nuclear norm \( \|X\|_* = \sum_{i=1}^{\min\{n,m\}} \sigma_i \) is often utilized, where \( \sigma_i \) represents the \( i \)-th largest singular value.
A tighter non-convex approximation of the rank regularizer is the truncated nuclear norm (TNN), defined as \( \sum_{i=r+1}^{\min\{n,m\}} \sigma_i \). TNN can be rewritten as \( \|X\|_* - \sum_{i=1}^{r} \sigma_i \), which is a \dc function. 
Thus, replacing it in \eqref{app:low_rank} gives a \bdc due to the \dc regularizer. Note that when \( r = 1 \), the regularizer is equivalent to \( \|X\|_* - \|X\|_2 \), which is a special case commonly used as a non-convex regularizer for the rank term \citep{jiang2021proximal}.

\begin{figure*}[!t]
    \centering
    \begin{tikzpicture}

    \begin{groupplot}[group style={group size=4 by 1, horizontal sep=1cm},        
        width=0.26\textwidth, 
        height=0.27\textwidth,         
        grid=both, grid style={gray!30},        
        tick label style={font=\tiny},
        ]
    
    \nextgroupplot[
        xmin=0, xmax=700,        
        xlabel={Iteration},
        xlabel shift = -0.1cm,
        ymin=0.04, ymax=0.1,
        ylabel={Reconstruction Error},
        title = {Synthetic Data},
        line width=0.2pt,
        xtick align=outside,
        ytick align=outside,
        minor x tick num=1,
        minor y tick num=4,
        every axis/.append style={font=\tiny},
        yticklabel style={/pgf/number format/.cd, fixed, fixed zerofill, precision=2},
        ]

    \addplot[Peach, line width=0.7pt, name path=rec_synth_L1] table[x=iter ,y=rec_errors_L1, col sep=comma]{./Data/rec_errors_3_alpha_0_1.csv};\label{plot:rec_synthetic_L1}
    \addplot[Peach,  draw = none, name path=rec_synth_lower_L1] table[x=iter ,y=lower_bound_L1, col sep=comma]{./Data/rec_errors_3_alpha_0_1.csv};
    \addplot[Peach,   draw = none, name path=rec_synth_upper_L1] table[x=iter ,y=upper_bound_L1, col sep=comma]{./Data/rec_errors_3_alpha_0_1.csv};
    \addplot[Peach, opacity=0.3] fill between[of=rec_synth_lower_L1 and rec_synth_upper_L1];

    \addplot[RoyalBlue, line width=0.7pt, name path=rec_synth_LQ] table[x=iter ,y=rec_errors_LQ, col sep=comma]{./Data/rec_errors_3_alpha_0_1.csv};\label{plot:rec_synthetic_LQ}
    \addplot[RoyalBlue,  draw = none, name path=rec_synth_lower_LQ] table[x=iter ,y=lower_bound_LQ, col sep=comma]{./Data/rec_errors_3_alpha_0_1.csv};
    \addplot[RoyalBlue,   draw = none, name path=rec_synth_upper_LQ] table[x=iter ,y=upper_bound_LQ, col sep=comma]{./Data/rec_errors_3_alpha_0_1.csv};
    \addplot[RoyalBlue, opacity=0.3] fill between[of=rec_synth_lower_LQ and rec_synth_upper_LQ];

    \nextgroupplot[
        xmin=0, xmax=200,        
        xlabel={Iteration},
        xlabel shift = -0.1cm,
        ymin=0.4, ymax=0.48,
        ylabel={},
        title = {BSDS500},
        line width=0.7pt,
        xtick align=outside,
        ytick align=outside,
        minor x tick num=1,
        minor y tick num=4,
        every axis/.append style={font=\tiny},
        yticklabel style={/pgf/number format/.cd, fixed, fixed zerofill, precision=2},
        ]

    \addplot[Peach, line width=0.7pt, name path=rec_real_L1] table[x=iter ,y=rec_errors_L1, col sep=comma]{./Data/rec_errors_real_data_alpha_0_1.csv};\label{plot:rec_real_L1}
    \addplot[Peach,  draw = none, name path=rec_real_lower_L1] table[x=iter ,y=lower_bound_L1, col sep=comma]{./Data/rec_errors_real_data_alpha_0_1.csv};
    \addplot[Peach,   draw = none, name path=rec_real_upper_L1] table[x=iter ,y=upper_bound_L1, col sep=comma]{./Data/rec_errors_real_data_alpha_0_1.csv};
    \addplot[Peach, opacity=0.3] fill between[of=rec_real_lower_L1 and rec_real_upper_L1];

    \addplot[RoyalBlue, line width=0.7pt, name path=rec_real_LQ] table[x=iter ,y=rec_errors_LQ, col sep=comma]{./Data/rec_errors_real_data_alpha_0_1.csv};\label{plot:rec_real_LQ}
    \addplot[RoyalBlue,  draw = none, name path=rec_real_lower_LQ] table[x=iter ,y=lower_bound_LQ, col sep=comma]{./Data/rec_errors_real_data_alpha_0_1.csv};
    \addplot[RoyalBlue,   draw = none, name path=rec_real_upper_LQ] table[x=iter ,y=upper_bound_LQ, col sep=comma]{./Data/rec_errors_real_data_alpha_0_1.csv};
    \addplot[RoyalBlue, opacity=0.3] fill between[of=rec_real_lower_LQ and rec_real_upper_LQ];  

    \nextgroupplot[
        xmin=0, xmax=700,        
        xlabel={Iteration},
        ymin=0.79, ymax=0.85,
        xlabel shift = -0.1cm,
        ylabel={Sparsity},
        title = {Synthetic Data},
        xtick align=outside,
        ytick align=outside,
        ylabel shift = -0.1cm,
        line width=0.7pt,
        minor x tick num=1,
        minor y tick num=4,
        every axis/.append style={font=\tiny},
        yticklabel style={/pgf/number format/.cd, fixed, fixed zerofill, precision=2},
        shift={(0.7cm,0)},
        ]

    \addplot[Peach, line width=0.7pt, name path=sparsities_synth_L1] table[x=iter ,y=sparsities_L1, col sep=comma]{./Data/sparsities_3_alpha_0_1.csv};\label{plot:sparsities_synth_L1}
    \addplot[Peach,  draw = none, name path=sparsities_synth_lower_L1] table[x=iter ,y=lower_bound_L1, col sep=comma]{./Data/sparsities_3_alpha_0_1.csv};
    \addplot[Peach,   draw = none, name path=sparsities_synth_upper_L1] table[x=iter ,y=upper_bound_L1, col sep=comma]{./Data/sparsities_3_alpha_0_1.csv};
    \addplot[Peach, opacity=0.3] fill between[of=sparsities_synth_upper_L1 and sparsities_synth_lower_L1];

    \addplot[RoyalBlue, line width=0.7pt, name path=sparsities_synth_LQ] table[x=iter ,y=sparsities_LQ, col sep=comma]{./Data/sparsities_3_alpha_0_1.csv};\label{plot:sparsities_synth_LQ}
    \addplot[RoyalBlue,  draw = none, name path=sparsities_synth_lower_LQ] table[x=iter ,y=lower_bound_LQ, col sep=comma]{./Data/sparsities_3_alpha_0_1.csv};
    \addplot[RoyalBlue,   draw = none, name path=sparsities_synth_upper_LQ] table[x=iter ,y=upper_bound_LQ, col sep=comma]{./Data/sparsities_3_alpha_0_1.csv};
    \addplot[RoyalBlue, opacity=0.3] fill between[of=sparsities_synth_lower_LQ and sparsities_synth_upper_LQ]; 

\addplot[ForestGreen, line width=1pt, dashed] coordinates {(0, 0.84375) (700, 0.84375)};

    \nextgroupplot[
        xmin=0, xmax=200,        
        xlabel={Iteration},
        xlabel shift = -0.1cm,
        ymin=0.96, ymax=0.98,
        ylabel={},
        title = {BSDS500},
        line width=0.7pt,
        xtick align=outside,
        ytick align=outside,
        minor x tick num=1,
        minor y tick num=4,
        every axis/.append style={font=\tiny}, 
        name = SDL,
        shift={(0.9cm,0)},
        yticklabel style={/pgf/number format/.cd, fixed, fixed zerofill, precision=3},
        ]

    \addplot[Peach, line width=0.7pt, name path=sparsities_real_L1] table[x=iter ,y=sparsities_L1, col sep=comma]{./Data/sparsities_real_data_alpha_0_1.csv};\label{plot:sparsities_real_L1}
    \addplot[Peach,  draw = none, name path=sparsities_real_lower_L1] table[x=iter ,y=lower_bound_L1, col sep=comma]{./Data/sparsities_real_data_alpha_0_1.csv};
    \addplot[Peach,   draw = none, name path=sparsities_real_upper_L1] table[x=iter ,y=upper_bound_L1, col sep=comma]{./Data/sparsities_real_data_alpha_0_1.csv};
    \addplot[Peach, opacity=0.3] fill between[of=sparsities_real_lower_L1 and sparsities_real_upper_L1];

    \addplot[RoyalBlue, line width=0.7pt, name path=sparsities_real_LQ] table[x=iter ,y=sparsities_LQ, col sep=comma]{./Data/sparsities_real_data_alpha_0_1.csv};\label{plot:sparsities_real_LQ}
    \addplot[RoyalBlue,  draw = none, name path=sparsities_real_lower_LQ] table[x=iter ,y=lower_bound_LQ, col sep=comma]{./Data/sparsities_real_data_alpha_0_1.csv};
    \addplot[RoyalBlue,   draw = none, name path=sparsities_real_upper_LQ] table[x=iter ,y=upper_bound_LQ, col sep=comma]{./Data/sparsities_real_data_alpha_0_1.csv};
    \addplot[RoyalBlue, opacity=0.3] fill between[of=sparsities_real_lower_LQ and sparsities_real_upper_LQ];

    \end{groupplot}

    \node[anchor=south west, draw = black, line width=0.5pt, fill=white, font=\tiny]  (legend) at ([shift={(-0.8cm,0.1cm)}]SDL.south east) {
        \begin{tabular}{@{}l@{}l@{}}
            $\ell_1$ & \textcolor{Peach}{\rule[0.5ex]{1.5ex}{0.6ex}} \\
            $\ell_1 - \ell_Q$ & \textcolor{RoyalBlue}{\rule[0.5ex]{1.5ex}{0.6ex}} \\
            True Sparsity \ \ & \textcolor{ForestGreen}{\rule[0.5ex]{1.5ex}{0.6ex}}
        \end{tabular}
    };

    \end{tikzpicture}
   \caption{\small Reconstruction error and sparsity of codes for the $\ell_1$ regularizer (orange) and the nonconvex $\ell_1-\ell_Q$ regularizer (blue) on synthetic data and BSDS500 patches. Solid curves denote the mean over 10 runs, with shaded bands showing 95\% confidence interval. The dashed green line indicates the true sparsity level in the synthetic data. The nonconvex $\ell_1-\ell_Q$ formulation yields both lower reconstruction error and sparser codes.}

    \label{fig:SDL}
\end{figure*}

\noindent\textbf{Sparse Dictionary Learning.}
We illustrate the applicability of our theoretical framework on SDL problem. Given a data matrix 
\(Y = [y_1, \dots, y_n] \in \mathbb{R}^{m \times n}\), SDL seeks a dictionary 
\(\mathbf{D} = [d_1, \dots, d_l] \in \mathbb{R}^{m \times l}\) 
and sparse codes 
\(X = [x_1, \dots, x_n] \in \mathbb{R}^{l \times n}\) 
by solving
\begin{equation}
    \label{eq:sdl-l0}
    \min_{\mathbf{D} \in \mathcal{C},\,X \in \mathbb{R}^{l \times n}} ~
    \sum_{i=1}^n \frac{1}{2}\bigl\|y_i - \mathbf{D}\,x_i\bigr\|_2^2
    + \alpha \sum_{i=1}^n \|x_i\|_0,
\end{equation}
where $\mathcal{C} = \bigl\{\mathbf{D}\in\mathbb{R}^{m\times l} \mid \|d_j\|_2 \le 1\ \forall j\bigr\}$. Since the \(\ell_0\) penalty is NP‐hard to optimize, it is often replaced with \(\ell_1\)-norm.
More recently, nonconvex regularizers have been used to yield a tighter approximation to sparsity.
Following \citep{deng2020efficiency, maskan2024block}, we consider
\begin{equation}
    \label{eq:sdl-lq}
    \min_{\mathbf{D} \in \mathcal{C},\,X}
    \sum_{i=1}^n \frac{1}{2}\bigl\|{y}_i - \mathbf{D}\,{x}_i\bigr\|_2^2
    + \alpha \sum_{i=1}^n \bigl(\|{x}_i\|_1 - \|{x}_i\|_{Q}\bigr).
\end{equation}
Problem \eqref{eq:sdl-lq} is \bdc. Indeed, fixing either \(\mathbf{D}\) or \(X\) yields a \dc problem.
We conducted numerical simulations to solve the SDL problem with \(\ell_1\) and nonconvex \(\ell_1 - \ell_Q\) regularizers (Eq.~\ref{eq:sdl-lq}) via \bdca \eqref{eqn:bdc_L_smooth}. 
Performance is measured by reconstruction error \(\|Y - \mathbf{D}X\|_F^2\) and the proportion of zeros in \(X\).
We compare using synthetic data and Berkeley segmentation dataset \citep{MartinFTM01}. The results are shown in \Cref{fig:SDL}. For more detail and a comparison between GD and the \bdca \eqref{eqn:bdc_L_smooth}, see \Cref{app:simulation_sdl}.

\noindent\textbf{Application to Neural Networks.}
In \Cref{sec:deepreL_Net}, we obtained explicit BDC formulations of neural-network training objectives with MSE and cross-entropy losses. 
Using these formulations and~\eqref{eqn:stochastic_multi-block_update}, we train ReLU networks on CIFAR10 \citep{krizhevsky2009learning} and \textsc{FashionMNIST} \citep{xiao2017fashion}, and compare with SGD. We use a batch size of $4096$ and report the mean over 10 Monte-Carlo
runs. The motivation for this high batch size regime is due to the noise variance requirement $\sigma^2 = \mathcal{O}(K^{-1/2})$ in \Cref{thm:sbdc_convergence}. One way to ensure this condition is through higher batch sizes.
Figure~\ref{fig:NN} shows the training and test cross-entropy together
with test accuracy. On \textsc{FashionMNIST}, \Cref{eqn:stochastic_multi-block_update} and SGD attain comparable test performance. On CIFAR10, \Cref{eqn:stochastic_multi-block_update} attains lower test cross-entropy and higher test accuracy under the reported settings.
Further implementation details are provided in \Cref{app:simulation_nn}. 
Our experiments use a simple ReLU network and are meant as a proof of concepts.
Nevertheless, these \bdc-aware updates appear to induce distinct optimization dynamics, motivating further investigation beyond the scope of this work. 

\begin{figure*}[!t]
    \centering
    \begin{tikzpicture}
    \begin{groupplot}[
        group style={group size=2 by 1, horizontal sep=3cm},
        width=0.4\textwidth,
        height=0.35\textwidth,
        grid=both,
        grid style={gray!15},
        tick label style={font=\normalsize},
        every axis/.append style={font=\normalsize},
        xtick align=outside,
        ytick align=outside,
        xlabel={Epoch},
        xlabel shift=-0.1cm,
        ylabel shift=-0.2cm,
        yticklabel style={/pgf/number format/.cd,fixed}
    ]

    \nextgroupplot[
        xmin=0, xmax=100,
        ymin=0.03, ymax=2,
        ymode=log,
        ylabel={Cross Entropy},
        minor x tick num=1,
        minor y tick num=4,
        title={\textsc{FashionMNIST}}
    ]
    \addplot[RoyalBlue, line width=0.9pt, name path=fashion_train_loss_bdc]
        table[x=epoch,y=ce_mean,col sep=comma]
        {./Data/fashion_train_ce_acc_with_ci_rebuttal.csv};
        \label{train_loss_bdc_fmnist}
    \addplot[Peach, line width=0.9pt, name path=fashion_test_loss_bdc]
        table[x=epoch,y=ce_mean,col sep=comma]
        {./Data/fashion_test_ce_acc_with_ci_rebuttal.csv};
        \label{test_loss_bdc_fmnist}
    \addplot[RoyalBlue, dashed, line width=0.9pt, name path=fashion_train_loss_sgd]
        table[x=epoch,y=ce_mean,col sep=comma]
        {./Data/fashion_train_ce_acc_with_ci_sgd_rebuttal.csv};
        \label{train_loss_sgd_fmnist}
    \addplot[Peach, dashed, line width=0.9pt, name path=fashion_test_loss_sgd]
        table[x=epoch,y=ce_mean,col sep=comma]
        {./Data/fashion_test_ce_acc_with_ci_sgd_rebuttal.csv};
        \label{test_loss_sgd_fmnist}

    \nextgroupplot[
        xmin=0, xmax=50,
        ymin=1, ymax=3,
        ymode=log,
        ylabel={Cross Entropy},
        minor x tick num=1,
        minor y tick num=1,
        title={\textsc{CIFAR10}}
    ]
    \addplot[RoyalBlue, line width=0.9pt, name path=cifar_train_loss_bdc]
        table[x=epoch,y=ce_mean,col sep=comma]
        {./Data/cifar10_results_train_by_epoch_rebuttal.csv};
        \label{train_loss_bdc_cifar10}
    \addplot[Peach, line width=0.9pt, name path=cifar_test_loss_bdc]
        table[x=epoch,y=ce_mean,col sep=comma]
        {./Data/cifar10_results_test_by_epoch_rebuttal.csv};
        \label{test_loss_bdc_cifar10}
    \addplot[RoyalBlue, dashed, line width=0.9pt, name path=cifar_train_loss_sgd]
        table[x=epoch,y=ce_mean,col sep=comma]
        {./Data/cifar10_mc_results_sgd_train_by_epoch_pct.csv};
        \label{train_loss_sgd_cifar10}
    \addplot[Peach, dashed, line width=0.9pt, name path=cifar_test_loss_sgd]
        table[x=epoch,y=ce_mean,col sep=comma]
        {./Data/cifar10_mc_results_sgd_test_by_epoch_pct.csv};
        \label{test_loss_sgd_cifar10}
    \end{groupplot}

    \begin{axis}[
        at={(group c1r1.west)}, anchor=west,
        width=0.4\textwidth,
        height=0.35\textwidth,
        xmin=0, xmax=100,
        ymin=50, ymax=90,
        axis x line=none,
        axis y line*=right,
        axis background/.style={fill=none},
        ylabel={Test Accuracy ($\%$)},
        ylabel shift=-0.2cm,
        ylabel style={Plum},
        xmajorticks=false,
        ytick align=outside,
        tick label style={font=\normalsize},
        every axis/.append style={font=\normalsize},
        yticklabel style={/pgf/number format/.cd,fixed}
    ]
    \addplot[Plum, line width=0.9pt, name path=fashion_test_acc_bdc]
        table[x=epoch,y=acc_mean,col sep=comma]
        {./Data/fashion_test_ce_acc_with_ci_rebuttal.csv};
        \label{test_acc_bdc_fmnist}
    \addplot[draw=none, name path=fashion_test_acc_bdc_lower]
        table[x=epoch,y=acc_ci_lower,col sep=comma]
        {./Data/fashion_test_ce_acc_with_ci_rebuttal.csv};
    \addplot[draw=none, name path=fashion_test_acc_bdc_upper]
        table[x=epoch,y=acc_ci_upper,col sep=comma]
        {./Data/fashion_test_ce_acc_with_ci_rebuttal.csv};
    \addplot[Plum, opacity=0.1]
        fill between[of=fashion_test_acc_bdc_lower and fashion_test_acc_bdc_upper];

    \addplot[Plum, dashed, line width=0.9pt, name path=fashion_test_acc_sgd]
        table[x=epoch,y=acc_mean,col sep=comma]
        {./Data/fashion_test_ce_acc_with_ci_sgd_rebuttal.csv};
        \label{test_acc_sgd_fmnist}
    \addplot[draw=none, name path=fashion_test_acc_sgd_lower]
        table[x=epoch,y=acc_ci_lower,col sep=comma]
        {./Data/fashion_test_ce_acc_with_ci_sgd_rebuttal.csv};
    \addplot[draw=none, name path=fashion_test_acc_sgd_upper]
        table[x=epoch,y=acc_ci_upper,col sep=comma]
        {./Data/fashion_test_ce_acc_with_ci_sgd_rebuttal.csv};
    \addplot[Plum, opacity=0.1]
        fill between[of=fashion_test_acc_sgd_lower and fashion_test_acc_sgd_upper];
    \end{axis}

    \begin{axis}[
        at={(group c2r1.west)}, anchor=west,
        width=0.4\textwidth,
        height=0.35\textwidth,
        xmin=0, xmax=50,
        ymin=20, ymax=60,
        axis x line=none,
        axis y line*=right,
        axis background/.style={fill=none},
        ylabel={Test Accuracy ($\%$)},
        ylabel shift=-0.2cm,
        ylabel style={Plum},
        xmajorticks=false,
        ytick align=outside,
        tick label style={font=\normalsize},
        every axis/.append style={font=\normalsize},
        yticklabel style={/pgf/number format/.cd,fixed}
    ]
    \addplot[Plum, line width=0.9pt, name path=cifar_test_acc_bdc]
        table[x=epoch,y=acc_mean,col sep=comma]
        {./Data/cifar10_results_test_by_epoch_rebuttal.csv};
        \label{test_acc_bdc_cifar10}
    \addplot[draw=none, name path=cifar_test_acc_bdc_lower]
        table[x=epoch,y=acc_ci_lower,col sep=comma]
        {./Data/cifar10_results_test_by_epoch_rebuttal.csv};
    \addplot[draw=none, name path=cifar_test_acc_bdc_upper]
        table[x=epoch,y=acc_ci_upper,col sep=comma]
        {./Data/cifar10_results_test_by_epoch_rebuttal.csv};
    \addplot[Plum, opacity=0.1]
        fill between[of=cifar_test_acc_bdc_lower and cifar_test_acc_bdc_upper];

    \addplot[Plum, dashed, line width=0.9pt, name path=cifar_test_acc_sgd]
        table[x=epoch,y=acc_mean,col sep=comma]
        {./Data/cifar10_mc_results_sgd_test_by_epoch_pct.csv};
        \label{test_acc_sgd_cifar10}
    \addplot[draw=none, name path=cifar_test_acc_sgd_lower]
        table[x=epoch,y=acc_lower,col sep=comma]
        {./Data/cifar10_mc_results_sgd_test_by_epoch_pct.csv};
    \addplot[draw=none, name path=cifar_test_acc_sgd_upper]
        table[x=epoch,y=acc_upper,col sep=comma]
        {./Data/cifar10_mc_results_sgd_test_by_epoch_pct.csv};
    \addplot[Plum, opacity=0.1]
        fill between[of=cifar_test_acc_sgd_lower and cifar_test_acc_sgd_upper];
    \end{axis}

    \node[
        anchor=south,
        draw=black,
        line width=0.9pt,
        fill=white,
        font=\normalsize
    ] at ([xshift=-1.2cm,yshift=0.8cm]group c2r1.north west) {
        \begin{tabular}{r l r l r l}
        \sbdc\ train & \ref*{train_loss_bdc_fmnist} &
        \sbdc\ test & \ref*{test_loss_bdc_fmnist} &
        \sbdc\ accuracy & \ref*{test_acc_bdc_fmnist} \\
        SGD train & \ref*{train_loss_sgd_fmnist} &
        SGD test & \ref*{test_loss_sgd_fmnist} &
        SGD accuracy & \ref*{test_acc_sgd_fmnist}
        \end{tabular}
    };
    \end{tikzpicture}
    \caption{\small Comparison of \sbdc\ with SGD on FashionMNIST and CIFAR10
    using batch size $4096$ over 10 Monte-Carlo runs. Solid and dashed curves
    correspond to \sbdc\ and SGD, respectively. The shaded bands show two-sided
    $90\%$ confidence intervals for test accuracy.}
    \label{fig:NN}
\end{figure*}
    
\section{Conclusion and Limitations}
We introduced and motivated the \emph{multi-block} \dc class—strictly richer than classical \dc---and demonstrated its practicality from two angles: $(i)$ compared to \dc decompositions, \bdc formulations are far cheaper to construct (e.g., exponentially cheaper for monomials),
and $(ii)$ obtaining \bdc decompositions for modern problems (e.g., training deep ReLU networks) is vastly easier and constructive. Subsequently, after developing foundational properties of the \bdc class, we leveraged multi-block convexity to propose a Gauss–Seidel–type \bdc algorithm with non-asymptotic guarantees under $L$-smoothness, generalized smoothness, and stochasticity. Applications to MMFL, rank regularization, sparse dictionary learning, and neural network training illustrate the framework’s practicality and breadth.

We conclude by outlining future directions and current limitations. On the theoretical side, an important direction is to further investigate the representation-complexity gap between the \bdc\ and \dc\ classes (e.g., for ReLU networks). On the algorithmic side, although \Cref{eqn:deterministic_multi-block_update} guarantees monotone descent, our convergence analysis assumes boundedness of \(g_i(\ibar{}{}{}{})\) along the trajectory, which we ensure by bounding \(h_i(\ibar{}{}{}{})\) at the update points; removing this assumption would strengthen the theoretical guarantees. In addition, our analysis under the generalized-smoothness assumption applies only to unconstrained \bdc\ problems; extending it to constrained settings remains an open challenge.
Finally,  in the context of neural networks, we only demonstrated the modeling applicability of \bdc\ functions to ReLU architectures within a simple algorithmic framework. A more comprehensive empirical evaluation and comparison with state-of-the-art training methods are left for future work.

\section*{Acknowledgments}
Pouria Fatemi and Suvrit Sra acknowledge generous support from the Alexander von Humboldt Foundation. Hoomaan Maskan acknowledges support from the Science for Life Laboratory (SciLifeLab). Hoomaan Maskan and Alp Yurtsever acknowledge support from the Wallenberg AI, Autonomous Systems, and Software Program (WASP), funded by the Knut and Alice Wallenberg Foundation.
The authors would like to thank Manish Krishan Lal for thoughtful discussions. The authors declare no competing interests.

\setlength{\bibsep}{4pt}
\bibliographystyle{apalike}
\bibliography{ref}

\clearpage
\appendix

\section{Discussions}\label[appendix]{app:section_discussion}
In this section, we provide an example showing the expressability of \bdc against difference of multiconvex functions, more general results under smoothness assumption in \Cref{app:BDCA_determined_smthness}, background on generalized smoothness in \Cref{app:discussion_l_smoothness}, useful lemmas in stochastic gradient estimator's variance in \Cref{app:section_discussion_variance}, and more detail on numerical results in \Cref{app:simulation}. All the proofs are given in \Cref{app:proofs}.

\subsection{\bdc is More Flexible than Difference of Multiconvex}
\label[appendix]{app:BDCvsMulti}
Consider
\[
f(\theta)=\theta_1\theta_2\theta_3^2\theta_4^4,
\qquad \theta\in\mathbb R^4,
\]
with blocks \(u=(\theta_1,\theta_2)\) and \(v=(\theta_3,\theta_4)\). This example illustrates how BDC permits different decompositions for different blocks, even when a common multiconvex pair exists.

Define
\[
A_\pm(u)=\frac14(\theta_1\pm\theta_2)^2,
\qquad
P_t(v)=(\theta_3+t\theta_4)^6+(\theta_3-t\theta_4)^6,
\]
and
\[
B_+(v)=\frac{5P_1(v)+3P_3(v)}{3600},
\qquad
B_-(v)=\frac{8\bigl(P_2(v)+420\theta_4^6\bigr)}{3600}.
\]
Each component is nonnegative and convex in its respective block. Moreover,
\[
A_+(u)-A_-(u)=\theta_1\theta_2,
\qquad
B_+(v)-B_-(v)=\theta_3^2\theta_4^4.
\]

A common multiconvex decomposition is therefore
\[
\begin{aligned}
W(u,v)&=A_+(u)B_+(v)+A_-(u)B_-(v),\\
Q(u,v)&=A_+(u)B_-(v)+A_-(u)B_+(v).
\end{aligned}
\]
Indeed, \(f=W-Q\), and fixing either block makes each expression a nonnegative linear combination of convex functions of the other block. Thus, both \(W\) and \(Q\) are multiconvex.

BDC permits alternative decompositions tailored to the active block. For the \(u\)-block, write \(b(v)=\theta_3^2\theta_4^4\ge0\) and take
\[
g_1(u;v)=b(v)A_+(u),
\qquad
h_1(u;v)=b(v)A_-(u).
\]
These components are convex in \(u\) for every fixed \(v\), and \(f=g_1-h_1\).

For the \(v\)-block, let \(a(u)=\theta_1\theta_2\) and \(a_\pm(u)=\max\{\pm a(u),0\}\). Then
\[
\begin{aligned}
g_2(v;u)&=a_+(u)B_+(v)+a_-(u)B_-(v),\\
h_2(v;u)&=a_+(u)B_-(v)+a_-(u)B_+(v).
\end{aligned}
\]
For each fixed \(u\), the coefficients \(a_\pm(u)\) are nonnegative constants. Hence both components are convex in \(v\), and \(f=g_2-h_2\). Their dependence on the frozen block need not be convex.

The distinction can be seen directly in the first block. The common multiconvex pair satisfies
\[
W=g_1+R,\qquad Q=h_1+R,
\qquad
R(u,v)=\bigl(A_+(u)+A_-(u)\bigr)B_-(v).
\]
The block-specific decomposition removes this shared term while preserving convexity in \(u\). It does not preserve convexity in \(v\). Indeed, at \(u=(1,0)\), both \(g_1\) and \(h_1\) equal \(b(v)/4\), and \(b\) is nonconvex since
\[
b(1,0)=b(0,1)=0,
\qquad
b\!\left(\tfrac12,\tfrac12\right)=\tfrac1{64}>0.
\]
Thus, these components are admissible for the \(u\)-block under the BDC definition but cannot serve as a common multiconvex pair. This illustrates the freedom to simplify each decomposition using the frozen variables and the requirements of the active block. The example concerns flexibility of representation, otherwise the monomial itself belongs to both function classes.

\subsection{Multi-Block DCA Under Smoothness Assumption}\label[appendix]{app:BDCA_determined_smthness}
Here, we focus on a more general problem of the form:
\begin{align}\label{eqn:main_dc_moregeneral}
    \begin{aligned}
    &\min_{\bm\theta \in \mathcal{M}} ~~ f(\bm\theta) ,
    \end{aligned}
\end{align}
where for each  block $\bm\theta_i$
\begin{align}\label{eqn:general_func}
   f(\bm\theta):= g_i(\theta_i;\tbar_{i}) - h_i(\theta_i;\tbar_{i}),
\end{align}
 when $g_i(\cdot\,;\tbar_{i}) $ is an $L$-smooth function
 and we have constraint set $\mathcal{M}=\mathcal{M}_1\times \mathcal{M}_2\times \ldots \times \mathcal{M}_n$ and each $\mathcal{M}_i\subseteq\mathbb{R}^{d_i}$ is a closed convex set and $d=\sum_{i=1}^n d_i$. The rest of the setup is similar to the unconstrained setting in \Cref{sec:prob_set}. 
 Problem \eqref{eqn:main_dc_moregeneral} was addressed for DC objective function $f(\bm\theta)$ in \cite{maskan2024block}.
 Here, we show that our formulation is capable of solving such problem formulation under a multi-block DC assumption on the objective $f(\bm\theta)$. Under this assumption, we propose a multi-block DCA (\bdca) algorithm, shown in \Cref{alg:smooth_BDCA}. The following Theorem shows the convergence of this method.
The proof of this theorem is given in \Cref{app:prf_bdc_smooth}. 
\begin{theorem}\label{thm:bdc_smooth}
 Suppose assumptions \ref{assump:multi-block-DC}, \ref{ass:smooth-max}, and \ref{ass:Lipschitz_h}. Then, the sequence generated by the update \eqref{eqn:deterministic_multi-block_update_L_smooth} for solving problem \eqref{eqn:main_dc_moregeneral}, satisfies
    \begin{align}
        \min_{0\leq k < K} \mathbb E\left[\sum_{i=1}^n\gap^L_{\mathcal{M}_i}(\bm{\theta}^k)\right]
        \leq \frac{n}{K} \left(f(\bm{\theta}^0) - f^\star\right)~,%
    \end{align}
   where the expectation is with respect to the random choice of blocks and
     \begin{equation*}
        \gap^L_{\mathcal{M}_i}(\bm\theta) := \max_{\bm{x} \in \mathcal{M}_i}\min_{\bm z\in \overline{\partial}_{i}f (\bm\theta)}  \left\{ \ip{\bm z}{ \theta_i - \bm{x}} - \frac{L}{2} \norm{\bm{x} - \theta_i}^2 \right\},\quad 
    \end{equation*}
    is a gap measure ensuring convergence to blockwise critical points.
\end{theorem}
This result, is more general than the one presented in \Cref{corr:BDC_lsmooth_convergence}. 
Specifically, when $\mathcal{M}$ becomes $\mathbb{R}^d$
(no constraint), the sum of the gap measure over all blocks becomes $\frac{1}{2L}\mathcal{\bm G}^2(\bm y)$ which results in \Cref{corr:BDC_lsmooth_convergence}.
\begin{algorithm}[tb]
   \caption{\bdc Algorithm (L-smooth)}
   \label{alg:smooth_BDCA} 
\begin{algorithmic}
   \STATE {\bfseries Input:} set \(k=0\), initialization $\bm\theta^0$ and number of blocks $n$
   \STATE \textsc{Repeat:}
        \STATE Randomly choose $i_k$ in \([1,...,n]\) with uniform distribution
        \STATE Evaluate \( u^k_{i_k} \in \partial h_{i_k}(\ibar{_k}{k}{_k}{k} )\),
        \STATE Solve  \begin{align}\label{eqn:deterministic_multi-block_update_L_smooth}
            \begin{aligned}
                 \theta^{k+1}_{i_k}  \in &\argmin_{{{\theta}_{i_k} } \in \mathcal{M}^{{i_k}}}  g_{i_k}(\ibar{_k}{}{_k}{k} ) %
                 -\ip{u_{i_k}^k}{{\theta}_{i_k}}  
            \end{aligned}
        \end{align}  
    \STATE  Update $\bm \theta^{k+1} = \bar{\bm\theta}_{i_k}^k + \bm\theta_{i_k}^{k+1}, $
    \STATE Set \( k=k+1,\)
   \STATE \textsc{Until} Stopping criterion.
\end{algorithmic}
\end{algorithm}

\subsection{Analysis of Multi-Block Proximal DCA under Generalized Smoothness Assumption} \label[appendix]{app:BDCA_determined_gen_smthness}

In this section, we provide a more detailed discussion and prove the results in \Cref{sec:BDCA_determined_gen_smthness}. Recall that we assumed a more relaxed assumption on the component $g_i(\cdot\,;\tbar_{i}) $, known as $\ell$-smoothness.
A first order reminiscent of the $\ell$-smoothness is the \((r,\ell)\)-smoothness. 
To simplify presentation, here we use $\nabla g_i $ and $\partial h_i$ instead of $\nabla_i g_i$ and $\partial_ih_i$ which are the gradient and the subdifferential set of the $i^\text{th}$ decomposition with respect to the $i^\text{th}$ block.  We try to avoid this replacement whenever it raises confusion.
\begin{defn}[\((r,\ell)\)-smoothness,\citep{li2024convex}]\label{app_def:rlsmooth_i}
    A real-valued differentiable function \(g_i:\mathcal{X}_i\times\bar{\mathcal{X}}_i\rightarrow \mathbb{R}\) is \((r,\ell)\)-smooth for continuous functions \(r,\ell:[0,+\infty)\rightarrow (0,+\infty)\) where \(\ell\) is non-decreasing and \(r\) is non-increasing, 
    if we have \[\mathcal{B}(\theta_i,r(\|\nabla g_i(\ibar{}{}{}{})\|))\subseteq \mathcal{X}_i,\quad \forall \theta_i\in\mathcal{X}_i,\]
    and, for all \(\theta_i^1,\theta_i^2\in \mathcal{B}(\theta_i,r(\|\nabla g_i(\ibar{}{}{}{})\|))\) it holds that 
    \[\|\nabla g_i(\ibar{}{1}{}{})-\nabla g_i(\ibar{}{2}{}{})\|\leq \ell(\|\nabla g_i(\ibar{}{}{}{})\|)\|\theta^1_i-\theta_i^2\|.\]
\end{defn}
Due to $\|\nabla g_i(\ibar{}{}{}{})\|\leq \| \nabla g_j( \theta_j,\Bar{\bm\theta}_j)\|$ for $j:= \arg\max_k \| \nabla g_k(\theta_k,\Bar{\bm\theta}_k)\|$ and the fact that $r$ is a non-increasing function, we get
    $ \mathcal{B}(\bm\theta_i,r(\|\nabla g_j( \theta_j,\Bar{\bm\theta}_j)\|))\subseteq \mathcal{B}(\bm\theta_i,r(\|\nabla g_i(\ibar{}{}{}{})\|)). $
Therefore, for any $\theta_i^1,\theta_i^2 \in  \mathcal{B}(\theta_i,r(\|\nabla g_j( \theta_j,\Bar{\bm\theta}_j)\|))$ that satisfy \((r,\ell)\)-smoothness, we have:
 \begin{align*}
    \|\nabla g_i(\ibar{}{1}{}{})-\nabla g_i(\ibar{}{2}{}{})\| \leq \ell(\|\nabla g_j( \theta_j,\Bar{\bm\theta}_j)\|)\|\theta_i^1-\theta_i^2\|.
 \end{align*}

It is possible to relate these two definitions, i.e., we can show that an $\ell$-smooth function is $(r,\ell)$-smooth and vice-versa under specific choices for $r$ and $\ell$. 
This connection, investigated by \citet{li2024convex}, with more discussion and related results are given in \Cref{app:discussion_l_smoothness}.

A necessary condition for $(r,\ell)$-smoothness is that the iterates of our sequential algorithm have a bounded distance \(\|\bm\theta^{k+1} - \bm\theta^{k}\|\) so that they stay in a local ball. 
Usually, this is satisfied through bounded gradient norm condition and the sequential form of the algorithm. 
For example, in GD we have \(\|\bm \theta^{k+1}-\bm \theta^{k}\| = \|\eta \nabla f(\bm \theta^k)\|\). In DCA, such a connection does not have trivial validity. Using the non-uniqueness of the \dc decomposition, we add and subtract $\frac{\rho}{2}\|\theta_{i_k}\|^2$ to \eqref{eqn:bdc_main} on each block.
This gives the subproblems \eqref{eqn:deterministic_multi-block_update} after applying DCA, which are proximal-type updates.
The expected convergence rate of \Cref{eqn:deterministic_multi-block_update} is finalized in the following proposition:

\begin{prop}
    Consider Assumptions \ref{ass:Lipschitz_h} and \ref{ass:differentiability} when $\bm\theta^k$ is the output of \Cref{eqn:deterministic_multi-block_update} for any initialization $\bm\theta^0\in\mathcal{X}$. 
    Then, for any $\ell$-smooth $g_i$ with subquadratic $\ell$, if 
    \begin{align*}
        &h_{i_k}(\ibar{_k}{k}{_k}{k}) - h_{i_0}(\ibar{_0}{0}{_0}{0}) \leq H\\
        &E:=\sup\{u>0:u^2\leq 2\ell(2u).G\}<\infty\\
        &G:=\max_j g_j(\theta_j^0;\bar{\bm\theta}_{j}^0) - g^* + H
    \end{align*}
    for a constant $H\geq 0$, 
    and $L:=\ell(2E)$, then the sequence $\bm\theta^k$ generated by \Cref{eqn:deterministic_multi-block_update} with $\rho\geq L\frac{2(E+R)}{E}$ will satisfy
    \begin{align}
        \min_{0\leq k < K} \mathbb E\left[\mathcal{\bm G}^2(\bm\theta^k)\right]
        \leq \frac{2n(L+\rho)}{K} \left(f(\bm{\theta}^0) - f^\star\right)~,%
    \end{align}
    where the expectation is with respect to the random choice of blocks.
\end{prop}
\begin{proof}
    We begin by bounding the updates through the following lemma.

\begin{lemma}\label{app_lem:bdc_deter_bounded}
    For any starting point $\bm \theta^k$ the update generated by \eqref{eqn:deterministic_multi-block_update} is in $\mathcal{B}\left(\bm \theta^k,\frac{2}{\rho}\|\nabla g_{i_k}(\ibar{_k}{k}{_k}{k}) - u_{i_k}^k\|\right)$.
\end{lemma}
See \Cref{app:prf_bdc_deter_bounded} for the proof.
This result guarantees $\|\bm\theta^{k+1} - \bm\theta^{k}\|\leq \frac{2}{\rho}\|\nabla g_{i_k}(\ibar{_k}{k}{_k}{k}) - u_{i_k}^k\|.$
Due to $\|\nabla g_{i_k}(\ibar{_k}{k}{_k}{k}) - u_{i_k}^k\|\leq \|\nabla g_{i_k}(\ibar{_k}{k}{_k}{k})\|+\|u_{i_k}^k\|$ for any $u_{i_k}^k \in \partial h_{i_k}(\ibar{_k}{k}{_k}{k})$ and $R$-Lipschitz $h_{i_k}$ (see \Cref{ass:Lipschitz_h}), we need to bound $\|\nabla g_{i_k}(\ibar{_k}{k}{_k}{k})\|$ in order to have a bounded $\|\nabla g_{i_k}(\ibar{_k}{k}{_k}{k}) - u_{i_k}^k\|$.
When $g_{i_k}$ is $\ell$-smooth with bounded $g_{i_k}(\ibar{_{k}}{}{_{k}}{k})-g^*$ for some $\theta_{i_k}\in\mathcal{X}_{i_k}$, we get $\|\nabla g_{i_k}(\ibar{_{k}}{}{_{k}}{k})\|\leq E$ for $E>0$ (see \Cref{corr:lsmoth-bounded-grad} in \Cref{app:discussion_l_smoothness}). 
The following Lemma bounds $g_{i_k}(\ibar{_k}{k}{_k}{k})-g^*$ and proposes a choice for $\rho$ such that we have local bound on the gradients.
\begin{lemma}\label{app_lem:bound_on_g}
    Consider Assumptions \ref{ass:differentiability} and \ref{ass:Lipschitz_h} when $\bm\theta^{k+1}$ is the output of \Cref{eqn:deterministic_multi-block_update} for any initialization $\bm\theta^0\in\mathcal{X}$. Then, if for any randomly selected $i_0$, we have $ h_{i_k}(\ibar{_k}{k}{_k}{k}) - h_{i_0}(\ibar{_0}{0}{_0}{0}) \leq H$ for a constant $H\geq 0$, then $$g_{i_k}(\ibar{_k}{k}{_k}{k})-g^* \leq g_{i_0}(\ibar{_0}{0}{_0}{0}) - g^* + H.$$ Additionally, for any $\ell$-smooth $g_i$ with subquadratic $\ell$, if $\rho \geq \ell(2E)\frac{2(E+R)}{E}$, then for any $i\in[n]$ and $\bm\theta^1,\bm\theta^2\in\mathcal{B}(\bm\theta^{k},\nicefrac{2(E+R)}{\rho})$ we have:
    \begin{align}
     \begin{aligned}\label{app_eqn:lem_smoothness_descent}
         &\|\nabla g_i(\ibar{}{2}{}{2})-\nabla g_i(\ibar{}{1}{}{1})\|\leq L\|\theta^1_i-\theta_i^2\|, \\
         &g_i(\ibar{}{2}{}{2})\leq g_i(\ibar{}{1}{}{1}) + \langle \nabla g_i(\ibar{}{1}{}{1}),\theta_i^2-\theta_i^1 \rangle + \frac{L}{2}\|\theta_i^1-\theta_i^2\|^2,
     \end{aligned}
    \end{align}
    where \(L=\ell(2E)\) is the effective smoothness for some $E>0$.
\end{lemma}
See \Cref{app:prf_lem_bound_on_g} for the proof.
Note that $\bm\theta^1,\bm\theta^2$ in \Cref{app_lem:bound_on_g} differ only in their $i_k^{\text{th}}$ block selected on iteration $k$ of \Cref{eqn:deterministic_multi-block_update}.
Now building on the previous lemmas, we propose our main convergence result for \Cref{eqn:deterministic_multi-block_update}. The proof of this result is given in \Cref{app:prf_thm_bdc_gen_smooth_convergence} 
\begin{prop}\label{app_prop:bdc_gen_smooth_convergence}
    Assume the conditions in \Cref{app_lem:bound_on_g} and take 
    \begin{align*}
        E:&=\sup\{u>0:u^2\leq 2\ell(2u).G\}<\infty\\
        G:&=\max_{i}g_i(\ibar{}{0}{}{0}) - g^* + H,
    \end{align*}
    and $L:=\ell(2E)$. 
    Then, the sequence $\bm\theta^k$ generated by \Cref{eqn:deterministic_multi-block_update} with $\rho\geq L\frac{2(E+R)}{E}$ will satisfy
    \begin{align}
        \min_{k\in\{1,\ldots,K\}} \mathbb E_{i}\left[\mathcal{\bm G}^2(\bm\theta^k)\right]
        \leq \frac{2n(L+\rho)}{K} \left(f(\bm{\theta}^1) - f^\star\right)~,%
    \end{align}
    where the expectation is with respect to the random choice of blocks.
\end{prop}
\end{proof}

\subsection{Detailed Analysis of Stochastic Multi-Block Proximal DCA under Generalized Smoothness Assumption} \label[appendix]{app:stochastic_gen_smooth_convergence}
        
In order to show the convergence of \Cref{eqn:stochastic_multi-block_update}, we start by ensuring the boundedness of the updates as in the following lemma. See \Cref{app:prf_bounding_differences_stock} for the proof. To simplify presentation, here we use $\nabla g_i $ and $\partial h_i$ instead of $\nabla_i g_i$ and $\partial_ih_i$ which are the gradient and the subdifferential set of the $i^\text{th}$ decomposition with respect to the $i^\text{th}$ block.
\begin{lemma}\label{lem:bounding_differences_stoc}
    Denote the sequence generated by \Cref{eqn:stochastic_multi-block_update} as $\bm\theta^k$. Then, for any \( u^k_{i_k} \in \partial h_{i_k}(\ibar{_k}{k}{_k}{k} )\), if $\nabla \hat g_{i_k}(\ibar{_k}{k}{_k}{k}),\hat{ u}^k_{i_k}$ are the respective stochastic approximations of \(\nabla  g_{i_k}(\ibar{_k}{k}{_k}{k})\)  $u^k_{i_k}$ and , we have:
    \begin{align*}
        \|{\bm \theta}^{k+1} - {\bm \theta}^k\| &\leq \frac{2}{\rho} \left( \|\nabla g_{i_k}(\ibar{_k}{k}{_k}{k}) -u_{i_k}^k\|+\|\nabla\hat g_{i_k}(\bm\theta^k) - \hat{ u}^k_{i_k} - (\nabla g_{i_k}(\ibar{_k}{k}{_k}{k}) -  u^k_{i_k} )\|\right).
    \end{align*}
\end{lemma}

Note that the bound in \Cref{lem:bounding_differences_stoc} does not immediately imply that the solutions to the subproblems \Cref{eqn:stochastic_multi-block_update} will fall inside a ball. For this, we take $\|\nabla\hat g_{i_k}(\bm\theta^k) - \hat{ u}^k_{i_k} - (\nabla g_{i_k}(\ibar{_k}{k}{_k}{k}) -  u^k_{i_k} )\|\leq F'$ for some $F'>0$. Later, we find the value of $F'$ such that the bound $\|\nabla\hat g_{i_k}(\bm\theta^k) - \hat{ u}^k_{i_k} - (\nabla g_{i_k}(\ibar{_k}{k}{_k}{k}) -  u^k_{i_k} )\|\leq F'$ holds with high probability. Then, a similar result to \Cref{app_lem:bound_on_g} holds in the stochastic setting.
\begin{lemma}\label{lem:bound_on_g_stochastic}
    Consider Assumptions \ref{ass:differentiability} and \ref{ass:Lipschitz_h}
    when $\bm\theta^k$ is the output of \Cref{eqn:stochastic_multi-block_update} for any initialization $\bm\theta^0\in\mathcal{X}$. Then, for any $\ell$-smooth $g_i$ with subquadratic $\ell$ if $g_{i_k}(\ibar{_k}{k}{_k}{k})-g^* \leq G$ and $\|\nabla \hat g_{i_k}(\ibar{_k}{k}{_k}{k}) - \hat{ u}^k_{i_k} - (\nabla g_{i_k}(\ibar{_k}{k}{_k}{k}) -  u^k_{i_k} )\|\leq F'$ for $G,F'>0$ and $\rho\geq L\frac{2(E+R+F')}{E}$ for $L:=\ell(2E)$, we have: 
\begin{align}\label{eqn:lem_smoothness_descent_stochastic}
    \begin{aligned}
         &\|\nabla g_i(\ibar{}{2}{}{2})-\nabla g_i(\ibar{}{1}{}{1})\|\leq L\|\theta^1_i-\theta_i^2\|, \\
         &g_i(\ibar{}{2}{}{2})\leq g_i(\ibar{}{1}{}{1}) + \langle \nabla g_i(\ibar{}{1}{}{1}),\theta_i^2-\theta_i^1 \rangle + \frac{L}{2}\|\theta_i^1-\theta_i^2\|^2,
     \end{aligned}
    \end{align}
    for any $\bm\theta^1,\bm\theta^2\in\mathcal{B}(\bm\theta^k,\nicefrac{2(E+R+F')}{\rho})$.
\end{lemma}
See \Cref{app:prf_lem_bound_on_g_stochastic} for the proof.
Note that if $F'=0$, we get $\rho\geq \nicefrac{2L(E+R)}{E}$ which was in \Cref{app_lem:bound_on_g}. In order to use \Cref{lem:bound_on_g_stochastic}, we need to show $g_{i_k}(\ibar{_k}{k}{_k}{k})-g^* \leq G$ and $\|\nabla\hat g_{i_k}(\ibar{_k}{k}{_k}{k}) - \hat{ u}^k_{i_k} - (\nabla g_{i_k}(\ibar{_k}{k}{_k}{k}) -  u^k_{i_k} )\|\leq F'$. Due to stochasticity, it is not possible to directly bound these values for all the iterations. Instead, we will show that the probabilities of the following events are low up to time $K$:
\begin{equation}\label{eqn:prf_events_stochastic}
\begin{aligned}
    t_1 &:= \min\left\{\min\{k|\max_{i\in[n]} g_{i}(\ibar{}{k+1}{}{k+1})-g^*>G\},K\right\},\\
    t_2 &:= \min\left\{\min\{k|\|\nabla\hat g_{i_k}(\ibar{_k}{k}{_k}{k}) - \hat{ u}^k_{i_k} - (\nabla g_{i_k}(\ibar{_k}{k}{_k}{k}) -  u^k_{i_k} )\|>F'\}, K\right\},\\
    t &:= \min\{t_1,t_2\},
\end{aligned}
\end{equation}
In \Cref{eqn:prf_events_stochastic}, the event $t_1=K$ will ensure $g_{i_k}(\ibar{_k}{k}{_k}{k}) - g^*\leq G$ before time $k<K$ and the event $t_2=K$ will ensure $\|\nabla\hat g_{i_k}(\ibar{_k}{k}{_k}{k}) - \hat{ u}^k_{i_k} - (\nabla g_{i_k}(\ibar{_k}{k}{_k}{k}) -  u^k_{i_k} )\|\leq F'$ before time $k<K$.
Next, we should show that the probability of the event $\{t<K\}$ is low. 
Alternatively, we can show a low probability for the event $\{t=t_2<K\}\cup \{t=t_1<K,t_2=K\}$.
This is a similar technique to \citep{li2024convex} in order to show convergence in the stochastic setting. 
Compared to their work, our proposed method in \Cref{eqn:stochastic_multi-block_update} targets a more general class of functions (\bdc). 
Although the generality of our function class, our guarantee in \Cref{thm:sbdc_convergence} requires only the first components of our \bdc structure to be $\ell$-smoothness.
In this sense, our work generalizes the prior result by \citet{li2024convex}. 
The main convergence result is given in the following theorem (see \Cref{app:prf_sbdc_convergence} for the proof). This result is the formal statement of \Cref{thm:sbdc_convergence}.

\begin{prop}\label{prop:sbdc_convergence}
Suppose Assumptions~\ref{ass:Lipschitz_h}, \ref{ass:differentiability}, and~\ref{ass:bounded_var} hold, and let
$\{\bm\theta^k\}_{k=0}^K$ be generated by
\eqref{eqn:stochastic_multi-block_update} from $\bm\theta^0\in \mathcal X$.
Assume that, for every block $i$ and any fixed $\Bar{\bm\theta}_i$,
$g_i(\cdot;\Bar{\bm\theta}_i)$ is $\ell$-smooth with subquadratic
$\ell$. Let $g^*>-\infty$ be a common lower bound for all $g_i$, and
assume that
\[
h_{i_k}(\ibar{_k}{k}{_k}{k})
-h_{i_0}(\ibar{_0}{0}{_0}{0})\leq H
\]
for a constant $H\geq0$, for any chosen block $i_k$, $i_0$ and every iterate
$0\leq k\leq K$.
For $\epsilon>0$, $0<\delta<1$, and $K\geq1$, choose positive
$G,E,L,\rho,F'$ satisfying
\begin{align*}
C'&:=\frac{K\sigma^2}{\rho}+H,
&
G&:=\frac{8\left(\max_j g_j(\theta_j^0;\Bar{\bm\theta}_j^0)
-g^*+C'\right)}{\delta},\\
E&:=\sup\{u>0:u^2\leq2\ell(2u)G\}<\infty,
& L&:=\ell(2E),\\
F'&:=\frac{E\rho}{9L}-(E+R),
& \sigma^2&=\mathcal O(K^{-1/2}),\\
\rho&:=\left(18L+\frac{9ER}{G}
+\frac{81L}{4}\frac{C'-H}{C'}\right)\sqrt K.
\end{align*}
When $C'=0$, the last term in parentheses is taken to be zero.
Suppose, in addition, that
\[
K\geq\frac{4n(L+\rho)}{\delta\epsilon^2}
\left[
\frac{G\delta}{8}
+\frac{3(E+R)^2+(2K+1)\sigma^2}{2\rho}
\right].
\]
Then, with probability at least $1-\delta$ over both the samples and
the random block choices,
\[
\min_{0 \leq k< K}\mathcal{\bm G}^2(\bm\theta^k)\leq\epsilon^2.
\]
\end{prop}
Now, we state the formal version of \Cref{thm:sbdc_convergence}.

\subsection{Background on $(r,\ell)$-Smoothness and $\ell$-Smoothness}\label[appendix]{app:discussion_l_smoothness}
Here, we discuss the required background and results on $\ell$-smoothness. We mainly represent the results from \citep{li2024convex} and briefly explain the results and connections with this work.

We start with the following lemma characterizing a local descent condition for any \(x\in\mathcal{X}\) when \(g\) is \((r,\ell)\)-smooth:
\begin{lemma}[\cite{li2024convex}]\label{lemma:rlsmoth-descent}
    If \(g\) is \((r,\ell)\)-smooth, for any \(x\in\mathcal{X}\) satisfying \(\|\nabla g(x)\|\leq E\) we have $ \mathcal{B}(x,r(E))\subset \mathcal{X}, $ and for any $x_1,x_2\in\mathcal{B}(x,r(E)),$
        \[\|\nabla g(x_2)-\nabla g(x_1)\|\leq L\|x_1-x_2\|\quad g(x_2)\leq g(x_1) + \langle \nabla g(x_1),x_2-x_1 \rangle + \frac{L}{2}\|x_1-x_2\|^2\]
    where \(L=\ell(E)\) is the effective smoothness.
\end{lemma}

The following proposition, bridges $\ell$-smoothness and $(r,\ell)$-smoothness. The importance of this result is due to the fact that it shows applicability of the descent \Cref{lemma:rlsmoth-descent} on $\ell$-smooth functions. 
\begin{prop}[\citet{li2024convex}]\label{prop:sbdc_rlsmooth}
    An \((r,l)\)-smooth function is \(l\)-smooth; and an \(l\)-smooth function is \((r,m)\)-smooth with \(m(u):=l(u+a)\) and \(r(u):=a/m(u)\) for any \(a>0\) if \(f\) is a closed function within its open domain \(\mathcal{X}\).
\end{prop}

With this result, one can use \Cref{lemma:rlsmoth-descent} on an $\ell$-smooth function which satisfies the conditions in \Cref{lemma:rlsmoth-descent}: bounded gradients and $(r,\ell)$-smoothness. Also, we need to ensure that the updates remain inside a ball. Despite the convexity of the function $g$ in our problem setup, DCA updates do not guarantee the boundedness of its gradients. Therefore, we use the following corollary which provides such bound when the function \(\ell\) is sub-quadratic in the sense that $\lim_{t\rightarrow\infty}\nicefrac{\ell(t)}{t^2} = 0$.
\begin{corr}[\cite{li2024convex}]\label{corr:lsmoth-bounded-grad}
    Suppose \Cref{ass:differentiability} and take \(g\) as \(\ell\)-smooth with sub-quadratic \(\ell\). If \(g(x)-inf_{y\in\mathcal{X}} g(y)\leq G\) for some \(x\in\mathcal{X}\) and \(G\geq 0,\) then \(E^2= 2\ell(2E)G\) and \(\|\nabla g(x)\|\leq E <\infty\) for \(E:=\sup\{u\geq 0 | u^2\leq 2\ell(2u)G\}\)
\end{corr}
With \Cref{corr:lsmoth-bounded-grad}, if we can show that the updates remain inside a ball, then the descent condition in \Cref{lemma:rlsmoth-descent} holds.

\subsection{Useful Lemma on Gradient Estimation Variance}
\label[appendix]{app:section_discussion_variance}

The following lemma is a classical result on the variance in terms of the mini-batch size. We used this Lemma for the discussions on our reduced variance assumption in \Cref{thm:sbdc_convergence}.
\begin{lemma}[Lemma 2 from \citep{reddi2016stochastic} ]\label{app:lem_batch_variance}
Suppose that $\mathcal{S}^{k}$ is a subset that samples $S^{k}$ i.i.d realizations from the distribution $\mathcal{P}$. Let the stochastic estimator $\nabla f(\bm\theta^k,s^k)$ satisfy the bounded variance condition \Cref{ass:bounded_var}. Then, the following bound holds:
\begin{equation}
    \mathbb{E}\left[\|\nabla f(\bm\theta^k,s^k) - \nabla f(\bm\theta^k)\|^2\right]\leq \frac{\sigma^2}{S^k},\quad  \forall \bm\theta\in\mathcal{X}.
\end{equation}
\end{lemma}

\subsection{Additional Details on Numerical Examples} \label[appendix]{app:simulation}
In this section, we provide the reader with more detail on our implementation settings and parameter choices. 

\subsubsection{Generalized Smoothness on Deep Networks.}
\label[appendix]{app:simulation_gen_smooth}
The relationship between the Hessian of the objective function in training language models and the norm of its gradient was already observed in \citep{zhang2019gradient}. This relationship was later extended to more general cases by \citet{li2024convex}. The previous analyses heavily relied on the trajectory of the optimization guided by GD updates. Here, we want to show that a similar relationship exists between the estimated smoothness of the first \bdc component and its gradient norm when the updates are done by the \bdca. In order to do this, we use the same smoothness estimator as in \citep{Santurkar2018HowDB} defined below:
\begin{align}\label{eqn:estim_smooth}
    \hat L_{g_i}(\bm\theta^k)=\max_{\gamma\in\{\delta,2\delta,\ldots,1\}} \frac{\|\nabla g_i(\theta^k_{i}+\gamma d)-\nabla g_i(\theta^k_i)\|}{\|\gamma d\|},
\end{align}
for a small value $\delta$ and $d=\theta_i^{k+1} - \theta_i^{k}$. This value determines the variations along $d$ on the block $i$. Note that unlike previous results, in \bdca we do not necessarily decrease the value of the first component along the update trajectory. To show this, we conducted numerical simulations on a regression task using a three-layer ReLU network of size $(8\times 64 \times 32 \times 1)$ on the California Housing dataset \citep{KELLEYPACE1997291}. We considered training for 30 epochs, a learning rate of $0.5\times 10^{-3}$ with 10 oracle calls to the \bdca sub-problem solver. We set $\delta = 0.25$. 
For each active block, we evaluate the first convex component on the current mini-batch, holding the complementary parameters fixed. We used the block subgradient returned by automatic differentiation in \eqref{eqn:estim_smooth} to estimate $\hat L_{g_i}(\bm\theta^k)$. 
Numerical diagnostics indicate that \(g_i\) was differentiable at all evaluated points.
The logarithm of the estimated smoothness constant of the first \bdc component in \eqref{eq:bdc-loss-pos} was depicted against its gradient norm for each block in \Cref{fig:gen_smoothness}. This figure motivates that a sub-quadratic relationship between the layer-wise smoothness constant and their gradient norms exists. Such condition is required for our convergence result in \Cref{thm_bdc_gen_smooth_convergence} and \Cref{thm:sbdc_convergence}. 

\subsubsection{Sparse Dictionary Learning}
\label[appendix]{app:simulation_sdl}
Here, we explain the implementation structures of the sparse dictionary learning problem with more detail. Note that the structure of SDL problem fits with the more general analysis provided in \Cref{app:BDCA_determined_smthness}.

\paragraph{Implementation Details.}
Both formulations (with $\ell_1$ norm and \eqref{eq:sdl-lq}) are solved via alternating minimization. In each iteration, we first update \(X\): for the \(\ell_1\) model, we use GD; for the nonconvex model \eqref{eq:sdl-lq}, we employ the \dc algorithm by linearizing the \(\|\cdot\|_Q\) term and then applying GD to the resulting convex surrogate. Next, we update \(\mathbf{D}\) using a Frank–Wolfe procedure, projecting onto \(\mathcal{C}\) to enforce the unit-\(\ell_2\) constraints. A line search determines the optimal step size in each Frank–Wolfe update. We evaluate performance by the reconstruction error \(\|Y - \mathbf{D}X\|_F^2\) and the proportion of zeros in \(X\). Each experiment is repeated 10 times, and we report a 95\% probabilistic bound in our plots. We compare the formulations on synthetic data and Berkeley segmentation dataset \cite{MartinFTM01}. 

\paragraph{Synthetic Data.}
We set \(m=10\), \(l=32\), and \(n=100\). A ground-truth dictionary \(\mathbf{D}^*\in\mathbb{R}^{m\times l}\) is generated by sampling each entry i.i.d.\ from \(\mathcal{N}(0,1)\) and normalizing each column to unit \(\ell_2\)-norm. The true sparse code matrix \(X^*\in\mathbb{R}^{l\times n}\) has exactly five nonzero entries per column, drawn i.i.d.\ from \(\mathcal{N}(0,1)\). We synthesize the data as \(Y = \mathbf{D}^*X^*\), using \(\alpha=0.1\) and \(Q=5\). Results are shown in \Cref{fig:SDL}.

\begin{figure}[t]
    \centering
    \includegraphics[width=\linewidth]{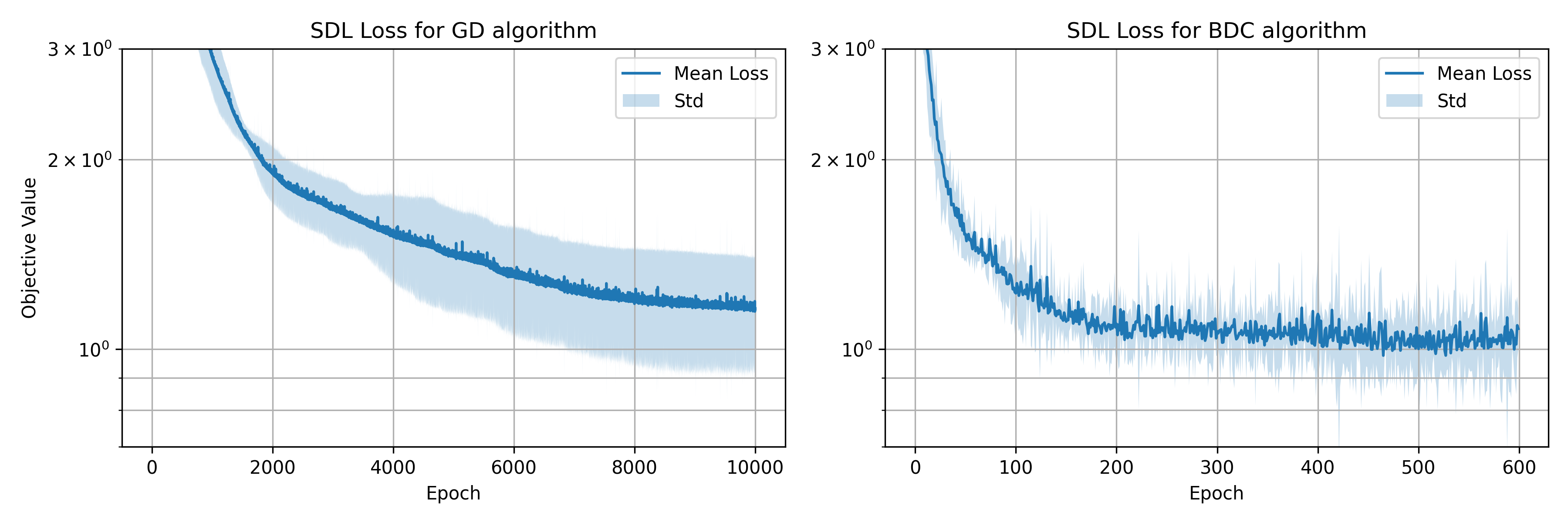}
    \caption{
    SDL loss evolution with the nonconvex \(\ell_1 - \ell_Q\) regularizer (Eq.~\ref{eq:sdl-lq}) on synthetic data.
    The left panel shows the loss trajectory of the joint full-batch gradient descent method with an adaptive step size, while the right panel shows the loss evolution of the \bdc\ algorithm.
    Each curve reports the mean and $\pm 2$ standard deviations over 3 independent runs.
    }
    \label{fig:sdl-loss-comparison}
\end{figure}

In addition to the previous experiment, we also compare the loss behavior of the SDL problem under the GD and \bdc\ algorithms for the nonconvex \(\ell_1 - \ell_Q\) regularizer (Eq.~\ref{eq:sdl-lq}). 
At each iteration, the gradient descent algorithm performs a single joint full-batch update of both the dictionary \(D\) and the code matrix \(X\), using the adaptive step size
$$
\eta = \frac{1}{\|D\|_2^2 + \|X\|_2^2}
$$
which is motivated by the block Lipschitz constants of the smooth reconstruction term.
Dictionary feasibility is maintained by projecting each column of \(D\) onto the unit \(\ell_2\)-ball after every update.
The resulting loss curves for both GD and \bdc\ are shown in \Cref{fig:sdl-loss-comparison}.
As shown in \Cref{fig:sdl-loss-comparison}, \bdca converges faster and attains a lower objective value compared to GD.

\paragraph{Berkeley Segmentation Dataset. \cite{MartinFTM01}}
From the BSDS500 training set (200 images), we randomly extract 50 grayscale patches of size \(8\times8\) from each image. Any patch that is identically zero is discarded; the remaining patches are demeaned and normalized to unit \(\ell_2\)-norm, then assembled as columns of \(Y\). For this experiment we use \(\alpha=0.2\), \(Q=5\), and \(l=256\). Results are shown in \Cref{fig:SDL}.

\subsubsection{Training Neural Networks}
\label[appendix]{app:simulation_nn}

Here, we explain the implementation structures of the training problem with more detail.

\paragraph{Implementation details (Classification Task).}
We evaluate CIFAR10 and \textsc{FashionMNIST} using a batch size of $4096$ and 10 random
seeds. For \textsc{FashionMNIST}, we use a three-layer ReLU network with sequential
input-output dimensions $(28\times 28,512,64,10)$ and train for 100 epochs.
For SGD, the step size is $8\times10^{-2}$. For \eqref{eqn:stochastic_multi-block_update}, each selected-block subproblem is solved by 50 inner gradient steps with step size $10^{-2}$, and we set $\rho=10^{-3}$. For CIFAR10, we use a four-layer ReLU network with sequential input-output dimensions $(3\times32\times32,256,128,64,10)$ and train for 50 epochs. For SGD, the step size is $10^{-1}$. For \eqref{eqn:stochastic_multi-block_update}, each selected-block subproblem is solved by 150 inner gradient steps with step size
$5\times10^{-3}$, and we set $\rho=10^{-1}$. Figure~\ref{fig:NN} reports the
mean curves and two-sided $90\%$ Student-$t$ confidence intervals across the
10 runs.

\newpage
\section{Proofs}\label[appendix]{app:proofs}

\subsection{Proof of \Cref{prop:bdc-Closure}}
\label[appendix]{app:prf_bdc-Closure}
Fix any block $i\in[n]$ and fix an arbitrary complement $\bar{\bm\theta}_i\in\bar{\mathcal X}_i$. 
We work with the $i$-th block (with $\bar{\bm\theta}_i$ fixed and $ \theta_i$ free). 
By \bdc assumption, for each $r\in\{1,\dots,m\}$ there exist functions
\[
g_i^{(r)}(\cdot\,;\bar{\bm\theta}_i),\ h_i^{(r)}(\cdot\,;\bar{\bm\theta}_i):\ \mathcal X_i\to\mathbb R
\]
that are convex in $ \theta_i$ such that
\[
f_r(\bm\theta)
\;=\;g_i^{(r)}(\theta_i;\bar{\bm\theta}_i)\;-\;h_i^{(r)}( \theta_i;\bar{\bm\theta}_i),
\qquad \bm \theta_i\in\mathcal X_i.
\]
We show that each operation preserves this \bdc form.
\paragraph{ Linear combinations.}
Let $\alpha_1,\dots,\alpha_m\in\mathbb R$ and write $\alpha_r=\alpha_r^+-\alpha_r^-$ with $\alpha_r^\pm\ge0$. Then, for every $\bm \theta_i\in\mathcal X_i$,
\begin{align*}
\sum_{r=1}^m \alpha_r\, f_r(\bm\theta)
&=\sum_{r=1}^m \alpha_r^+\!\big(g_i^{(r)}(\theta_i;\bar{\bm\theta}_i)-h_i^{(r)}(\theta_i;\bar{\bm\theta}_i)\big) \\
& \quad  -\sum_{r=1}^m \alpha_r^-\!\big(g_i^{(r)}(\theta_i;\bar{\bm\theta}_i)-h_i^{(r)}(\theta_i;\bar{\bm\theta}_i)\big)\\
&=\underbrace{\Big(\sum_{r=1}^m \alpha_r^+\, g_i^{(r)}(\theta_i;\bar{\bm\theta}_i)
               +\sum_{r=1}^m \alpha_r^-\, h_i^{(r)}(\theta_i;\bar{\bm\theta}_i)\Big)}_{\text{convex in }\bm\theta_i} \\
& \quad
 -\underbrace{\Big(\sum_{r=1}^m \alpha_r^+\, h_i^{(r)}(\theta_i;\bar{\bm\theta}_i)
               +\sum_{r=1}^m \alpha_r^-\, g_i^{(r)}(\theta_i;\bar{\bm\theta}_i)\Big)}_{\text{convex in }\bm\theta_i}.
\end{align*}
Each bracket is a nonnegative sum of convex functions of $\theta_i$, hence convex. Therefore $\sum_{r=1}^m \alpha_r f_r$ is \bdc.
\paragraph{ Maximum.}
Using the \bdc decompositions of all $f_r$, for every $\theta_i\in\mathcal X_i$,
\begin{align}
\max_{1\le r\le m} f_r(\bm\theta)
&=\max_{1\le r\le m}\Big\{\,g_i^{(r)}(\theta_i;\bar{\bm\theta}_i)-h_i^{(r)}(\theta_i;\bar{\bm\theta}_i)\,\Big\} \nonumber\\
&=\max_{1\le r\le m}\Big\{\,g_i^{(r)}(\theta_i;\bar{\bm\theta}_i)
        +\sum_{\substack{s=1\\ s\neq r}}^{m} h_i^{(s)}(\theta_i;\bar{\bm\theta}_i)\,\Big\}
 \;-\;\sum_{k=1}^{m} h_i^{(k)}(\theta_i;\bar{\bm\theta}_i).
\label{eq:bdc-hartman-block-sec}
\end{align}
For the fixed $\bar{\bm\theta}_i$, each inner map
\[
\theta_i\ \mapsto\ g_i^{(r)}(\theta_i;\bar{\bm\theta}_i)+\sum_{\substack{s=1\\ s\neq r}}^{m} h_i^{(s)}(\theta_i;\bar{\bm\theta}_i)
\]
is convex in $\theta_i$ (sum of convex functions); the pointwise maximum over finitely many convex functions is convex in $\theta_i$; and the final sum $\sum_{k=1}^m h_i^{(k)}(\theta_i;\bar{\bm\theta}_i)$ is convex in $\theta_i$. Hence the right-hand side of \eqref{eq:bdc-hartman-block-sec} is a difference of two convex functions of $\theta_i$, proving that $\max_{r} f_r$ is \bdc.
\paragraph{ Minimum.}
By the linear combination with $\alpha_r=-1$, the function $-f_r$ is \bdc\ for each $r$. Applying the Maximum property to $\{-f_r\}_{r=1}^m$ and using
\[
\min_{1\le r\le m} f_r(\bm\theta)\;=\;-\,\max_{1\le r\le m}\big(-f_r(\bm\theta)\big),
\]
we conclude that $\min_{r} f_r$ is \bdc.

Since the block $i$ was arbitrary, all three operations preserve the \bdc property. 

\subsection{Proof of \Cref{prop:dc_complexity}}
\label[appendix]{app:proof_dc_complexity}
We prove the theorem by treating the \emph{even-degree} and \emph{odd-degree} cases separately. In each case we first derive an \emph{upper bound} via the polarization identity of monomials together with a precise pairing argument that halves the raw atom count, and then obtain a \emph{lower bound} by relating any \dc decomposition to a (real) Waring decomposition and invoking known rank formulas for monomials.

\noindent\textbf{Preliminaries}

\noindent\emph{Polarization identity of monomials.}
We use the following polynomial identity.

\begin{lemma}[Polarization identity of monomials {\cite{kan2008moments}}]\label{lem:polar}
Let \(b_1,\dots,b_M\in\mathbb{Z}_{\ge 0}\) with \(S=\sum_{i=1}^M b_i\) and variables \(\theta_1,\dots,\theta_M\). Then
\[
\prod_{i=1}^M \theta_i^{\,b_i}
=\frac{1}{S!}
\sum_{v_1=0}^{b_1}\cdots\sum_{v_M=0}^{b_M}
(-1)^{\sum_{i=1}^M v_i}
\prod_{i=1}^M\binom{b_i}{v_i}\,
\Bigl(\sum_{i=1}^M\bigl(\tfrac{b_i}{2}-v_i\bigr)\theta_i\Bigr)^{\!S}.
\]
\end{lemma}

\emph{Waring decompositions and ranks.}
A \emph{Waring decomposition} of a degree-\(S\) homogeneous polynomial (form) \(F\) is an identity
\[
F(\bm\theta)=\sum_{j=1}^{r} c_j\,\ell_j(\bm\theta)^S,
\qquad
\ell_j(\bm\theta)=a_{j1}\theta_1+\cdots+a_{jn}\theta_n .
\]
We distinguish two notions:
\begin{itemize}
\item \textbf{Complex Waring rank} \(\mathrm{rk}_{\mathbb{C}}(F)\): the minimal \(r\) for which there exist real scalars \(c_j\) and linear forms \(\ell_j\) with complex coefficients such that \(F=\sum_{j=1}^r c_j\,\ell_j^S\).
\item \textbf{Real Waring rank} \(\mathrm{rk}_{\mathbb{R}}(F)\): the minimal \(r\) for which there exist \emph{real} \(c_j\) and \emph{real-coefficient} \(\ell_j\) such that \(F=\sum_{j=1}^r c_j\,\ell_j^S\).
\end{itemize}
Allowing complex coefficients cannot increase the minimum, hence
\[
\mathrm{rk}_{\mathbb{C}}(F)\;\le\;\mathrm{rk}_{\mathbb{R}}(F).
\]

\noindent
\textbf{Case 1: \(s\) even (linear-even atoms \((u^\top \bm\theta)^s\)).}

First, we upper bound $N$ via the polarization identity.
Apply \Cref{lem:polar} with \(M=n\), \(S=s\). This expresses \(f\) as a linear combination of \(\prod_{i=1}^n(b_i+1)\) degree-\(s\) powers of linear forms:
\[
f(\bm\theta)=\frac{1}{s!}\sum_{v_1=0}^{b_1}\cdots\sum_{v_n=0}^{b_n}
(-1)^{\sum_{i=1}^n v_i}\,\Big(\textstyle\prod_{i=1}^n\binom{b_i}{v_i}\Big)\,
\Bigl(\sum_{i=1}^n\bigl(\tfrac{b_i}{2}-v_i\bigr)\theta_i\Bigr)^{\!s}.
\]
Write \(v=(v_1,\dots,v_n)\) and let the \emph{complement} be \(b-v=(b_1-v_1,\dots,b_n-v_n)\).
Since \(s\) is even, we have
\[
\Bigl(\sum_{i=1}^n\bigl(\tfrac{b_i}{2}-(b_i-v_i)\bigr)\theta_i\Bigr)^{\!s}
=\Bigl(-\sum_{i=1}^n\bigl(\tfrac{b_i}{2}-v_i\bigr)\theta_i\Bigr)^{\!s}
=\Bigl(\sum_{i=1}^n\bigl(\tfrac{b_i}{2}-v_i\bigr)\theta_i\Bigr)^{\!s},
\]
so the two atoms coincide.
Thus each complementary pair \(\{v,b-v\}\) contributes \emph{twice} the same atom, and pairing halves the count to \(\tfrac12\prod_{i=1}^n(b_i+1)\).
If all \(b_i\) are even, when \(v_i=b_i/2\), we have $\frac{b_i}{2} - v_i = 0$, so the exact number of nonzero atoms equals \(\bigl(\prod_{i=1}^n(b_i+1)-1\bigr)/2=\bigl\lfloor \tfrac12\prod_{i=1}^n(b_i+1)\bigr\rfloor\).
Therefore,
\[
N \;\le\; \Bigl\lfloor \tfrac12\prod_{i=1}^n(b_i+1)\Bigr\rfloor .
\]

Second, we find the lower bound $N$ via Waring rank.
Any \dc split in this model can be written as
\[
f(\bm\theta)=g(\bm\theta)-h(\bm\theta)
=\sum_{i=1}^{r}\alpha_i\,(u_i^\top \bm\theta)^s
\;-\;\sum_{i=r+1}^{r+q}\alpha_i\,(u_i^\top \bm\theta)^s
=\sum_{i=1}^{N} c_i\,(u_i^\top \bm\theta)^s,
\]
with \(\alpha_i>0\), \(c_i=\pm\alpha_i\), and \(N=r+q\). This is a Waring decomposition with real coefficients and real linear forms, hence
\[
N \;\ge\; \mathrm{rk}_{\mathbb{R}}(f) \;\ge\; \mathrm{rk}_{\mathbb{C}}(f).
\]
For monomials the complex rank is known exactly:

\begin{lemma}[Complex Waring rank of a monomial {\cite{carlini2012solution}}]\label{lemma:waring}
Let \(b_1,\dots,b_n\in\mathbb{Z}_{\ge 0}\) with \(s=\sum_{i=1}^n b_i\), and \(1\le b_1\le\cdots\le b_n\).
For the monomial \(f(\bm\theta)=\theta_1^{b_1}\cdots \theta_n^{b_n}\),
\[
\mathrm{rk}_{\mathbb{C}}(f)=\prod_{i=2}^n(b_i+1).
\]
\end{lemma}

Combining \(N\ge \mathrm{rk}_{\mathbb{C}}(f)\) with \Cref{lemma:waring} yields
\[
N \;\ge\; \prod_{i=2}^n(b_i+1).
\]
Finally, the relationship between real and complex ranks clarifies tightness:

\begin{theorem}[{\cite{carlini2017real}}]\label{thm:RC}
Let \(f(\bm\theta)=\theta_1^{b_1}\cdots \theta_n^{b_n}\) with \(1\le b_1\le\cdots\le b_n\). Then
\(\mathrm{rk}_{\mathbb{R}}(f)=\mathrm{rk}_{\mathbb{C}}(f)\) if and only if \(b_1=1\).
\end{theorem}

Hence, when \(b_1=1\) the lower bound \(\prod_{i=2}^n(b_i+1)\) equals the complex (and real) rank, and together with the polarization upper bound we obtain matching bounds; for \(b_1>1\) a strict gap can remain.

\noindent
\textbf{Case 2: \(s\) odd (affine-even atoms \((u^\top \bm\theta+ \kappa)^{s+1}\)).}
We now handle \(s\) odd, where \((u^\top \bm\theta)^s\) is not convex. To remain in a convex-atom setting we use \emph{even-degree affine atoms}, obtained via degree \((s+1)\) homogenization.

Define
\[
F(\bm\theta,t)\;=\;t\,f(\bm\theta)\;=\;t\,\theta_1^{b_1}\cdots \theta_n^{b_n},
\]
which is homogeneous of even degree \(S=s+1\) in the variables \((\bm\theta,t)\in\mathbb{R}^n\times\mathbb{R}\). Any atom of the form \((u^\top \bm\theta + \kappa\,t)^{S}\) is convex (even power of an affine form). Evaluating any homogeneous decomposition of \(F\) at \(t=1\) yields atoms \((u^\top \bm\theta+ \kappa)^{s+1}\), which remain convex in \(\bm\theta\).

Now, we upper bound $N$ via the polarization identity and pairing.
Apply \Cref{lem:polar} to the \((n{+}1)\)-variate monomial \(t\,\theta_1^{b_1}\cdots \theta_n^{b_n}\) with \((z_0,\dots,z_n)=(t,\theta_1,\dots,\theta_n)\), \((b_0,\dots,b_n)=(1,b_1,\dots,b_n)\), and \(S=s+1\). We obtain
\[
F(\bm\theta,t)=\frac{1}{S!}\sum_{v_0=0}^{1}\sum_{v_1=0}^{b_1}\cdots\sum_{v_n=0}^{b_n}
(-1)^{\sum_{i=0}^n v_i}\Big(\textstyle\prod_{i=0}^n \binom{b_i}{v_i}\Big)\,
\Bigl(\sum_{i=0}^n\bigl(\tfrac{b_i}{2}-v_i\bigr)z_i\Bigr)^{\!S},
\]
a signed sum of even powers of affine forms \((\kappa\,t+ u^\top \bm\theta)^{S}\).
Pair each index \(v=(v_0,\dots,v_n)\) with its complement \(b-v\).
Hence the raw count
\((b_0+1)\prod_{i=1}^n(b_i+1)=2\prod_{i=1}^n(b_i+1)\) collapses \emph{exactly} by a factor \(2\), yielding
\[
\#\text{atoms in }F \;=\; \prod_{i=1}^n(b_i+1).
\]
Setting \(t=1\) gives a \dc decomposition of \(f\) with convex affine-power atoms \((u^\top \bm\theta+\kappa)^{s+1}\) and
\[
N \;\le\; \prod_{i=1}^n (b_i+1).
\]

\emph{Affine vs.\ homogeneous decompositions.}
In the odd-degree case we use degree-\(d=s+1\) powers of \emph{affine} forms \((\ell(\bm\theta)+\beta)^d\).
It is crucial that sums of such affine powers correspond \emph{exactly} to homogeneous sums of degree-\(d\) powers of linear forms in one extra variable, with a \emph{term-by-term} correspondence that preserves the number of terms.

\begin{lemma}[Affine--homogeneous correspondence]\label{lem:affine_homog}
Let \(f:\mathbb{R}^n\to\mathbb{R}\) be of degree \(d\), and let its degree-\(d\) homogenization be \(F(X_0,X)=X_0^{\,d}\,f(X/X_0)\), so that \(F\) is homogeneous of degree \(d\) and \(F(1,\bm\theta)=f(\bm\theta)\).
Then the following are equivalent:
\begin{enumerate}
\item \(f(\bm\theta)=\sum_{j=1}^r c_j\,(\ell_j(\bm\theta)+\beta_j)^d\) (affine sum of degree-\(d\) powers).
\item \(F(X_0,X)=\sum_{j=1}^r c_j\,(\beta_j X_0+\ell_j(X))^d\) (homogeneous sum of degree-\(d\) powers).
\end{enumerate}
Moreover, the number of terms \(r\) is preserved in both directions.
\end{lemma}

\begin{proof}
(1) \(\Rightarrow\) (2): Substitute \(\bm\theta=X/X_0\) and multiply by \(X_0^{\,d}\), then expand:
\(
F(X_0,X)=X_0^{\,d}f(X/X_0)=\sum_j c_j\,(\beta_j X_0+\ell_j(X))^d.
\)
(2) \(\Rightarrow\) (1): Evaluate at \(X_0=1\) to get
\(
f(\bm\theta)=F(1,\bm\theta)=\sum_j c_j\,(\ell_j(\bm\theta)+\beta_j)^d.
\)
Thus the atoms correspond bijectively and the count \(r\) is unchanged.
\end{proof}

By \Cref{lem:affine_homog} with \(d=S=s+1\), every \dc decomposition of \(f\) into affine atoms \((\ell(\bm\theta)+\beta)^{s+1}\) induces a homogeneous Waring decomposition of \(F\) into atoms \((\beta t+\ell(\bm\theta))^{s+1}\) with the \emph{same} number of terms, and conversely any homogeneous decomposition of \(F\) restricts at \(t=1\) to a decomposition of \(f\) with the \emph{same} number of terms. Thus the minimal atom count in our odd-\(s\) \dc model equals the affine Waring rank of \(f\) at degree \(s+1\), which by \Cref{lem:affine_homog} equals the Waring rank of \(F\).

We lower bound $N$ via Waring rank of the lifted monomial.
Any such \dc decomposition of \(f\) induces a decomposition of \(F\) of the form
\[
F(\bm\theta,t)=\sum_{j=1}^{N} c_j\,\bigl(u_j^\top \bm\theta + \kappa_j t\bigr)^{s+1},
\]
which is a real Waring decomposition of the \((n{+}1)\)-variate monomial \(t^{1}\theta_1^{b_1}\cdots \theta_n^{b_n}\) of degree \(S=s+1\). The complex Waring rank of this monomial equals
\[
\mathrm{rk}_{\mathbb{C}}(t^{1}\theta_1^{b_1}\cdots \theta_n^{b_n})=\prod_{i=1}^n(b_i+1),
\]
and, since the smallest exponent is \(1\), real and complex ranks coincide (see \Cref{thm:RC}). Therefore
\[
N\;\ge\;\mathrm{rk}_{\mathbb{R}}(F)\;=\;\mathrm{rk}_{\mathbb{C}}(F)\;=\;\prod_{i=1}^n(b_i+1).
\]
Together with the upper bound we conclude
\[
N \;=\; \prod_{i=1}^n(b_i+1).
\]

\noindent\textbf{Conclusion.}
For even \(s\), the polarization identity and complementary-index pairing yield \(N\le \bigl\lfloor \tfrac12\prod_{i=1}^n(b_i+1)\bigr\rfloor\), while the Waring-rank argument gives \(N\ge \prod_{i=2}^n(b_i+1)\); when \(b_1=1\) these bounds are tight. For odd \(s\), the degree-\((s+1)\) homogenization \(F(\bm\theta,t)=t f(\bm\theta)\), the polarization identity in \(n{+}1\) variables, and the corresponding Waring-rank lower bound match exactly, giving \(N=\prod_{i=1}^n(b_i+1)\).

\subsection{Proof of \Cref{thm:bdc-struct-Relu}}
\label[appendix]{app:proof-bdc-struct-Relu}

Fix an arbitrary block $ \theta_l$ and hold all other blocks fixed. When we refer to `convex' for a vector-valued function in the proof, we mean it in the componentwise sense. We consider two cases. 

\textbf{Case 1: Hidden block $ \theta_l=(  W_l,  b_l)$.}

First,
We prove by induction on $k$ that $  Z_k^\pm$ are componentwise convex in $(  W_l,  b_l)$ and satisfy $  Z_k^\pm\!\ge\!  0$.

\emph{Base ($k<l$):} For $k<l$, the quantities $  Z_k^\pm$ do not depend on $(  W_l,  b_l)$ and are thus constant (hence convex) w.r.t.\ $(  W_l,  b_l)$. It remains to justify nonnegativity for these layers. at $k=1$,
\[
  Z_1^+=\sigma(  W_1  x+  b_1)\ge   0,\qquad   Z_1^-=  0.
\]
Assume $  Z_s^\pm\ge   0$ for some $s<l-1$. Then
\[
  Z_{s+1}^-=\sigma(  W_{s+1})  Z_s^-+\sigma(-  W_{s+1})  Z_s^+\ \ge\   0,
\]
because $\sigma(  W_{s+1})$ and $\sigma(-  W_{s+1})$ are entrywise nonnegative. Moreover,
\[
  Z_{s+1}^+=\max\!\Big\{\underbrace{\sigma(  W_{s+1})  Z_s^+ + \sigma(-  W_{s+1})  Z_s^- +   b_{s+1}}_{  p_{s+1}},\   Z_{s+1}^- \Big\}\ \ge\   Z_{s+1}^- \ \ge\   0.
\]
By induction, $  Z_k^\pm\ge   0$ for all $k<l$.

\emph{Layer $k=l$:} With $  Z_{l-1}^\pm\ge   0$ fixed,
\[
  p_l=\sigma(  W_l)  Z_{l-1}^+ + \sigma(-  W_l)  Z_{l-1}^- +   b_l,
\quad
  Z_l^-=\sigma(  W_l)  Z_{l-1}^- + \sigma(-  W_l)  Z_{l-1}^+ .
\]
Entrywise $w\mapsto\sigma(\pm w)$ are convex and nonnegative; multiplying by fixed nonnegative vectors $  Z_{l-1}^\pm$ and adding the affine term $  b_l$ preserve convexity. Hence $  p_l$ and $  Z_l^-$ are convex, with $  Z_l^-\ge  0$. Set
\[
  Z_l^+ \;=\; \max\{  p_l,  Z_l^-\},
\]
which is convex (pointwise max preserves convexity) and satisfies $  Z_l^+\ge   Z_l^-\ge   0$.

\emph{Induction ($k\to k{+}1$ for $k\ge l$):} Assume $  Z_k^\pm$ are convex in $(  W_l,  b_l)$ and $  Z_k^\pm\ge   0$. For fixed $(  W_{k+1},  b_{k+1})$, the matrices $\sigma(  W_{k+1})$ and $\sigma(-  W_{k+1})$ are entrywise nonnegative constants. Thus
\[
  p_{k+1}=\sigma(  W_{k+1})  Z_k^+ + \sigma(-  W_{k+1})  Z_k^- +   b_{k+1},\quad
  Z_{k+1}^-=\sigma(  W_{k+1})  Z_k^- + \sigma(-  W_{k+1})  Z_k^+
\]
are nonnegative linear images of $(  Z_k^+,  Z_k^-)$ plus a constant; hence $  Z_{k+1}^-\ge   0$ and both $  p_{k+1},  Z_{k+1}^-$ are convex. Finally,
\[
  Z_{k+1}^+=\max\{  p_{k+1},  Z_{k+1}^-\}
\]
is convex and satisfies $  Z_{k+1}^+\ge   Z_{k+1}^-\ge   0$. By induction, this holds for all $k\ge l$, in particular for $k=L - 1$. Using $\sigma(a-b)=\max\{a,b\}-b$ coordinatewise,
\[
  Z_{k+1}^+ -   Z_{k+1}^- \;=\; \sigma\!\big(  W_{k+1}(  Z_k^+-  Z_k^-)+  b_{k+1}\big),
\]
so $  a_{k+1}=  Z_{k+1}^+-  Z_{k+1}^-$ and in particular $  a_{L-1}=  Z_{L-1}^+-  Z_{L-1}^-$.

At the end, 
keep $(  W_L,b_L)$ fixed. For each class $c$,
\[
\begin{aligned}
A_c(\bm \theta)&=\sigma( W_{L,c}) Z_{L-1}^+ + \sigma(- W_{L,c}) Z_{L-1}^- + \sigma(b_{L,c}),\\
B_c(\bm \theta)&=\sigma( W_{L, c}) Z_{L-1}^- + \sigma(- W_{L,c}) Z_{L-1}^+ + \sigma(-b_{L,c}).
\end{aligned}
\]
Here $\sigma(\pm W_{L,c})\ge  0$ and $\sigma(\pm b_{L, c})\ge 0$ are constants; therefore $A_c(\cdot\; ; \bar{\bm\theta}_l),B_c(\cdot\; ; \bar{\bm\theta}_l)$ are nonnegative linear combinations of the convex functions $ Z_{L-1}^\pm$ plus constants, hence are convex and nonnegative in $ \theta_l = ( W_l, b_l)$. Using $\sigma(t)-\sigma(-t)=t$ entrywise and $ a_{L-1}= Z_{L-1}^+- Z_{L-1}^-$, we obtain
\[
 A(\bm \theta)- B(\bm \theta)
=\big(\sigma( W_L)-\sigma(- W_L)\big)( Z_{L-1}^+- Z_{L-1}^-)
+\big(\sigma(b_L)-\sigma(-b_L)\big)
= W_L a_{L-1}+b_L
=F_{ x}(\bm\theta).
\]

\textbf{Case 2: Output block $\theta_L=( W_L,b_L)$.}

Here $ Z_{L-1}^\pm$ are fixed and nonnegative. The entrywise maps $ W_L\mapsto\sigma(\pm  W_L)$ and $b_L\mapsto\sigma(\pm b_L)$ are convex and nonnegative. Hence each component of $ A(\cdot\; ; \bar{\bm\theta}_L), B(\cdot\; ; \bar{\bm\theta}_L)$ in~\eqref{equ:relu-bdc-dec} is a nonnegative linear combination of nonnegative convex functions, and is therefore convex and nonnegative in $( W_L,b_L)$.

\subsection{Proof of \Cref{thm:bdc-conjugate-compact}}
\label[appendix]{app:proof-bdc-conjugate-compact}
Fix an arbitrary block $i\in[n]$ and fix $\bar{\bm\theta}_i$. By the componentwise \bdc\ assumption,
for each coordinate $j=1,\dots,m$ there exist convex functions $a_{ij}(\cdot\,;\bar{\bm\theta}_i)$ and
$b_{ij}(\cdot\,;\bar{\bm\theta}_i)$ in $ \theta_i$ such that
$E_j(\bm\theta)=a_{ij}(\theta_i;\bar{\bm\theta}_i)-b_{ij}( \theta_i;\bar{\bm\theta}_i)$.  

\emph{Multi-Block convexity of $g$.}
From the conjugate definition and $E_j=a_{ij}-b_{ij}$,
\[
f^{*}(  E(\bm\theta))
=\max_{  u\in U}\Big\{\langle   u,  a_i( \theta_i;\bar{\bm\theta}_i)\rangle
-\langle   u,  b_i(\theta_i;\bar{\bm \theta}_i)\rangle - f(  u)\Big\}.
\]
Adding $h_i$ yields the variational form
\[
g_i(\theta_i;\bar{\bm\theta}_i)
=\max_{  u\in U}\Big\{\langle   u+  c^+,\,  a_i( \theta_i;\bar{\bm\theta}_i)\rangle
+\langle -  u+  d^+,\,  b_i(\theta_i;\bar{ \bm \theta}_i)\rangle - f(  u)\Big\}.
\]
For any fixed $  u\in U$, the map
\[
 \theta_i\ \mapsto\ \langle   u+  c^+,\,  a_i(\theta_i;\bar{ \bm \theta}_i)\rangle
+\langle -  u+  d^+,\,  b_i( \theta_i;\bar{\bm\theta}_i)\rangle - f(  u)
\]
is convex in $ \theta_i$ since $u_j+c_j^+\ge0$ and $-u_j+d_j^+\ge0$ for all $j$, making it a nonnegative linear combination of convex functions. Taking the pointwise maximum over $  u\in U$ preserves convexity, so $g_i(\cdot;\bar{\bm\theta}_i)$ is convex.

\emph{Multi-Block convexity of $h$.}
By definition,
\[
h_i( \theta_i;\bar{\bm\theta}_i)
=\langle   c^+,\,  a_i( \theta_i;\bar{\bm\theta}_i)\rangle
+\langle   d^+,\,  b_i( \theta_i;\bar{\bm\theta}_i)\rangle,
\]
which is a nonnegative linear combination of convex functions of $ \theta_i$, hence convex.

Finally, by construction,
\[
f^{*}(  E(\bm\theta))=g_i( \theta_i;\bar{\bm\theta}_i)-h_i( \theta_i;\bar{\bm\theta}_i),
\]
so $f^{*}\!\circ   E$ admits a multi-block DC decomposition. Since block $i$ was arbitrary, $f^{*}\!\circ   E$ is \bdc.

\subsection{Proof of \Cref{app_lem:bdc_deter_bounded}}\label[appendix]{app:prf_bdc_deter_bounded}
    Due to \eqref{eqn:deterministic_multi-block_update}, we have 
\begin{align}\label{eqn:deterministic_BDCA_subdiffs}
    u_{i_k}^k\in \partial h_{i_k}(\ibar{_k}{k}{_k}{k})\quad \text{and} \quad \langle u_{i_k}^k,  \theta_{i_k}^{k} - \theta_{i_k}^{k+1} \rangle \leq g_{i_k}(\ibar{_k}{k}{_k}{k}) - g_{i_k}(\ibar{_k}{k+1}{_k}{k})-\frac{\rho}{2}\|{\theta}_{i_k}^{k+1} - {\theta}_{i_k}^{k}\|^2.
\end{align}
Using convexity of $g_{i_k}(\cdot\, ,\Bar{\bm\theta}_{i_k}^k)$ in \eqref{eqn:deterministic_BDCA_subdiffs}, we have
\begin{align*}
     \langle u_{i_k}^k,  \theta_{i_k}^{k} - \theta_{i_k}^{k+1} \rangle &\leq g_{i_k}(\ibar{_k}{k}{_k}{k}) - g_{i_k}(\ibar{_k}{k+1}{_k}{k})-\frac{\rho}{2}\|{\theta}_{i_k}^{k+1} - {\theta}_{i_k}^{k}\|^2,\\
      \implies \langle u_{i_k}^k,  \theta_{i_k}^{k} - \theta_{i_k}^{k+1} \rangle &\leq -\ip{\nabla g_{i_k}(\ibar{_k}{k}{_k}{k})}{\theta_{i_k}^{k+1}-\theta_{i_k}^k}-\frac{\rho}{2}\|{\theta}_{i_k}^{k+1} - {\theta}_{i_k}^{k}\|^2,\\
      \implies \frac{\rho}{2}\|{\theta}_{i_k}^{k+1} - {\theta}_{i_k}^{k}\|^2 &\leq \ip{\nabla g_{i_k}(\ibar{_k}{k}{_k}{k}) - u_{i_k}^k}{\theta_{i_k}^{k}-\theta_{i_k}^{k+1}},\\
      \implies  \frac{\rho}{2}\|{\theta}_{i_k}^{k+1} - {\theta}_{i_k}^{k}\|^2 &\leq \|\nabla g_{i_k}(\ibar{_k}{k}{_k}{k}) - u_{i_k}^k\|\|{\theta}_{i_k}^{k+1} - {\theta}_{i_k}^{k}\|,\\
       \implies  \|{\bm\theta}^{k+1} - {\bm\theta}^{k}\| & \leq \frac{2}{\rho} \|\nabla g_{i_k}(\ibar{_k}{k}{_k}{k}) - u_{i_k}^k\|.
\end{align*}
We get that the update $\bm\theta^{k+1}$ is in $\mathcal{B}\left(\bm \theta^k,\frac{2}{\rho}\|\nabla g_{i_k}(\ibar{_k}{k}{_k}{k}) - u_{i_k}^k\|\right)$.

\subsection{Proof of \Cref{app_lem:bound_on_g}}\label[appendix]{app:prf_lem_bound_on_g}
From \eqref{eqn:deterministic_multi-block_update} we know:
\begin{align*}
    u_{i_k}^k\in \partial h_{i_k}(\ibar{_k}{k}{_k}{k})\quad \text{and}\quad \langle  u_{i_k}^k,  \theta_{i_k}^{k} - \theta_{i_k}^{k+1} \rangle \leq g_{i_k}(\ibar{_k}{k}{_k}{k}) - g_{i_k}(\ibar{_k}{k+1}{_k}{k})-\frac{\rho}{2}\|{\theta}_{i_k}^{k+1} - {\theta}_{i_k}^{k}\|^2.
\end{align*}
Now, using convexity of $h_{i_k}$ we get:
\begin{align*}
   h_{i_k}(\ibar{_k}{k}{_k}{k}) - h_{i_k}(\ibar{_k}{k+1}{_k}{k}) &\leq g_{i_k}(\ibar{_k}{k}{_k}{k}) - g_{i_k}(\ibar{_k}{k+1}{_k}{k})-\frac{\rho}{2}\|{\theta}_{i_k}^{k+1} - {\theta}_{i_k}^{k}\|^2,\\
   g_{i_k}(\ibar{_k}{k+1}{_k}{k}) & \leq g_{i_k}(\ibar{_k}{k}{_k}{k}) + h_{i_k}(\ibar{_k}{k+1}{_k}{k}) -h_{i_k}(\ibar{_k}{k}{_k}{k}),\\
   f(\bm\theta^{k+1})&\leq f(\bm\theta^k)
\end{align*}
Unrolling this inequality to the initialization gives
\begin{align*}
    f(\bm\theta^{k+1})&\leq f(\bm\theta^0)\\
    g_{i_k}(\ibar{_k}{k+1}{_k}{k}) & \leq  g_{i_0}(\ibar{_0}{0}{_0}{0}) + h_{i_k}(\ibar{_k}{k+1}{_k}{k}) - h_{i_0}(\ibar{_0}{0}{_0}{0}),\\
   &\leq g_{i_0}(\ibar{_0}{0}{_0}{0}) + H,
\end{align*}
where we have used $h_{i_k}(\ibar{_k}{k+1}{_k}{k}) - h_{i_0}(\ibar{_0}{0}{_0}{0}) \leq H$ and the fact that $\bar{\bm\theta}_{i_k}^{k+1} = \bar{\bm\theta}_{i_k}^{k}$. Since this result holds for any $k$, through \Cref{corr:lsmoth-bounded-grad} we have $\|\nabla g_{i_k}(\ibar{_k}{k}{_k}{k})\|\leq E$.
Recall that $g_{i_k}$ is $\ell$-smooth with a subquadratic $\ell$.
Using \Cref{prop:sbdc_rlsmooth}, $g_{i_k}$ is also $(r,m)$-smooth with $r(u)=\frac{a}{m(u)}$ and $m(u):=\ell(u+a)$ for some $a>0$.
Therefore, we can use \Cref{lemma:rlsmoth-descent} if we ensure the updates are inside $\mathcal{B}(\bm\theta^k,r(E))$. Similar to the proof of \Cref{app_lem:bdc_deter_bounded} (\Cref{app:prf_bdc_deter_bounded}), we know 
\begin{align*}
     \|{\bm\theta}^{k+1} - {\bm\theta}^{k}\| \leq \sup_{ u\in\partial h_{i_k}(\ibar{_k}{k}{_k}{k})} \frac{2(\|\nabla g_{i_k}(\ibar{_k}{k}{_k}{k})\|+\| u\|)}{\rho}\leq \frac{2(E+R)}{\rho}.
\end{align*}
As a result taking $\rho\geq \frac{2(E+R)}{r(E)}=\ell(2E) \frac{2(E+R)}{E}$ will satisfy the conditions in \Cref{lemma:rlsmoth-descent}.
This implies the desired result.

\subsection{Proof of \Cref{app_prop:bdc_gen_smooth_convergence}}\label[appendix]{app:prf_thm_bdc_gen_smooth_convergence}
Using the assumptions in the theorem statement, we know that for any $\bm\theta^k\in\mathcal{X}$, the update $\bm\theta^{k+1}\in \mathcal{B}(\bm\theta^{k},r(E))$.
For any $ \theta_{i_k}\in\mathcal{B}(\theta_{i_k}^{k},r(E))$, consider the surrogate function 
    \begin{equation*}
        \hat f(\ibar{_k}{}{_k}{k}) := g_{i_k}(\ibar{_k}{}{_k}{k}) - h_{i_k}(\ibar{_k}{k}{_k}{k}) - \langle u_{i_k}^k,{\theta}_{i_k}-{ \theta}_{i_k}^k \rangle + \frac{\rho}{2}\|\theta_{i_k}^k - \theta_{i_k}\|^2.
    \end{equation*}
    where \(u_{i_k}^k \in\partial h_{i_k}(\ibar{_k}{k}{_k}{k}) \) is the subgradient used at iteration $k$ of the algorithm. From \eqref{eqn:deterministic_multi-block_update} we know that 
    \[ \hat f(\ibar{_k}{k+1}{_k}{k}) \leq  \hat f(\ibar{_k}{k}{_k}{k}),\]
    and further considering the descent \Cref{lemma:rlsmoth-descent} for \(g_{i_k}\) with $L=\ell(2E)$, we get
    \begin{align*}
     \hat f(\ibar{_k}{k+1}{_k}{k}) \leq   g_{i_k}(\ibar{_k}{k}{_k}{k}) &+ \langle \nabla g_{i_k}(\ibar{_k}{k}{_k}{k}),{\theta}_{i_k}-{ \theta}_{i_k}^k \rangle \\
     &+ \frac{L+\rho}{2}\| \theta_{i_k} - \theta_{i_k}^k\|^2 - h_{i_k}(\ibar{_k}{k}{_k}{k}) - \langle u_{i_k}^k,{\theta}_{i_k}-{ \theta}_{i_k}^k \rangle.
    \end{align*}
    By \Cref{assump:multi-block-DC}, we get
    \begin{equation}\label{eqn:prf_deterministic_gen_smooth_1}
        \begin{aligned}
                  \quad &\hat f(\ibar{_k}{k+1}{_k}{k}) \leq   f({\bm \theta}^k) + \langle \nabla g_{i_k}(\ibar{_k}{k}{_k}{k}) - u_{i_k}^k,{\theta}_{i_k}-{ \theta}_{i_k}^k \rangle + \frac{L+\rho}{2}\|{ \theta}_{i_k} - { \theta}_{i_k}^k\|^2 ,\\
                  & \langle \nabla g_{i_k}(\ibar{_k}{k}{_k}{k})- u_{i_k}^k,{\theta}_{i_k}^k -{\theta}_{i_k} \rangle \leq f({\bm \theta}^k) + \frac{L+\rho}{2}\|{ \theta}_{i_k} - { \theta}_{i_k}^k\|^2 - \hat f(\ibar{_k}{k+1}{_k}{k}) ,\\
                   & \langle \nabla g_{i_k}(\ibar{_k}{k}{_k}{k})- u_{i_k}^k,{\theta}_{i_k}^k -{\theta}_{i_k} \rangle \leq g_{i_k}(\ibar{_k}{k}{_k}{k}) - g_{i_k}(\ibar{_k}{k+1}{_k}{k}) + \frac{L+\rho}{2}\|{ \theta}_{i_k} - { \theta}_{i_k}^k\|^2 \\
                   & - \frac{\rho}{2}\|{ \theta}_{i_k}^{k+1} - { \theta}_{i_k}^k\|^2 + \langle u_{i_k}^k,{ \theta}_{i_k}^{k+1}-{ \theta}_{i_k}^k \rangle,
        \end{aligned}
    \end{equation}
  Using the non-positivity of $-\frac{\rho}{2}\|{ \theta}_{i_k}^{k+1} - { \theta}_{i_k}^k\|^2\leq 0$, and recalling the fact that here we used $\nabla g_i $ and $\partial h_i$ instead of $\nabla_i g_i$ and $\partial_ih_i$, we have
  \begin{equation}
        \begin{aligned}
                    \langle \nabla_{i_k} g_{i_k}(\ibar{_k}{k}{_k}{k})- u_{i_k}^k,{\theta}_{i_k}^k -{\theta}_{i_k} \rangle &- \frac{L+\rho}{2}\|{ \theta}_{i_k} - { \theta}_{i_k}^k\|^2  \\
                    &\leq g_{i_k}(\ibar{_k}{k}{_k}{k}) - g_{i_k}(\ibar{_k}{k+1}{_k}{k})  + \langle u_{i_k}^k,{ \theta}_{i_k}^{k+1}-{ \theta}_{i_k}^k \rangle,\\
                   & \leq g_{i_k}(\ibar{_k}{k}{_k}{k}) - g_{i_k}(\ibar{_k}{k+1}{_k}{k})  +  h_{i_k}(\ibar{_k}{k+1}{_k}{k})-h_{i_k}(\ibar{_k}{k}{_k}{k}),\\
                   &  \leq f({\bm \theta}^k) - f({\bm \theta}^{k+1}).
        \end{aligned}
    \end{equation}
  Define $z_{i_k}^k:= \nabla_{i_k} g_{i_k}(\ibar{_k}{k}{_k}{k})- u_{i_k}^k$. Using the infimum over $\overline{\partial}_{i_k} f(\bm{\theta}^k)$ instead of $z_{i_k}^k$ and maximizing the left hand side with respect to ${ \theta}_{i_k}$ gives
   \begin{align}\label{eqn:prf_lem_cnv_0}
      \inf_{z_{i_k}^k\in \overline{\partial}_{i_k} f(\bm{\theta}^k)} ~\frac{1}{2(L+\rho)}\|z_{i_k}^k\|^2 \leq f({\bm \theta}^k) - f({\bm \theta}^{k+1}).
  \end{align}
  Let us denote by $\mathbb{E}_{|k}$ the conditional expectation with respect to the random selection of $i_k$, given all the random choices in the previous iterations. Then, we have  
    \begin{equation}
    \begin{aligned}\label{eqn:prf_lem_cnv_1}
         \expectk{\inf_{z_{i_k}^k\in \overline{\partial}_{i_k} f(\bm{\theta}^k)} ~\frac{1}{2(L+\rho)}\|z_{i_k}^k\|^2 }
         &= \frac{1}{n} \left(\sum_{i=1}^n  \inf_{z_{i}^k\in \overline{\partial}_{i} f(\bm{\theta}^k)} ~\frac{1}{2(L+\rho)}\|z_{i}^k\|^2 \right), \\
         &= \frac{1}{n} \left(\frac{1}{2(L+\rho)}\sum_{i=1}^n  \text{dist}(0,\overline{\partial}_{i} f(\bm{\theta}^k))^2 \right),
    \end{aligned}
    \end{equation}
    
    Using \eqref{eqn:prf_lem_cnv_0} and \eqref{eqn:prf_lem_cnv_1} we have:
    \begin{equation}
        \frac{1}{2(L+\rho)}\mathcal{\bm G}^2(\bm\theta^k)
        \leq n f({\bm \theta}^k) - n\expectk{f({\bm \theta}^{k+1})}.
    \end{equation}
    Now, taking expectation w.r.t. all the iterations we have
    \begin{equation}
        \expect{ \frac{1}{2(L+\rho)}\mathcal{\bm G}^2(\bm\theta^k)}
        \leq n \expect{f(\bm \theta^k)} - n \expect{f(\bm \theta^{k+1})}.
    \end{equation}
    
    Finally, we take the average of this inequality over $k=1,\ldots,K$:
    \begin{equation}
        \frac{1}{K} \sum_{k=1}^K \expect{ \frac{1}{2(L+\rho)}\mathcal{\bm G}^2(\bm\theta^k)}
        \leq \frac{n}{K} \left(f(\bm{\theta}^1) - \expect{f(\bm{\theta}^{K+1})})\right) 
        \leq \frac{n}{K} \left(f(\bm{\theta}^1) - f^\star\right).
    \end{equation}
    which concludes the proof.
\subsection{Proof of \Cref{prop:sbdc_convergence}}\label[appendix]{app:prf_sbdc_convergence}

 Denote
\[
\bm\epsilon_k:=
\nabla\hat g_{i_k}(\ibar{_k}{k}{_k}{k})-\hat u_{i_k}^k
-\left(\nabla g_{i_k}(\ibar{_k}{k}{_k}{k})-u_{i_k}^k\right).
\]
Let $\mathcal F_k$ contain all information before drawing $i_k$ and $s^k$.
The block $i_k$ is uniform on $[n]$, independently of $\mathcal F_k$, and
$s^k$ is independent of $(\mathcal F_k,i_k)$. The subgradients $u_i^k$
are $\mathcal F_k$-measurable, and $\nabla\hat g_{i_k}$ is the gradient
of the sampled function used in \eqref{eqn:stochastic_multi-block_update}.
\Cref{ass:bounded_var} therefore gives
\[
\mathbb E[\bm\epsilon_k\mid\mathcal F_k,i_k]=0,
\qquad
\mathbb E[\|\bm\epsilon_k\|^2\mid\mathcal F_k,i_k]\leq\sigma^2.
\]
All probabilities and unconditional expectations below are over both
samples and block choices.

We retain the stopping rules in \eqref{eqn:prf_events_stochastic}, with $t_1$ testing
$\max_{i\in[n]}g_i(\ibar{}{k+1}{}{k+1})-g^*>G$ and $t_2$ testing
$\|\bm\epsilon_k\|>F'$, each capped at $K$, and $t=\min\{t_1,t_2\}$.
Here, every component in the maximum is evaluated at the full updated
vector $\bm\theta^{k+1}$. The definition of $G$ ensures
$\max_i g_i(\ibar{}{0}{}{0})-g^*\leq G$. Consequently,
\[
\max_{i\in[n]}g_i(\ibar{}{k}{}{k})-g^*\leq G
\qquad\text{for }0\leq k\leq t_1.
\]
We want to bound the probability of
\[
\{t<K\}=\{t_2<K\}\cup\bigl(\{t_1<K\}\cap\{t_2=K\}\bigr).
\]
For the first event, Chebyshev's inequality and the union bound give
\begin{align}
\mathbb P(t_2<K)
&=\mathbb P\left(\bigcup_{k<K}\{\|\bm\epsilon_k\|>F'\}\right)
\leq\sum_{k<K}\mathbb P(\|\bm\epsilon_k\|>F')
\leq\frac{K\sigma^2}{F'^2}.
\end{align}
We choose parameters so that $K\sigma^2/F'^2\leq\delta/4$.

For $k\leq t_1$ and $k<t_2$, we have
$g_{i_k}(\ibar{_k}{k}{_k}{k})-g^*\leq G$ and
$\|\bm\epsilon_k\|\leq F'$. Thus
$\|\nabla g_{i_k}(\ibar{_k}{k}{_k}{k})\|\leq E$ by
\Cref{corr:lsmoth-bounded-grad}, and the update is bounded by
\Cref{lem:bounding_differences_stoc}. Using
\Cref{lem:bound_on_g_stochastic}, we get
\begin{equation}\label{eqn:event2_1}
    \begin{aligned}
        g_{i_k}(\ibar{_k}{k+1}{_k}{k})-g_{i_k}(\ibar{_k}{k}{_k}{k}) &\leq \ip{\nabla g_{i_k}(\ibar{_k}{k}{_k}{k})}{\theta_{i_k}^{k+1}-\theta_{i_k}^k} + \frac{L}{2}\|\theta_{i_k}^{k+1}-\theta_{i_k}^k\|^2\\
        & \leq \|\nabla g_{i_k}(\ibar{_k}{k}{_k}{k})\|\|\theta_{i_k}^{k+1}-\bm\theta^k_{i_k}\|+ \frac{L}{2}\|\theta_{i_k}^{k+1}-\theta^k_{i_k}\|^2\\
        & \leq E \frac{2}{\rho}\left(E+R+F'\right) + \frac{L}{2}\left[  \frac{2}{\rho}\left(E+R+F'\right) \right]^2.
    \end{aligned}
\end{equation}
Take $F' = \nicefrac{E\rho}{9L}- (E+R)$ and note that $E^2 = 2LG$. $F'$ is positive for $\rho> \nicefrac{9 L(E+R)}{E}$. This is a valid choice of $\rho$ since it satisfies (see \Cref{lem:bound_on_g_stochastic}):
$$\rho\geq \frac{2(E+R+F')}{r(E)}.$$
Now, replacing in \eqref{eqn:event2_1} gives:
\begin{align}
    \begin{aligned}
         g_{i_k}(\ibar{_k}{k+1}{_k}{k})-g_{i_k}(\ibar{_k}{k}{_k}{k}) &\leq \frac{G}{2}.
    \end{aligned}
\end{align}

For the exit probability, we use the population objective so that the
argument applies to whichever component attains the maximum.
Due to \eqref{eqn:stochastic_multi-block_update}, we have $u_{i_k}^k\in \partial h_{i_k}(\ibar{_k}{k}{_k}{k}) $ and
\begin{align}
    \langle \hat{ u}_{i_k}^k,  \theta_{i_k}^{k} - \theta_{i_k}^{k+1} \rangle \leq g_{i_k}(\ibar{_k}{k}{_k}{k},s^k) - g_{i_k}(\ibar{_k}{k+1}{_k}{k},s^k)-\frac{\rho}{2}\|{\theta}_{i_k}^{k+1} - {\theta}_{i_k}^{k}\|^2.
\end{align}
By adding and subtracting $\langle {\hat{ u}}_{i_k}^k - u_{i_k}^k,  \theta_{i_k}^{k+1} - \theta_{i_k}^{k} \rangle$ and using convexity of the sampled function $g_{i_k}(\cdot;\Bar{\bm\theta}_{i_k}^k,s^k)$, we have
\begin{equation*}
    \begin{aligned}
       \langle u_{i_k}^k,  \theta_{i_k}^{k} - \theta_{i_k}^{k+1} \rangle \leq \langle \hat{ u}_{i_k}^k - u_{i_k}^k, \theta_{i_k}^{k+1}& - \theta_{i_k}^{k} \rangle-\ip{\nabla  \hat g_{i_k}(\ibar{_k}{k}{_k}{k})}{\theta_{i_k}^{k+1}-\theta_{i_k}^k}-\frac{\rho}{2}\|{\theta}_{i_k}^{k+1} - {\theta}_{i_k}^{k}\|^2,\\
      \implies \langle  \nabla g_{i_k}(\ibar{_k}{k}{_k}{k}),    \theta_{i_k}^{k+1} - \theta_{i_k}^{k}\rangle \leq & \langle \bm\epsilon_k,  \theta_{i_k}^{k} - \theta_{i_k}^{k+1} \rangle -\frac{\rho}{2}\|{\theta}_{i_k}^{k+1} - {\theta}_{i_k}^{k}\|^2 - \langle u_{i_k}^k ,  \theta_{i_k}^{k} - \theta_{i_k}^{k+1} \rangle, 
      \end{aligned}
\end{equation*}

For $k\leq t_1$ and $k<t_2$, the conditions of \Cref{lem:bound_on_g_stochastic} are satisfied. Therefore, local smoothness of $g_{i_k}$ together with Young's inequality imply:
\begin{align*}
    g_{i_k}(\ibar{_k}{k+1}{_k}{k})-g_{i_k}(\ibar{_k}{k}{_k}{k}) &- \frac{L}{2}\|\theta_{i_k}^{k+1} - {\theta}_{i_k}^{k}\|^2\\ 
    \leq& \frac{\rho}{4}\|\theta_{i_k}^{k+1} - {\theta}_{i_k}^{k}\|^2  + \frac{1}{\rho} \|\bm\epsilon_k\|^2- \frac{\rho}{2}\|\theta_{i_k}^{k+1} - {\theta}_{i_k}^{k}\|^2 - \langle u_{i_k}^k ,  \theta_{i_k}^{k} - \theta_{i_k}^{k+1} \rangle\\
      = &\frac{1}{\rho} \|\bm\epsilon_k\|^2  - \langle u_{i_k}^k ,  \theta_{i_k}^{k} - \theta_{i_k}^{k+1} \rangle - \frac{\rho}{4}\|\theta_{i_k}^{k+1} - {\theta}_{i_k}^{k}\|^2. 
\end{align*}
 Now, using $\rho\geq \nicefrac{9 L(E+R)}{E}$, we know that $\frac{L}{2}\|\theta_{i_k}^{k+1} - {\theta}_{i_k}^{k}\|^2 \leq \frac{\rho}{4}\|\theta_{i_k}^{k+1} - {\theta}_{i_k}^{k}\|^2$. Therefore, we have:
\begin{align*}
     g_{i_k}(\ibar{_k}{k+1}{_k}{k})-g_{i_k}(\ibar{_k}{k}{_k}{k}) \leq & \frac{1}{\rho} \|\bm\epsilon_k\|^2 + h_{i_k}(\ibar{_k}{k+1}{_k}{k})-h_{i_k}(\ibar{_k}{k}{_k}{k}).
\end{align*}

The last inequality follows from convexity of $h_{i_k}$. Since
$f(\bm\theta)=g_i(\theta_i;\Bar{\bm\theta}_i)
-h_i(\theta_i;\Bar{\bm\theta}_i)$ for every block $i$, it follows that
\[
f(\bm\theta^{k+1})-f(\bm\theta^k)
\leq\frac{\|\bm\epsilon_k\|^2}{\rho}
\qquad(k\leq t_1,\ k<t_2).
\]
The assumed bound on $h$ holds for any chosen block. In particular,
$h_i(\ibar{}{m}{}{m})\leq\max_j h_j(\theta_j^0;\Bar{\bm\theta}_j^0)+H$
for every $i$ and every iterate $0\leq m\leq K$ used below. Therefore,
\begin{equation}\label{eqn:sbdc_gen_population_gap}
\begin{aligned}
0\leq\max_i g_i(\ibar{}{m}{}{m})-g^*
&\leq f(\bm\theta^m)-f(\bm\theta^0)
+\max_j g_j(\theta_j^0;\Bar{\bm\theta}_j^0)-g^*+H.
\end{aligned}
\end{equation}
Summing the population descent inequality along each trajectory, before
taking expectations, gives, for
$0\leq m\leq\min\{t_1+1,t_2,K\}$,
\[
\max_i g_i(\ibar{}{m}{}{m})-g^*
\leq\max_j g_j(\theta_j^0;\Bar{\bm\theta}_j^0)-g^*+H
+\frac1\rho\sum_{k<m}\|\bm\epsilon_k\|^2.
\]
On $\{t_1<K,t_2=K\}$, take $m=t_1+1$. The left-hand side exceeds
$G$, and the noise sum is bounded by the sum over all $k<K$. Hence
\begin{equation}\label{eqn:sbdc_gen_event1_1}
\begin{aligned}
\mathbb P(t_1<K,t_2=K)
&\leq\mathbb P\left(
\max_j g_j(\theta_j^0;\Bar{\bm\theta}_j^0)-g^*+H
+\frac1\rho\sum_{k<K}\|\bm\epsilon_k\|^2>G\right).
\end{aligned}
\end{equation}
\Cref{ass:bounded_var} gives
\begin{equation}\label{eqn:sbdc_gen_expected_bound}
\begin{aligned}
\mathbb E_{s,i}\left[
\max_j g_j(\theta_j^0;\Bar{\bm\theta}_j^0)-g^*+H
+\frac1\rho\sum_{k<K}\|\bm\epsilon_k\|^2\right]
&\leq\max_j g_j(\theta_j^0;\Bar{\bm\theta}_j^0)-g^*+C',
\end{aligned}
\end{equation}
where $C':=K\sigma^2/\rho+H$. The quantity inside the expectation
is nonnegative. Markov's inequality and
$G=8(\max_j g_j(\theta_j^0;\Bar{\bm\theta}_j^0)-g^*+C')/\delta$
therefore imply
\begin{equation}\label{eqn:sbdc_gen_event1_2}
\mathbb P(t_1<K,t_2=K)
\leq\frac{\max_j g_j(\theta_j^0;\Bar{\bm\theta}_j^0)-g^*+C'}{G}
=\frac\delta8\leq\frac\delta4.
\end{equation}

Using $F'=E\rho/(9L)-(E+R)$, the remaining parameter condition is
\begin{equation}\label{eqn:prf_Thm11_1}
\frac{K\sigma^2}{(\frac{E\rho}{9L}-(E+R))^2}\leq\frac\delta4.
\end{equation}
Since $E^2=2LG$, we have
\begin{equation}\label{eqn:prf_Thm11_2}
\begin{aligned}
F'^2
&=\frac{2G\rho^2}{81L}+(E+R)^2
-\frac{2\rho}{9L}(2LG+ER)\\
&\geq\frac{2G\rho^2}{81L}
-\frac{2\rho}{9L}(2LG+ER).
\end{aligned}
\end{equation}
For $C'>0$, using $C'-H=K\sigma^2/\rho$ and $G\delta\geq8C'$, a
sufficient condition for the last expression to be at least
$4K\sigma^2/\delta$ is
\begin{align}\label{eqn:prf_rho_choice}
\rho^2-\frac{9\rho}{G}(2LG+ER)
&\geq\frac{81L\rho(C'-H)}{4C'},\notag\\
\text{equivalently,}\qquad
\rho&\geq18L+\frac{9ER}{G}
+\frac{81L}{4}\frac{C'-H}{C'}.
\end{align}
This also implies $\rho>9L(E+R)/E$ and $\rho\geq2L$.
Thus the stated choice
$\rho=(18L+9ER/G+(81L/4)(C'-H)/C')\sqrt K$ is sufficient for
$K\geq1$. If $C'=0$, then $H=\sigma^2=0$; the last term in this
choice can be omitted.
Combining the two exit estimates, we obtain
\[
\mathbb P(t<K)
\leq\mathbb P(t_2<K)+\mathbb P(t_1<K,t_2=K)
\leq\frac\delta2.
\]

We next prove the stationarity bound using a population comparator.
For every block, denote
\[
z_i^k:=\nabla g_i(\ibar{}{k}{}{k})-u_i^k
\in\overline{\partial}_i f(\bm\theta^k),
\]
and define
\[
\Psi_k:=\ip{z_{i_k}^k+\bm\epsilon_k}
{\theta_{i_k}^{k+1}-\theta_{i_k}^k}
+\frac\rho2\|\theta_{i_k}^{k+1}-\theta_{i_k}^k\|^2.
\]
Convexity of the sampled function and optimality of the update imply,
for any feasible comparison point $\theta_{i_k}$,
\begin{equation}\label{eqn:sbdc_gen_comparison}
\begin{aligned}
\Psi_k
&\leq g_{i_k}(\ibar{_k}{k+1}{_k}{k},s^k)
-g_{i_k}(\ibar{_k}{k}{_k}{k},s^k)-\ip{\hat u_{i_k}^k}{\theta_{i_k}^{k+1}-\theta_{i_k}^k}
+\frac\rho2\|\theta_{i_k}^{k+1}-\theta_{i_k}^k\|^2\\
&\leq g_{i_k}(\ibar{_k}{}{_k}{k},s^k)
-g_{i_k}(\ibar{_k}{k}{_k}{k},s^k)-\ip{\hat u_{i_k}^k}{\theta_{i_k}-\theta_{i_k}^k}
+\frac\rho2\|\theta_{i_k}-\theta_{i_k}^k\|^2.
\end{aligned}
\end{equation}
In particular, comparison with $\theta_{i_k}^k$ and completion of the
square give the bounds
\begin{equation}\label{eqn:sbdc_gen_psi_bound}
-\frac{\|z_{i_k}^k+\bm\epsilon_k\|^2}{2\rho}
\leq\Psi_k\leq0.
\end{equation}

At a point satisfying
$\max_i g_i(\ibar{}{k}{}{k})-g^*\leq G$, we have
$\|z_i^k\|\leq E+R$ for every block. Condition on $\mathcal F_k$ and
$i_k$, and choose the comparison point
\[
\theta_{i_k}=\theta_{i_k}^k-\frac{z_{i_k}^k}{L+\rho}.
\]
This point is fixed before drawing $s^k$, and its displacement satisfies
\[
\frac{\|z_{i_k}^k\|}{L+\rho}
\leq\frac{E+R}{L+\rho}<\frac{E}{9L}<r(E)=\frac EL.
\]
Thus it lies in the population smoothness ball. Taking expectation in
\eqref{eqn:sbdc_gen_comparison}, using the population representation of
$g_i$, unbiasedness of $\hat u_{i_k}^k$, and population smoothness, gives
\begin{equation}\label{eqn:derivation_1}
\begin{aligned}
\mathbb E[\Psi_k\mid\mathcal F_k,i_k]
&\leq
 g_{i_k}\left(\theta_{i_k}^k-\frac{z_{i_k}^k}{L+\rho};
 \Bar{\bm\theta}_{i_k}^k\right)
-g_{i_k}(\ibar{_k}{k}{_k}{k})+\frac{\ip{u_{i_k}^k}{z_{i_k}^k}}{L+\rho}
+\frac{\rho\|z_{i_k}^k\|^2}{2(L+\rho)^2}\\
&\leq-\frac{\|z_{i_k}^k\|^2}{L+\rho}
+\frac{\|z_{i_k}^k\|^2}{2(L+\rho)}
=-\frac{\|z_{i_k}^k\|^2}{2(L+\rho)}.
\end{aligned}
\end{equation}
Note that the expectation here uses the full current sample distribution and is
not conditioned on $k<t$.
The event $\{k\leq t\}$ is $\mathcal F_k$-measurable, since it only
requires that no previous iteration has triggered a stop. On this
event, the starting point satisfies the gradient bound for every block.
Averaging \eqref{eqn:derivation_1} over the current uniform block gives
\begin{equation}\label{eqn:derivation_2}
\begin{aligned}
\mathbb E\left[\mathbf1_{\{k\leq t\}}\Psi_k\mid\mathcal F_k\right]
&\leq-\frac{\mathbf1_{\{k\leq t\}}}{2n(L+\rho)}
\sum_{i=1}^n\|z_i^k\|^2\leq-\frac{\mathbf1_{\{k\leq t\}}}{2n(L+\rho)}
\mathcal{\bm G}^2(\bm\theta^k).
\end{aligned}
\end{equation}
Define, for $0\leq k<K$,
\[
\xi_k:=\mathbf1_{\{k\leq t\}}\Psi_k
-\mathbb E\left[\mathbf1_{\{k\leq t\}}\Psi_k\mid\mathcal F_k\right],
\qquad \xi_K:=0.
\]
By \eqref{eqn:sbdc_gen_psi_bound} and \Cref{ass:bounded_var} we know that $\mathbb E[\xi_k\mid\mathcal F_k]=0$, and $\xi_k=0$ for
$k>t$ which makes these variables integrable. For $k<t$, smoothness and convexity of $h_{i_k}$ give
\begin{align*}
f(\bm\theta^{k+1})-f(\bm\theta^k)
&\leq\ip{z_{i_k}^k}{\theta_{i_k}^{k+1}-\theta_{i_k}^k}
+\frac L2\|\theta_{i_k}^{k+1}-\theta_{i_k}^k\|^2\\
&=\Psi_k-\ip{\bm\epsilon_k}{\theta_{i_k}^{k+1}-\theta_{i_k}^k}
-\frac{\rho-L}{2}\|\theta_{i_k}^{k+1}-\theta_{i_k}^k\|^2\\
&\leq\Psi_k+\frac{\|\bm\epsilon_k\|^2}{2(\rho-L)}
\leq\Psi_k+\frac{\|\bm\epsilon_k\|^2}{\rho},
\end{align*}
where $\rho\geq2L$. Consequently, \eqref{eqn:derivation_2} yields
\begin{equation}\label{eqn:sbdc_gen_stopped_descent}
f(\bm\theta^k)-f(\bm\theta^{k+1})
\geq\frac{\mathcal{\bm G}^2(\bm\theta^k)}{2n(L+\rho)}
-\xi_k-\frac{\|\bm\epsilon_k\|^2}{\rho},
\qquad k<t.
\end{equation}

Since $\xi_k$ is centered and vanishes after $t$,
\[
\mathbb E_{s,i}\left[\sum_{k<t}\xi_k\right]
=-\mathbb E_{s,i}\left[\mathbf1_{\{t<K\}}\xi_t\right].
\]
The terminal term must be retained because $\{k<t\}$ depends on the
current draws. For $k<K$, on $\{k\leq t\}$, \eqref{eqn:sbdc_gen_psi_bound},
conditional unbiasedness, and $\|z_i^k\|\leq E+R$ imply
\[
-\frac{(E+R)^2+\sigma^2}{2\rho}
\leq\mathbb E[\Psi_k\mid\mathcal F_k]\leq0,
\qquad
|\xi_k|\leq
\frac{\|z_{i_k}^k+\bm\epsilon_k\|^2+(E+R)^2+\sigma^2}{2\rho}.
\]
Moreover,
\[
\mathbb E_{s,i}\left[
\sum_{k<K}\mathbf1_{\{t=k\}}\|\bm\epsilon_k\|^2\right]
\leq\sum_{k<K}\mathbb E_{s,i}[\|\bm\epsilon_k\|^2]
\leq K\sigma^2.
\]
It follows that
\begin{equation}\label{eqn:sbdc_gen_terminal_bound}
\mathbb E_{s,i}\left[\mathbf1_{\{t<K\}}|\xi_t|\right]
\leq\frac{3(E+R)^2+(2K+1)\sigma^2}{2\rho}.
\end{equation}

From \eqref{eqn:sbdc_gen_population_gap},
$f(\bm\theta^0)-f(\bm\theta^t)
\leq\max_j g_j(\theta_j^0;\Bar{\bm\theta}_j^0)-g^*+H$.
Summing \eqref{eqn:sbdc_gen_stopped_descent} and using
\eqref{eqn:sbdc_gen_terminal_bound}, we obtain
\begin{equation}\label{eqn:expected_both_sum_simplified_gen}
\begin{aligned}
\frac1{2n(L+\rho)}
\mathbb E_{s,i}\left[\sum_{k<t}\mathcal{\bm G}^2(\bm\theta^k)\right]
&\leq\mathbb E_{s,i}\left[f(\bm\theta^0)-f(\bm\theta^t)\right]
+\frac{K\sigma^2}{\rho}
+\mathbb E_{s,i}\left[\mathbf1_{\{t<K\}}|\xi_t|\right]\\
&\leq\max_j g_j(\theta_j^0;\Bar{\bm\theta}_j^0)-g^*+C'
+\frac{3(E+R)^2+(2K+1)\sigma^2}{2\rho}\\
&=\frac{G\delta}{8}
+\frac{3(E+R)^2+(2K+1)\sigma^2}{2\rho}.
\end{aligned}
\end{equation}
Since the squared residuals are nonnegative, this implies
\[
\mathbb E_{s,i}\left[
\mathbf1_{\{t=K\}}\frac1K\sum_{k<K}\mathcal{\bm G}^2(\bm\theta^k)
\right]
\leq\frac{2n(L+\rho)}K
\left[\frac{G\delta}{8}
+\frac{3(E+R)^2+(2K+1)\sigma^2}{2\rho}\right].
\]
Choose $K$ so that
\begin{equation}\label{eqn:sbdc_gen_horizon}
K\geq\frac{4n(L+\rho)}{\delta\epsilon^2}
\left[\frac{G\delta}{8}
+\frac{3(E+R)^2+(2K+1)\sigma^2}{2\rho}\right].
\end{equation}
The preceding expectation is then at most $\delta\epsilon^2/2$.
Define the event
\[
\varrho:=\left\{\frac1K\sum_{k<K}
\mathcal{\bm G}^2(\bm\theta^k)>\epsilon^2\right\}.
\]
Markov's inequality gives
\[
\mathbb P(\{t=K\}\cap\varrho)
\leq\frac1{\epsilon^2}
\mathbb E_{s,i}\left[
\mathbf1_{\{t=K\}}\frac1K\sum_{k<K}\mathcal{\bm G}^2(\bm\theta^k)
\right]\leq\frac\delta2.
\]
Therefore,
\[
\mathbb P(\varrho)
\leq\mathbb P(t<K)+\mathbb P(\{t=K\}\cap\varrho)
\leq\delta.
\]
In particular, with probability at least $1-\delta$ over samples and
block choices,
\[
\min_{0 \leq k<K}\mathcal{\bm G}^2(\bm\theta^k)
\leq\frac1K\sum_{k<K}\mathcal{\bm G}^2(\bm\theta^k)
\leq\epsilon^2.
\]
With $\rho=\Theta(\sqrt K)$ and $\sigma^2=\mathcal O(K^{-1/2})$,
both $C'$ and the terminal correction are bounded independently of $K$.
For fixed confidence and problem constants, the sufficient horizon
\eqref{eqn:sbdc_gen_horizon} therefore implies the iteration order
$\mathcal O(n^2/\epsilon^4)$.

\subsection{Proof of \Cref{lem:bounding_differences_stoc}}\label[appendix]{app:prf_bounding_differences_stock}
    Due to \eqref{eqn:stochastic_multi-block_update}, we have $ u_{i_k}^k\in \partial h_{i_k}(\ibar{_k}{k}{_k}{k})$ and
\begin{align}\label{eqn:prf_lem8_1}
    \langle \hat{ u}_{i_k}^k,  \theta_{i_k}^{k} - \theta_{i_k}^{k+1} \rangle \leq g_{i_k}(\ibar{_k}{k}{_k}{k},s^k) - g_{i_k}(\ibar{_k}{k+1}{_k}{k},s^k)-\frac{\rho}{2}\|{\theta}_{i_k}^{k+1} - {\theta}_{i_k}^{k}\|^2.
\end{align}
By adding and subtracting $\langle {\hat{ u}}_{i_k}^k - u_{i_k}^k,  \theta_{i_k}^{k+1} - \theta_{i_k}^{k} \rangle$ and using convexity of $g_{i_k}$, we have
\begin{equation*}
    \begin{aligned}
       \langle u_{i_k}^k,  \theta_{i_k}^{k} - \theta_{i_k}^{k+1} \rangle \leq \langle \hat{ u}_{i_k}^k - u_{i_k}^k, \theta_{i_k}^{k+1} - \theta_{i_k}^{k} &\rangle-\ip{\nabla  \hat g_{i_k}(\ibar{_k}{k}{_k}{k})}{\theta_{i_k}^{k+1}-\theta_{i_k}^k}-\frac{\rho}{2}\|{\theta}_{i_k}^{k+1} - {\theta}_{i_k}^{k}\|^2,\\
      \implies \langle u_{i_k}^k - \nabla g_{i_k}(\ibar{_k}{k}{_k}{k}),  \theta_{i_k}^{k} - \theta_{i_k}^{k+1} \rangle \leq & \langle \nabla  \hat g_{i_k}(\ibar{_k}{k}{_k}{k})-\hat{ u}_{i_k}^k -\left( \nabla g_{i_k}(\ibar{_k}{k}{_k}{k}) -u_{i_k}^k\right),  \theta_{i_k}^{k} - \theta_{i_k}^{k+1} \rangle \\
      & -\frac{\rho}{2}\|{\theta}_{i_k}^{k+1} - {\theta}_{i_k}^{k}\|^2,
      \end{aligned}
\end{equation*}
Now, through Cauchy–Schwarz inequality we get
\begin{equation*}
    \begin{aligned}
      \frac{\rho}{2}\|{\theta}_{i_k}^{k+1} - {\theta}_{i_k}^{k}\|^2 \leq& \ip{\nabla g_{i_k}(\ibar{_k}{k}{_k}{k}) - u_{i_k}^k}{\theta_{i_k}^{k}-\theta_{i_k}^{k+1}} \\
      &+ \|\nabla \hat g_{i_k}(\ibar{_k}{k}{_k}{k})-\hat{ u}_{i_k}^k -\left( \nabla g_{i_k}(\ibar{_k}{k}{_k}{k}) -u_{i_k}^k\right)\| \|{\theta}_{i_k}^{k+1} - {\theta}_{i_k}^{k}\|,\\
      \implies  \frac{\rho}{2}\|{\theta}_{i_k}^{k+1} - {\theta}_{i_k}^{k}\|^2 \leq& \|\nabla g_{i_k}(\ibar{_k}{k}{_k}{k}) - u_{i_k}^k\|\|{\theta}_{i_k}^{k+1} - {\theta}_{i_k}^{k}\| + \| \hat{   z}_{i_k}^k - z_{i_k}^k\| \|{\theta}_{i_k}^{k+1} - {\theta}_{i_k}^{k}\|,\\
      \implies  \|{\bm\theta}^{k+1} - {\bm\theta}^{k}\| \leq& \frac{2}{\rho}\left(\| z_{i_k}^k\|+ \| \hat{   z}_{i_k}^k -  z_{i_k}^k\|\right),\\
\end{aligned}    
\end{equation*}
where $ z_{i_k}^k\in \overline{\partial_{i_k}} f(\bm\theta^k), \hat{   z}_{i_k}^k\in \overline{\partial_{i_k}} \hat{f}(\bm\theta^k)$, and the last line holds due to the fact that \Cref{eqn:stochastic_multi-block_update} updates only the $i_k^{\text{th}}$ block at each iteration. This implies the desired result. 

\subsection{Proof of \Cref{lem:bound_on_g_stochastic}}\label[appendix]{app:prf_lem_bound_on_g_stochastic}

By assumption, we know that $g_{i_k}(\ibar{_k}{k}{_k}{k})-g^* \leq G$ and $\|\nabla\hat g_{i_k}(\ibar{_k}{k}{_k}{k}) - \hat{ u}^k_{i_k} - (\nabla_{i_k} g(\bm\theta^k) -  u^k_{i_k} )\|\leq F'$ for $G,F'>0$. Since this result holds for any $k$, through \Cref{corr:lsmoth-bounded-grad} we have $\|\nabla g_{i_k}(\ibar{_k}{k}{_k}{k})\|\leq E$.
Recall that $g_{i_k}$ is $\ell$-smooth with a subquadratic $\ell$.
Using \Cref{prop:sbdc_rlsmooth}, $g$ is also $(r,m)$-smooth with $r(u)=\frac{a}{m(u)}$ and $m(u):=\ell(u+a)$ for some $a>0$.
Therefore, we can use \Cref{lemma:rlsmoth-descent} if we ensure the updates are inside $\mathcal{B}(\bm\theta^k,r(E))$. From the proof of \Cref{lem:bounding_differences_stoc} (\Cref{app:prf_bounding_differences_stock}), we know 
\begin{align*}
       \|{\bm \theta}^{k+1} - {\bm \theta}^k\| 
      &\leq \sup_{ u\in\partial h_{i_k}(\ibar{_k}{k}{_k}{k})} \frac{2(\|\nabla g_{i_k}(\ibar{_k}{k}{_k}{k})\|+\| u\| + F')}{\rho}\leq \frac{2(E+R + F')}{\rho}.
\end{align*}
As a result taking $\rho\geq \frac{2(E + R + F')}{r(E)}=\ell(2E) \frac{2(E + R + F')}{E}$ will satisfy the conditions in \Cref{lemma:rlsmoth-descent}.
This implies the desired result.

\subsection{Proof of \Cref{thm:bdc_smooth}}\label[appendix]{app:prf_bdc_smooth}

For any $\bm \theta_{i_k}\in\mathcal{M}^{i_k}$, Consider the surrogate function 
    \begin{equation*}
        \hat f({ \theta}_{i_k}; \bar{\bm\theta}_{i_k}^k) := g_{i_k}({ \theta}_{i_k}; \bar{\bm\theta}_{i_k}^k) -  h_{i_k}({ \theta}^k_{i_k}; \bar{\bm\theta}_{i_k}^k) - \langle u_{i_k}^k,{\theta}_{i_k}-{ \theta}_{i_k}^k \rangle.
    \end{equation*}
    where \(u_{i_k}^k \in\partial_{i_k} h_{i_k}({ \theta}_{i_k}^k; \bar{\bm\theta}_{i_k}^k) \). 
    Also, it is important to mention that 
    $f({ \theta}_{i_k}^k; \bar{\bm\theta}_{i_k}^k) = f(\bm \theta ^k)$ and $\hat f({ \theta}_{i_k}^k; \bar{\bm\theta}_{i_k}^k) = \hat f(\bm \theta ^k).$
    From \eqref{eqn:deterministic_multi-block_update_L_smooth} we know that 
    \[ \hat f({\bm \theta}^{k+1}) \leq  \hat f({ \theta}_{i_k}; \bar{\bm\theta}_{i_k}^k) = g_{i_k}({ \theta}_{i_k}; \bar{\bm\theta}_{i_k}^k) -   h_{i_k}({ \theta}^k_{i_k}; \bar{\bm\theta}_{i_k}^k) - \langle u_{i_k}^k,{\theta}_{i_k}-{ \theta}_{i_k}^k \rangle,\]
    and further considering the smoothness of \(g\), we get
    \begin{equation*}
     \hat f({\bm \theta}^{k+1}) \leq   g_{i_k}({ \theta}_{i_k}^k; \bar{\bm\theta}_{i_k}^k) + \langle \nabla_{i_k} g_{i_k}({ \theta}_{i_k}^k; \bar{\bm\theta}_{i_k}^k),{\theta}_{i_k}-{ \theta}_{i_k}^k \rangle + \frac{L}{2}\| \theta_{i_k} - \theta_{i_k}^k\|^2 -h_{i_k}({ \theta}^k_{i_k}; \bar{\bm\theta}_{i_k}^k) - \langle u_{i_k}^k,{\theta}_{i_k}-{ \theta}_{i_k}^k \rangle.
    \end{equation*}
    By \eqref{eqn:general_func}, we get
    \begin{equation*}
        \begin{aligned}
                  \hat f({\bm \theta}^{k+1}) \leq   f({\bm \theta}^{k}) + \langle \nabla_{i_k} g_{i_k}({ \theta}_{i_k}^k; \bar{\bm\theta}_{i_k}^k)- u_{i_k}^k,{\theta}_{i_k}-{ \theta}_{i_k}^k \rangle + \frac{L}{2}\| \theta_{i_k} - \theta_{i_k}^k\|^2 
        \end{aligned}
    \end{equation*}
    Therefore,
    \begin{equation*}
        \begin{aligned}
                 \langle \nabla_{i_k} g_{i_k}({ \theta}_{i_k}^k; \bar{\bm\theta}_{i_k}^k)- u_{i_k}^k,{\theta}_{i_k}^k- { \theta}_{i_k} \rangle   \leq& f({\bm \theta}^{k})  - \hat f({\bm \theta}^{k+1}) + \frac{L}{2}\|{ \theta}_{i_k} - { \theta}_{i_k}^k\|^2 \\
                   \leq&  g_{i_k}({ \theta}_{i_k}^k; \bar{\bm\theta}_{i_k}^k) - g_{i_k}({\theta}_{i_k}^{k+1}; \bar{\bm\theta}_{i_k}^k) + \frac{L}{2}\|{ \theta}_{i_k} - { \theta}_{i_k}^k\|^2\\
                   & + \langle u_{i_k}^k,{ \theta}_{i_k}^{k+1}-{ \theta}_{i_k}^k \rangle.
                   \end{aligned}
                   \end{equation*}
Therefore, by convexity of $h_{i_k}(\cdot\;; \bar{\bm\theta}_{i_k}^k)$ we have:
            \begin{equation*}
                \begin{aligned}
                   \langle \nabla_{i_k} g_{i_k}({ \theta}_{i_k}^k; \bar{\bm\theta}_{i_k}^k)- u_{i_k}^k,{\theta}_{i_k}^k- { \theta}_{i_k} \rangle - \frac{L}{2}\|{ \theta}_{i_k} - { \theta}_{i_k}^k\|^2    \leq& g_{i_k}({ \theta}_{i_k}^k; \bar{\bm\theta}_{i_k}^k) - g_{i_k}({ \theta}_{i_k}^{k+1}; \bar{\bm\theta}_{i_k}^k)  \\
                   & +  h_{i_k}({ \theta}_{i_k}^{k+1}; \bar{\bm\theta}_{i_k}^k)-h_{i_k}({ \theta}_{i_k}^{k}; \bar{\bm\theta}_{i_k}^k)
                   ,\\
                    \langle \nabla_{i_k} g_{i_k}({ \theta}_{i_k}^k; \bar{\bm\theta}_{i_k}^k)-u_{i_k}^k,{ \theta}_{i_k}^k-{ \theta}_{i_k} \rangle - \frac{L}{2}\|{ \theta}_{i_k} - { \theta}_{i_k}^k\|^2    \leq &f({\bm \theta}^{k}) - f({\bm \theta}^{k+1}).
        \end{aligned}
    \end{equation*}
Let us denote by $\mathbb{E}_{|k}$ the conditional expectation with respect to the random selection of $i_k$, given all the random choices in the previous iterations. Then, maximizing over $\theta_{i_k}$, we have  
    \begin{equation}
    \begin{aligned}\label{eqn:prf_lem_cnv_2}
         &\expectk{\max_{\theta_{i_k}\in\mathcal{M}^{{i_k}}} \inf_{z_{i_k}^k\in \overline{\partial}_{i_k} f(\bm{\theta}^k)} \langle z_{i_k}^k,{ \theta}_{i_k}^k-{ \theta}_{i_k} \rangle - \frac{L}{2}\|{ \theta}_{i_k} - { \theta}_{i_k}^k\|^2  } = \frac{1}{n} \sum_{i=1}^n \gap_{\mathcal{M}^{d_{i}}  }^L (\bm\theta^k)  
    \end{aligned}
    \end{equation}  
    for $z_{i}^k:=\nabla_i g_{i}(\ibar{}{k}{}{k}) -  u_{i}^k\in \overline{\partial}_{i_k} f(\bm{\theta}^k).$
    Using \eqref{eqn:prf_lem_cnv_2} we have
    \begin{equation*}
        \sum_{i=1}^n\gap^L_{\mathcal{M}^i}(\bm{\theta}^k)
        \leq n f({\bm \theta}^k) - n\expectk{f({\bm \theta}^{k+1})}.
    \end{equation*}
    Now, taking expectation w.r.t. all the iterations we have
    \begin{equation*}
        \expect{\sum_{i=1}^n\gap^L_{\mathcal{M}^i}(\bm{\theta}^k)}
        \leq n \expect{f(\bm \theta^k)} - n \expect{f(\bm \theta^{k+1})}.
    \end{equation*}
    Finally, we take the average of this inequality over $k=1,\ldots,K$:
    \begin{equation*}
        \frac{1}{K} \sum_{k=1}^K \expect{\sum_{i=1}^n\gap^L_{\mathcal{M}^i}(\bm{\theta}^k)}
        \leq \frac{n}{K} \left(f(\bm{\theta}^1) - \expect{f(\bm{\theta}^{K+1})})\right) 
        \leq \frac{n}{K} \left(f(\bm{\theta}^1) - f^\star\right).
    \end{equation*}
    We complete the proof by noting that the minimum of $\sum_{i=1}^n\gap^L_{\mathcal{M}^i}(\bm{\theta}^k)$ over $k=1,\ldots,K$ is smaller than or equal to the average gap summation.

\end{document}